\let\ifdebugmode\iffalse 
\documentclass[a4paper,11pt]{article}
\usepackage[hidelinks]{hyperref}
\usepackage[top=2.8cm, bottom=2.8cm, hmargin=2.5cm, headsep=.75cm, twoside]{geometry}
\usepackage{macros}
\usepackage{amssymb}

\title{An extension of Batanin's approach to globular algebras\thanks{This work
    was partially supported by the French ANR project PPS (ANR-19-CE48-0014).}}
\author{Simon Forest\\
  \small Aix-Marseille Univ, CNRS, I2M, Marseille, France}
\date{}

\begin{document}

\maketitle
\begin{abstract}
  In earlier work, Batanin has shown that an important class of definitions of
higher categories could be apprehended together simply as monads over globular
sets. This allowed him to generalize the notion of polygraph, initially
introduced by Street and Burroni for strict categories, to all algebraic
globular higher categories. In this work, we refine this perspective and
introduce new constructions and properties for this class of higher categories.
In particular, we define the notion of cellular extension and its associated
free construction, from which we obtain another definition of polygraphs and the
adjunction between globular algebras and polygraphs. We moreover introduce two
criteria allowing one to use most of the constructions of this article without
having to describe explicitly the underlying globular monad.
\end{abstract}

\section*{Introduction}

Over the last years, higher categories have emerged as a convenient tool to
study problems in mathematics, physics and computer
science~\cite{baez2010physics,leinster2004higher}. The notion of ``higher
category'' encompasses informally all the structures that have
higher-dimensional cells which can be composed together with several operations.
It admits a vast number of definitions, which can make it hard to apprehend.

In order to get a more global view and factor out several common constructions
and properties across the different possible definitions higher categories, it
is useful to consider a restriction of this general notion to a more formal
class of theories. This was done by Batanin~\cite{batanin1998computads}, who
introduced a unified formalism for algebraic globular higher categories. The
latter are very common, since they include all the globular higher categories
defined by a set of operations and equations between them. Moreover, the
instances of such higher categories form \emph{locally finitely presentable
  categories} and, as such, have very good properties, like being complete and
cocomplete~\cite{adamek1994locally}. The setting of Batanin then enables one to
derive several common constructions for such higher categories. In particular,
one can generalize to those the notion of \emph{polygraph}, which can be thought
as higher-dimensional signatures, originally defined by
Street~\cite{street1976limits} for strict $2$\categories and Burroni for strict
$\omega$\categories. Batanin moreover showed that there is an adjunction between
$n$\polygraphs and $n$\categories, which is \eq{most of the time} monadic (as
proved by Street for $2$\polygraphs of strict categories), allowing for
presenting higher categories of this framework using polygraphs.

Although this is already valuable, it seems that the work of Batanin can be
refined in several aspects. First, several constructions of interest for
concrete applications are too implicit, if not absent of Batanin's treatment.
For example, in the context of higher dimensional rewriting, one is usually
interested in a notion intermediate to higher category and polygraph, called
\emph{cellular extension}, which is simply a higher category with a set of
generators in the next dimension. One can then consider the free
$(n{+}1)$\category obtained from an $n$\cellular extension by freely generating
the $(n{+}1)$\cells from the set of generators. Thus, cellular extensions can be
considered as a more flexible construction than the one on polygraphs, where
each dimension has to be freely generated.

Second, in order to use the results of Batanin, one has to work directly with a
monad and its Eilenberg-Moore category. However, one usually does not have
direct access to the underlying monad of the considered theory of higher
category, but only the operations of the theory together with the equations
satisfied by the theory, from which one can derive a monad, but the description
of the latter can be very tedious~\cite{penon1999approche}. Instead, one would
prefer results that can work directly with the equational theory, the category
of models of this theory, and the straight-forward functors that one can build
from it. In particular, such a shortcut feels particularly desirable for showing
the \emph{truncability} of the monad, which is required for the constructions
and properties of Batanin to work.

Third, it seems that deeper understanding about globular algebras can be gained
by seeing them through the lenses of the formal theory of monads of
Street~\cite{street1972formal}. Indeed, several natural constructions on
globular algebras, like truncations functors and their right adjoints, are in
fact projections, through the Eilenberg-Moore functor, of constructions
happening in the category $\MND$ of monads and monad functors, whose
compositions involve vertical composition of square-shaped natural
transformation. Then, several computations on these functors can be nicely
described using string diagrams, facilitating their readability and making them
more intuitive.

In this work, we attempt to address these three points.

\paragraph*{Outline}

In \Cref{text:formal-theory-monad}, we recall some elements of Street's formal
theory of monads~\cite{street1972formal}. In particular, we define the category
$\MND$ of monads over categories together with the Eilenberg-Moore construction,
which produces a category of algebras from a monad. In
\Cref{text:higher-categories-globular-algebras}, we recall Batanin's framework
for globular higher categories, where theories of higher categories are simply
studied as monads on globular sets, and extend it with new results and
constructions. Among other, we prove two criterions
(\Cref{thm:charact-globular-algebras} and \Cref{thm:truncable-charact}) which
help bridge the gap between that framework and the usual definitions of higher
categories as structures with operations satisfying equations. In
\Cref{text:free-higher}, we recover the notion of polygraph through a different
path than the one of Batanin using the intermediate and intuitive notion of
cellular extensions. We also describe explicitly the free constructions
associated to both notions. The adjunctions shown by Batanin between higher
categories and polygraphs can then be recovered from an adjunction between
cellular extensions and polygraphs (\Cref{thm:poltoce-algptopol-adj}). Finally,
in \Cref{text:strict-cats-and-precats}, we illustrate the use of our properties
and constructions on the case of strict categories.

\paragraph*{Acknowledgements}

I would like to thank all the people with whom I was able to discuss this work
and who gave me useful feedback about it. In particular, I would like to thank
Samuel Mimram, Tom Hirschowitz and Dominic Verity who took the time to read in
depth an earlier version of this work under the form of the first chapter of my
PhD thesis. The comments of Dominic were particularly valuable to this work,
since he opened my eyes on some very naïve arguments I used in that
earlier version, and indicated me some more natural ones. Most of the changes
from the earlier version to this one can be attributed to him, so that I am
grateful to him.

Also, some additional thanks to Samuel Mimram for having developed
\texttt{satex}\footnote{\url{https://github.com/smimram/satex}.} that I have
used to draw most of the string diagrams of this article.


\paragraph*{Notations}

In this article, given $n \in \N$, $\N_n$ denotes the set~$\set{0,\ldots,n}$
and~$\N^*_n$ denotes the set~$\set{1,\ldots,n}$. We extend this notation to the
infinity and put~$\N_\omega = \N$ and $\N^*_\omega = \N^*$. As a consequence of
the above choices, we might write~$\N_n \cup \set n$ as a convenient
abbreviation to denote either~$\N_n$ when~$n \in \N$, or~$\Ninf$ when~$n =
\omega$.


\section{Highlights of the formal theory of monads}
\label{text:formal-theory-monad}

\label{text:algebras-over-a-monad}

In order to better manipulate globular algebras and constructions upon them, we
use the formal theory of monads of Street~\cite{street1972formal}, of which we
recall briefly the salient points.

We start by recalling the notion of algebra for a monad and the associated
Eilenberg-Moore category in \Cref{text:algebras}. Then, in
\Cref{text:mnd-category}, we recall that the construction of the Eilenberg-Moore
category is part of an adjunction between the $2$-categories $\MND$, describing
monads over categories, and $\CAT$. Finally, in
\Cref{text:morphs-in-mnd-from-adjs}, we give several useful properties to build
morphisms in $\MND$ from adjunctions in $\CAT$.

In the following, we introduce monads either as $(T,\eta,\mu)$, where $T$ is the
endofunctor of the monad and $(\eta,\mu)$ is the unit-counit pair; or we
introduce them as $(\cC,T)$, where $\cC$ is the base category for which $T$ is
the endofunctor. When the latter notation is used, the unit and counit, when
required, are introduced explicitly.

\subsection{Algebras}
\label{text:algebras}

Given a monad~$(T,\eta,\mu)$ on a category~$\mcal C$, a \index{algebra over a
  monad}\emph{$T$\algebra} is the data of an object~$X \in \mcal C$ together
with a morphism~${h \co TX \to X}$ such that
\[
  h \circ \eta_X = \unit X
  \qtand
  h \circ \mu_X = h \circ T(h)\zbox.
\]
A \index{morphism!of algebras}\emph{morphism} between two algebras~$(X,h)$
and~$(X',h')$ is given by a morphism~$f \co X \to X'$ of~$\mcal C$ satisfying
\[
  f \circ h = h' \circ T(f)\zbox.
\]
We \glossary(.aaa){$\mcal C^T$}{the Eilenberg-Moore category \wrt the
  monad~$T$}write~$\mcal C^T$ for the category of $T$\algebras, also called
\index{Eilenberg-Moore category}\emph{Eilenberg-Moore category of~$T$}. There is
a canonical forgetful \glossary(UT){$\bigfun U^T$}{the canonical forgetful
  functor of an Eilenberg-Moore category}functor
\[
  \bigfun U^T \co \mcal C^T \to \mcal C
\]
which maps the $T$\algebra~$(X,h)$ to~$X$. This functor has a canonical left
\glossary(FT){$\bigfun F^T$}{the canonical free functor of an Eilenberg-Moore
  category}adjoint
\[
  \bigfun F^T \co \mcal C \to \mcal C^T
\]
which maps~$X \in \mcal C$ to the $T$\algebra~$(TX,\mu_X)$, such that the unit
of~$\bigfun F^T \dashv \bigfun U^T$ is~$\eta^T = \eta$, and the associated
counit, denoted~$\eps^T$, is such that~$\eps^T_{(X,h)} = h$ for a given
$T$\algebra~$(X,h)$. The monad induced by~$\smash{\bigfun F^T \dashv \bigfun
  U^T}$ is then exactly~$(T,\eta,\mu)$.

\subsection{The category $\MND$}
\label{text:mnd-category}

Given two monads~$(S,\gamma,\nu)$ and~$(T,\eta,\mu)$ on two categories~$\mcal C$
and~$\mcal D$ respectively, a \emph{monad functor} betweem~$(S,\gamma,\nu)$
and~$(T,\eta,\mu)$ is the data of a functor $F \co \mcal C \to \mcal D$ together
with a natural transformation $\alpha \co TF \To FS$ such that
\[
  \alpha \circ (\eta F) = F\gamma
  \qqtand
  \alpha \circ (\mu F) = (F \nu) \circ (\alpha S) \circ (S \alpha)
  \zbox.
\]

A \emph{monad functor transformation} (often abbreviated \emph{truncable
  monads}) between two monad functors $(F,\alpha),(G,\beta) \co (\cC,S) \to
(\cD,T)$ is a natural transformation $m \co F \To G$ such that
\[
  (m S) \circ \alpha = \beta \circ (T m)
  \zbox.
\]

Monads, monad functors and monad transformations form a (very large) strict
$2$\category denoted~$\MND$. There is a functor
\[
  \mndinc \co \CAT \to \MND
\]
which maps a large category $\mcal C$ to the identity monad $(\mcal
C,\catunit[\mcal C])$.
This functor admits a right adjoint
\[
  \EM \co \MND \to \CAT
\]
which maps a monad $(\mcal C,T)$ to the Eilenberg-Moore category $\mcal C^T$,
and a monad functor
\[
  (F,\alpha) \co (\mcal C,S) \to (\mcal D,T)
\]
to a
functor
\[
  \emfunct F \alpha \co \emcat{\mcal C} S \to \emcat{\mcal D} T
\]
mapping
an algebra $(X,h) \in \emcat {\mcal C} S$ to the algebra $(FX,Fh \circ \alpha_X)
\in \emcat {\mcal D} T$, and mapping an algebra morphism~$f$ to $F(f)$. Finally,
a monad transformation
\[
  m \co (F,\alpha) \To (G,\beta) \co (\cC,S) \to (\cD,T)
\]
is mapped by $\EM$ to the natural transformation $F^\alpha \To G^\beta$ whose
component at an $S$\algebra $(X,h)$ is $m_X$. The component of the counit of the
adjunction at a monad $(\mcal C,T)$ is given by $(\emcanf T,\emcanf T \eps^T)$.
Moreover, the monad functor $(T,\mu) \co (\cC,\catunit) \to (\cC,T)$ corresponds
to the functor $\freealgf^T\co \cC \to \cC^T$ by the adjunction $\mndinc \dashv
\EM$.

There is also a forgetful functor
\[
  \mndund \co \MND \to \CAT
\]
which maps a monad $(\cC,T)$ to $\cC$ and whose action on monad functors and
monad transformations is the expected one. We mention that this functor is left
adjoint to $\mndinc$, even though this is not really useful for our purposes.


\subsection{Morphisms of $\MND$ from adjunctions}
\label{text:morphs-in-mnd-from-adjs}

We give several results that allows building morphisms of $\MND$ from
adjunctions. First, we prove that every functor which is a right adjoint can be
lifted canonically to $\MND$:
\begin{prop}
  \label{prop:ex-algra}
  Let $\cC$ and $\cD$ be two categories, $(S,\eta^S,\mu^S)$ be a monad on~$\cC$,
  and
  \[
    L \dashv R \co \cC \to \cD
  \]
  be an adjunction. The functor $R$ lifts
  through $\mndund$ to a cocartesian map
  \[
    (R,RS\eps)\co (\cC,S) \to (\cD,T)
  \]
  where
  $T = (T,\eta^T,\mu^T)$ is the canonical monad with $T = RSL$ and $\eps$ is the
  counit of $L \dashv R$.
\end{prop}

\begin{proof}
  The fact that we obtain a monad functor can readily be verified using string
  diagrams. For example, the equation $(RS\eps) \circ (\eta^T R) = R \eta^S$
  asserts that the two diagrams
  \[
    \satex{und-cocart-lift-eq-unit-1}
    \qtand
    \satex{und-cocart-lift-eq-unit-2}
  \]
  represents the same natural transformation, which is true by the zigzag
  equations for $L \dashv R$. The other equation $(RS\eps) \circ (\mu^T R) = (R
  \mu^S) \circ (RS\eps S) \circ (T RS \eps)$ is verified similarly.

  We are now left to prove the cocartesianness of $(R,RS\eps)$. So let
  $(\cD',T')$ be a monad, $U \co \cC \to \cD'$ and $\bar U \co \cD \to \cD'$
  such that $U = \bar U R$, and $(U,\alpha) \co (\cC,S) \to (\cD',T')$ be a
  monad functor. We define $\bar \alpha \co T' \bar U \To \bar U T$ by
  \[
    \bar \alpha
    \qquad
    =
    \qquad
    \satex{und-cocart-lift-bar-alpha-def}
  \]
  and verify that it defines a monad functor $(\bar U,\bar \alpha) \co (\cD,T)
  \to (\cD',T')$. We start with the unit equation, \ie 
  $\bar\alpha \circ (\eta^{T'} \bar U) = \bar U \eta^T$:
  \[
    \satex{und-cocart-lift-bar-alpha-eq-unit-1}
    \qquad
    =
    \qquad
    \satex{und-cocart-lift-bar-alpha-eq-unit-2}
    \qquad
    =
    \qquad
    \satex{und-cocart-lift-bar-alpha-eq-unit-3}
    \qquad
    =
    \qquad
    \satex{und-cocart-lift-bar-alpha-eq-unit-4}
    \zbox.
  \]
  The second equation, \ie $\bar \alpha \circ (\mu^T \bar U) = (\bar U \mu^S)
  \circ (\bar \alpha T) \circ (T' \bar \alpha)$, is proved in
  \Cref{fig:bar-alpha-mu}.
  \begin{figure}
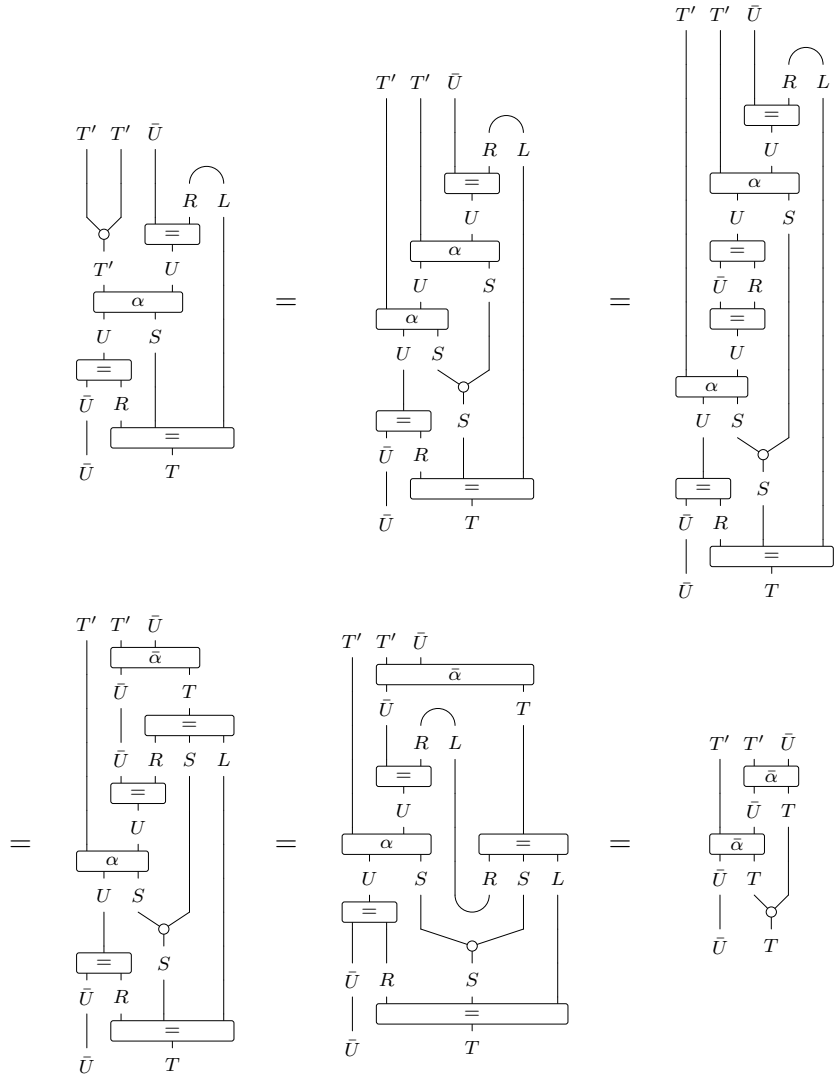

    \centering
    \[
      \begin{array}{cccccc}
        &
          \satex{und-cocart-lift-bar-alpha-eq-mult-1}
        &
          =
        &
          \satex{und-cocart-lift-bar-alpha-eq-mult-2}
        &
          =
        &
          \satex{und-cocart-lift-bar-alpha-eq-mult-3}
        \\
        =
        &
          \satex{und-cocart-lift-bar-alpha-eq-mult-4}
        &
          =
        &
          \satex{und-cocart-lift-bar-alpha-eq-mult-5}
        &
          =
        &
          \satex{und-cocart-lift-bar-alpha-eq-mult-6}
      \end{array}
    \]
    \caption{Proof of $\bar \alpha \circ (\mu^T \bar U) =
      (\bar U
      \mu^S) \circ (\bar \alpha T) \circ (T' \bar \alpha)$.}
    \label{fig:bar-alpha-mu}
  \end{figure}
  Thus, we conclude that $(\bar U,\bar \alpha)$ is a monad functor. We verify
  that it is a factorization of $(U,\alpha)$ through $(R,RS\eps)$:
  \[
    \satex{und-cocart-lift-factor-1}
    \qquad
    =
    \qquad
    \satex{und-cocart-lift-factor-2}
    \qquad
    =
    \qquad
    \satex{und-cocart-lift-factor-3}
    \zbox{\qquad.}
  \]
  Moreover, by the zigzag equations, we observe that the vertical 
  pasting operation
  \[
    \begin{tikzcd}
      \cD
      \ar[d,"T"']
      \ar[r,"\bar U"]
      \cphar[rd,"\Downarrow \beta"]
      &
      \cD'
      \ar[d,"T'"]
      \\
      \cD
      \ar[r,"\bar U"']
      &
      \cD'
    \end{tikzcd}
    \qquad
    \mapsto
    \qquad
    \begin{tikzcd}
      \cC
      \ar[r,"R"]
      \ar[d,"S"']
      \cphar[rd,"\Downarrow RS\eps"]
      &
      \cD
      \ar[d,"T"]
      \ar[r,"\bar U"]
      \cphar[rd,"\Downarrow \beta"]
      &
      \cD'
      \ar[d,"T'"]
      \\
      \cC
      \ar[r,"R"']
      &
      \cD
      \ar[r,"\bar U"']
      &
      \cD'
    \end{tikzcd}
  \]
  is a bijective operation between natural transformations of adequate types,
  whose inverse is $\beta' \mapsto (\beta' L) \circ (T'\bar U \eta)$ where
  $\eta$ is the unit of $L \dashv R$. Thus, we conclude that $(\bar U,\bar
  \alpha)$ is the unique monad functor above $\bar U$ which factors $(U,\alpha)$
  through $(R,RS\eps)$. Thus, the latter is a cocartesian morphism for $\mndund$.
\end{proof}

\begin{rem}
  The morphism of \Cref{prop:ex-algra} is even \emph{$2$-cocartesian}, meaning
  that given a monad transformation $\theta \co (U,\alpha) \To (U',\alpha') \co
  (\cD,T) \to (\cD',T')$ and $\bar \theta \co \bar U \To \bar U'$ such that
  $\bar \theta R = \theta$, then $\bar\theta$ uniquely lifts through $\mndund$
  to a monad modification $\bar \theta \co (\bar U,\bar \alpha) \To (\bar
  U',\bar \alpha')$ which factors $\theta$ seen as a monad modification through
  $(R,RS\eps)$. The proof that $\bar \theta$ induces a monad transformation is
  done by observing that the bijection introduced in the end of the proof of
  \Cref{prop:ex-algra} generalizes to a bijection
  \[
    \begin{tikzcd}
      \cD
      \ar[d,"T"']
      \ar[r,"\bar U"]
      \cphar[rd,"\Downarrow \beta"]
      &
      \cD'
      \ar[d,"T'"]
      \\
      \cD
      \ar[r,"\bar U'"']
      &
      \cD'
    \end{tikzcd}
    \qquad
    \mapsto
    \qquad
    \begin{tikzcd}
      \cC
      \ar[d,"S"']
      \ar[r,"U"]
      \cphar[rd,"\Downarrow \beta'"]
      &
      \cD'
      \ar[d,"T'"]
      \\
      \cC
      \ar[r,"U'"']
      &
      \cD'
    \end{tikzcd}
    \zbox{\quad.}
  \]
  Then, the equation required for $\bar\theta$ to be a monad transformation is
  given by the one of $\theta$, through this bijection.
\end{rem}

Given a situation as in the statement of \Cref{prop:ex-algra}, we might wonder
whether the adjunction $L \dashv R$ lifts to an adjunction in $\MND$. In fact,
we can ask this question in a more general setting: given $(\cC,S)$ and
$(\cD,T)$ two objects of $\MND$ and $(R,\rho) \co (\cC,S) \to (\cD,T) \in
\MND$ such that $R = \mndund(R,\rho)$ is part of an adjunction $L \dashv R$,
when does this adjunction lift to an adjunction in $\MND$, \ie there exists
$\lambda \co SL \To LT$ such that $(L,\lambda) \dashv (R,\rho)$ in $\MND$?

Given the situation just described, consider the natural transformation $\rho^*
\co LT \To SL$, also called the \emph{mate}~\cite{kelly1974review} of $\rho$,
defined by
\[
  \rho^*
  \qquad
  =
  \qquad
  \satex{und-adj-lift-rho-star-def}
  \zbox{\qquad.}
\]
We then have the following property:
\begin{prop}
  \label{prop:und-adj-lift}
  Given an adjunction $L \dashv R \co \cC \to \cD$ and a monad functor
  \[
    (R,\rho) \co (\cC,S) \to (\cD,T),
  \]
  the following are equivalent:
  \begin{enumerateroman}
  \item \label{prop:und-adj-lift:lambda} there exists $\lambda \co SL \To LT$
    such that $(L,\lambda) \dashv (R,\rho)$ is an adjunction in $\MND$, whose
    image by $\mndund$ is the adjunction $L \dashv R$;
  \item \label{prop:und-adj-lift:isom} the natural transformation $\rho^*$ is an isomorphism.
  \end{enumerateroman}
\end{prop}

\begin{proof}
  We first show that \ref{prop:und-adj-lift:lambda} implies
  \ref{prop:und-adj-lift:isom}. So let $\lambda$ be as in
  \ref{prop:und-adj-lift:lambda}. We prove that $\lambda$ is an inverse for
  $\rho^*$. We start by showing that $\lambda \circ \rho^* = \unit{LT}$.
  By the correspondence given by the adjunction $L \dashv R$, it is equivalent
  to show that $(R(\lambda \circ \rho^*)) \circ (\eta T) = \eta T$. From the
  fact that $(L,\lambda) \dashv (R,\rho)$ is an adjunction, we get in particular
  the following equality
  \begin{equation}
    \label{eq:und-adj-lift-eq-eta}
    \begin{tikzcd}[baseline=(mybase.base)]
      \cD
      \ar[r,"L"{description}]
      \ar[d,"T"']
      \ar[rr,"\catunit{\cD}",bend left=50,myname=top]
      \cphar[rd,"\Downarrow\lambda"]
      &
      |[alias=midC]|\cC
      \ar[r,"R"{description}]
      \ar[d,"S"{description}]
      \cphar[rd,"\Downarrow\rho"]
      &
      \cD
      \ar[d,"T"{name=mybase}]
      \\
      \cD
      \ar[r,"L"']
      &
      \cC
      \ar[r,"R"']
      &
      \cD
      \cphar[from=top,to=midC,"\Downarrow \eta"]
    \end{tikzcd}
    \qquad
    =
    \qquad
    \begin{tikzcd}[baseline=(mybase.base)]
      \cD
      \ar[rr,"\catunit{\cD}"]
      \ar[d,"T"']
      \cphar[rrd,"="]
      &
      &
      \cD
      \ar[d,"T"{name=mybase}]
      \\
      \cD
      \ar[rr,"\catunit{\cD}"{description},myname=midhor]
      \ar[rd,"L"',bend right=30]
      &
      {}
      &
      \cD
      \\[-2ex]
      &
      |[alias=midC]|\cC
      \ar[ru,bend right=30,"R"']
      \cphar[to=midhor,"\Downarrow\eta"]
    \end{tikzcd}
  \end{equation}
  from which we can prove the wanted equality as follows:
  \[
    \satex{und-adj-lift-inv-first-1}
    \qquad
    =
    \qquad
    \satex{und-adj-lift-inv-first-2}
    \qquad
    =
    \qquad
    \satex{und-adj-lift-inv-first-3}
    \zbox{\qquad.}
  \]
  Symmetrically, in order to prove that $\rho^* \circ \lambda = \unit{SL}$, it
  is equivalent to prove that
  \[
    (S \eps) \circ ((\rho^* \circ \lambda) R) = S \eps
    \zbox.
  \]
  In order to prove the latter, we use the equality satisfied by $\eps$ as a
  monad transformation:
  \begin{equation}
    \label{eq:und-adj-lift-eq-eps}
    \begin{tikzcd}[baseline=(mybase.base)]
      \cC
      \ar[r,"R"]
      \ar[d,"S"']
      \cphar[rd,"\Downarrow\rho"]
      &
      |[alias=midC]|\cD
      \ar[r,"L"]
      \ar[d,"T"{description}]
      \cphar[rd,"\Downarrow\lambda"]
      &
      \cC
      \ar[d,"S"{name=mybase}]
      \\
      \cC
      \ar[r,"R"'{description}]
      \ar[rr,bend right=50,"\catunit{\cC}"',myname=bot]
      &
      |[alias=midD]|\cD
      \ar[r,"L"'{description}]
      &
      \cC
      \cphar[from=midD,to=bot,"\Downarrow \eps"]
    \end{tikzcd}
    \qquad
    =
    \qquad
    \begin{tikzcd}[baseline=(mybase.base)]
      &
      |[alias=midC]|\cD
      \ar[rd,bend left=30,"L"]
      \\[-2ex]
      \cC
      \ar[rr,"\catunit{\cC}"{description},myname=midhor]
      \ar[d,"S"']
      \cphar[rrd,"="]
      \ar[ru,"R",bend left=30]
      &
      &
      \cC
      \ar[d,"S"{name=mybase}]
      \\
      \cC
      \ar[rr,"\catunit{\cC}"']
      &
      {}
      &
      \cC
      \cphar[from=midC,to=midhor,"\Downarrow\eps"]
    \end{tikzcd}
    \zbox{\qquad.}
  \end{equation}
  Then, the wanted equality can be proved using string diagrams just as before.
  Thus, \ref{prop:und-adj-lift:isom} is proved.

  Conversely, assume that \ref{prop:und-adj-lift:isom} holds. We put $\lambda =
  \finv{(\rho^*)}$ and show that $(L,\lambda)$ is a monad functor. First, we
  need to show that $\lambda \circ (\eta^S L) = (L \eta^T)$, or equivalently,
  that $(\eta^S L) = \rho^* \circ (L \eta^T)$. We show the latter using string diagrams:
  \[
    \satex{und-adj-lift-monfunct-first-1}
    \qqeq
    \satex{und-adj-lift-monfunct-first-2}
    \qqeq
    \satex{und-adj-lift-monfunct-first-3}
    \zbox{\qquad.}
  \]
  Second, we need to show that $\lambda \circ (\mu^S L) = (L \mu^T) \circ
  (\lambda T) \circ (S \lambda)$, or equivalently, that the equation $(\mu^S L) \circ (S \rho^*)
  \circ (\rho^* T) = \rho^* \circ (L \mu^T)$ holds. We show the latter using
  string diagrams again:
  \[
    \satex{und-adj-lift-monfunct-second-1}
    \qqeq
    \satex{und-adj-lift-monfunct-second-2}
    \qqeq
    \satex{und-adj-lift-monfunct-second-3}
    \zbox{\qquad.}
  \]
  Thus, $(L,\lambda)$ is a monad functor. We must now prove that $\eta$ and
  $\eps$ induce monad transformations. First, we show that
  \eqref{eq:und-adj-lift-eq-eta} holds, or equivalently, that
  \[
    \begin{tikzcd}[baseline=(mybase.base)]
      \cD
      \ar[r,"L"{description}]
      \ar[rr,"\catunit{\cD}",bend left=50,myname=top]
      &
      |[alias=midC]|\cC
      \ar[r,"R"{description}]
      \ar[d,"S"{description}]
      \cphar[rd,"\Downarrow\rho"]
      &
      \cD
      \ar[d,"T"{name=mybase}]
      \\
      &
      \cC
      \ar[r,"R"']
      &
      \cD
      \cphar[from=top,to=midC,"\Downarrow \eta"]
    \end{tikzcd}
    \qquad
    =
    \qquad
    \begin{tikzcd}[baseline=(mybase.base)]
      \cD
      \ar[r,"T"]
      \ar[d,"L"']
      \cphar[rd,"\Downarrow\rho^*"]
      &
      \cD
      \ar[r,"\catunit{\cD}"]
      \ar[d,"L"{description}]
      \cphar[rd,"\Downarrow\eta"]
      &
      \cD
      \ar[d,"\catunit{\cD}"]
      \\
      \cC
      \ar[r,"S"']
      &
      \cC
      \ar[r,"R"']
      &
      \cD
    \end{tikzcd}
    \zbox{\quad.}
  \]
  But the latter equation follows directly from the string diagram definition of
  $\rho^*$ and the zigzag equations of the adjunction $L \dashv R$.

  The other required equation \eqref{eq:und-adj-lift-eq-eps}, or equivalently,
  \[
    \begin{tikzcd}[baseline=(mybase.base)]
      \cC
      \ar[r,"R"]
      \ar[d,"S"'{name=mybase}]
      \cphar[rd,"\Downarrow\rho"]
      &
      |[alias=midC]|\cD
      \ar[d,"T"{description,name=mybase}]
      &
      \\
      \cC
      \ar[r,"R"'{description}]
      \ar[rr,bend right=50,"\catunit{\cC}"',myname=bot]
      &
      |[alias=midD]|\cD
      \ar[r,"L"'{description}]
      &
      \cC
      \cphar[from=midD,to=bot,"\Downarrow \eps"]
    \end{tikzcd}
    \qqeq
    \begin{tikzcd}[baseline=(mybase.base)]
      \cC
      \ar[r,"R"]
      \ar[d,"\catunit{\cC}"']
      \cphar[rd,"\Downarrow\eps"]
      &
      \cD
      \ar[r,"T"]
      \ar[d,"L"{description,name=mybase}]
      \cphar[rd,"\Downarrow\rho^*"]
      &
      \cD
      \ar[d,"L"]
      \\
      \cC
      \ar[r,"\catunit{\cC}"']
      &
      \cC
      \ar[r,"S"']
      &
      \cC
    \end{tikzcd}
  \]
  holds by a similar argument. Finally, the zigzag equations for $\eta$ and
  $\eps$ seen as monad transformations follows from the zigzag equations
  satisfied by them as natural transformation in $\CAT$. Hence,
  \ref{prop:und-adj-lift:lambda} holds.
\end{proof}

\noindent We have the same kind of property for left adjoints (first suggested
by Dominic Verity to the author):
\begin{prop}
  \label{prop:und-adjl-lift}
  Given an adjunction $L \dashv R \co \cC \to \cD$ and a monad functor
  \[
    (L,\lambda) \co (\cD,T) \to (\cC,S),
  \]
  the following are equivalent:
  \begin{enumerateroman}
  \item \label{prop:und-adjl-lift:rho} there exists $\rho \co SR \To RT$
    such that $(L,\lambda) \dashv (R,\rho)$ is an adjunction in $\MND$, whose
    image by $\mndund$ is the adjunction $L \dashv R$;
  \item \label{prop:und-adjl-lift:isom} the natural transformation $\lambda$ is
    an isomorphism.
  \end{enumerateroman}
\end{prop}

\begin{proof}
  The proof is essentially the same as the one for \Cref{prop:und-adj-lift}.
  For \ref{prop:und-adjl-lift:isom} implies \ref{prop:und-adjl-lift:rho}, $\rho$
  is now defined by the string diagram
  \[
    \rho
    \qqeq
    \satex{und-adjl-lift-rho}
  \]
  and we prove just as before that it induces a monad functor $(R,\rho)$ such
  that $(L,\lambda) \dashv (R,\rho)$ is an adjunction in $\MND$.
\end{proof}

\begin{prop}
  \label{prop:mnd-morph-eta}
  Given two adjunctions $L \dashv R$ and $\bar L \dashv \bar R$, in the configuration
  \[
    \begin{tikzcd}
      \cC
      \ar[r,shift left,"L"]
      &
      \cD
      \ar[l,shift left,"R"]
      \ar[r,shift left,"\bar L"]
      &
      \cE
      \ar[l,shift left,"\bar R"]
    \end{tikzcd}
    \zbox{,}
  \]
  if we write $S$ for the monad associated with $L \dashv R$ and $T$ for the
  monad associated with $\bar L L \dashv R \bar R$, and $\bar \eta$ for the unit
  of $\bar L \dashv \bar R$, we have a monad functor
  \[
    (\catunit[\cC],R \bar \eta L) \co (\cC,T) \to (\cC,S)
    \zbox.
  \]
\end{prop}

\begin{proof}
  The fact that it is a monad functor can be checked using string diagrams. On
  the one hand, the compatibility of $R \bar \eta L$ with the units of $S$ and
  $T$ asserts that the two diagrams
  \[
    \satex{lem-mnd-morph-eta-unit-l}
    \qqtand
    \satex{lem-mnd-morph-eta-unit-r}
  \]
  represents the same cell, which is true by definition of the unit of $T$. On
  the other hand, the compatibility of $R \bar \eta L$ with the multiplications
  of $S$ and $T$ asserts that the two diagrams
  \[
    \satex{lem-mnd-morph-eta-mult-l}
    \qqtand
    \satex{lem-mnd-morph-eta-mult-r}
  \]
  represents the same cell, which is true by the zigzag equations of $\bar L
  \dashv \bar R$. Thus, $(\catunit[\cC],R \bar \eta
  L)$ is indeed a monad functor.
\end{proof}


\section{Higher categories as globular algebras}
\label{text:higher-categories-globular-algebras}
The notion of ``higher category'' encompasses informally all the structures that
have higher-dimensional cells which can be composed together with several
operations. Such structures can differ on many points. First, there are several
possible shapes for the cells of higher categories. For example, \emph{globular
  higher categories} have $0$\cells, $1$\cells, $2$\cells, $3$\cells, \etc of
the form
\[
  \begin{tikzcd}[cramped]
    x
  \end{tikzcd}
  ,
  \qquad
  \quad
  \begin{tikzcd}[cramped,column sep={between origins,5em}]
    x
    \ar[r,"f"]
    &
    y
  \end{tikzcd}
  ,
  \qquad
  \quad
  \begin{tikzcd}[cramped,column sep={between origins,5em}]
    x
    \ar[r,bend left=70,"f",""{name=f,auto=false}]
    \ar[r,bend right=70,"g"'{pos=0.52},""{name=g,auto=false}]
    &
    y
    \ar[from=f,to=g,phantom,"\Downarrow\!\phi"]
  \end{tikzcd}
  ,
  \qquad
  \quad
  \begin{tikzcd}[cramped,column sep={between origins,5em}]
    x
    \ar[r,bend left=70,"f",""{name=f,auto=false}]
    \ar[r,bend right=70,"g"'{pos=0.52},""{name=g,auto=false}]
    &
    y
    \ar[from=f,to=g,phantom,"\phi\!\Downarrow\!\smash{\xTO{F}}\!\Downarrow\!\psi"]
  \end{tikzcd}
  ,
  \qquad
  \quad
  \text{\etc}
\]
But one can consider higher categories with other shapes than the globular ones.
Common variants include \emph{cubical}~\cite{al2000multiple} and
\emph{simplicial}~\cite{joyal2002quasi} higher categories, whose $2$\cells for
example are respectively of the form
\[
  \begin{tikzcd}[cramped,column sep={between origins,5em}]
    x
    \ar[r,"f"]
    \ar[d,"g"']
    \ar[rd,phantom,"\Downarrow\!\phi"]
    &
    y
    \ar[d,"h"]
    \\
    x'
    \ar[r,"f'"']
    &
    y'
  \end{tikzcd}
  \qquad
  \qtand
  \qquad
  \begin{tikzcd}[cramped,column sep={between origins,3.5em}]
    &
    |[alias=y]| y
    \ar[rd,"g"]
    & \\
    x
    \ar[ru,"f"]
    \ar[rr,"h"',""{auto=false,name=h}]
    & & z
    \ar[from=y,to=h,phantom,"\Downarrow\!\phi"]
  \end{tikzcd}
  \pbox.
\]
Moreover, higher categories have several operations which satisfy axioms that
can take different forms, according to their position in the strict/weak
spectrum. For example, a \emph{strict $2$\category} is a globular
$2$\dimensional category that have, among others, an operation~$\comp_0$ to
compose $1$\cells in dimension~$0$, as in
\begin{gather*}
  \begin{tikzcd}[ampersand replacement=\&,cramped,column sep={between origins,5em}]
    x
    \ar[r,"f"]
    \&
    y
  \end{tikzcd}
  \;
  \comp_0
  \;
  \begin{tikzcd}[ampersand replacement=\&,cramped,column sep={between origins,5em}]
    y
    \ar[r,"\vphantom x\smash g"]
    \&
    z
  \end{tikzcd}
  \;=\;
  \begin{tikzcd}[ampersand replacement=\&,cramped,column sep={between origins,5em}]
    x
    \ar[r,"f\comp_0 g"]
    \&
    z
  \end{tikzcd}
  \zbox,
\end{gather*} 
and operations~$\comp_0$ and~$\comp_1$ to compose $2$\cells in dimensions~$0$
and~$1$ respectively, as in
\begingroup
\allowdisplaybreaks
\begin{gather*}
  \begin{tikzcd}[ampersand replacement=\&,cramped,column sep={between origins,5em}]
    x
    \ar[r,bend left=70,"f",""{name=f,auto=false}]
    \ar[r,bend right=70,"g"'{pos=0.52},""{name=g,auto=false}]
    \&
    y
    \ar[from=f,to=g,phantom,"\Downarrow\!\phi"]
  \end{tikzcd}
  \;
  \comp_0
  \;
  \begin{tikzcd}[ampersand replacement=\&,cramped,column sep={between origins,5em}]
    y
    \ar[r,bend left=70,"f'",""{name=f,auto=false}]
    \ar[r,bend right=70,"{g'}"'{pos=0.49},""{name=g,auto=false}]
    \&
    z
    \ar[from=f,to=g,phantom,"\Downarrow\!\phi'"]
  \end{tikzcd}
  \;
  =
  \;
  \begin{tikzcd}[ampersand replacement=\&,cramped,column sep={between origins,5em}]
    x
    \ar[r,bend left=70,"f \comp_0 f'",""{name=f,auto=false}]
    \ar[r,bend right=70,"g \comp_0 g'"'{pos=0.52},""{name=g,auto=false}]
    \&
    z
    \ar[from=f,to=g,phantom,"\Downarrow\!\phi\! \comp_0 \!\phi'"]
  \end{tikzcd}
  \shortintertext{and}
  \begin{tikzcd}[ampersand replacement=\&,cramped,column sep={between origins,5em}]
    x
    \ar[r,bend left=70,"f",""{name=f,auto=false}]
    \ar[r,bend right=70,"g"'{pos=0.52},""{name=g,auto=false}]
    \&
    y
    \ar[from=f,to=g,phantom,"\Downarrow\!\phi"]
  \end{tikzcd}
  \;
  \comp_1
  \;
  \begin{tikzcd}[ampersand replacement=\&,cramped,column sep={between origins,5em}]
    x
    \ar[r,bend left=70,"g",""{name=f,auto=false}]
    \ar[r,bend right=70,"{h}"'{pos=0.52},""{name=g,auto=false}]
    \&
    y
    \ar[from=f,to=g,phantom,"\Downarrow\!\psi"]
  \end{tikzcd}
  \;
  =
  \;
  \begin{tikzcd}[ampersand replacement=\&,cramped,column sep={between origins,5em}]
    x
    \ar[r,bend left=70,"f",""{name=f,auto=false}]
    \ar[r,bend right=70,"h"'{pos=0.52},""{name=g,auto=false}]
    \&
    y
    \ar[from=f,to=g,phantom,"\Downarrow\!\phi\! \comp_1 \!\psi"]
  \end{tikzcd}
  \zbox.
\end{gather*}
\endgroup
These operations are required to satisfy several axioms consisting in
equalities, like the associativity axiom: given $0$\composable~$1$-
or~$2$\cells~$u,v,w$, \abovelongtoshortskip
\begin{align*}
  (u \comp_0 v) \comp_0 w &= u \comp_0 (v \comp_0 w)
  \shortintertext{and, given $1$\composable $2$\cells~$\phi,\psi,\chi$,}
  (\phi \comp_1 \psi) \comp_1 \chi &= \phi \comp_1 (\psi \comp_1 \chi)\zbox.
\end{align*}
\belowlongtoshortskip
An example of a weak higher category is given by a
\index{bicategory}\emph{bicategory}, which is a globular $2$\dimensional
category that has operations similar to a strict $2$\category but which satisfy
axioms in the form of ``weak equalities''. For example, the $0$\composition of
$1$\cells is only required to be weakly associative, in the sense that, given
$0$\composable $1$\cells
\[
  \begin{tikzcd}[ampersand replacement=\&,cramped,column sep={between origins,5em}]
    w\ar[r,"f"]\& x\ar[r,"g"]\& y\ar[r,"h"]\& z
  \end{tikzcd}
  \zbox,
\]
the equality $(f \comp_0 g) \comp_0 h = f \comp_0 (g \comp_0 h)$ does not hold
necessarily, but there should exist a \index{coherence cell}\emph{coherence
  cell} between the two sides, \ie an invertible $2$\cell~$\alpha_{f,g,h}$ as in
\[
  \begin{tikzcd}[cramped,column sep={between origins,5em}]
    w
    \ar[r,bend left=70,"(f \comp_0 g)\comp_0 h",""{name=f,auto=false}]
    \ar[r,bend right=70,"f \comp_0 (g \comp_0 h)"'{pos=0.52},""{name=g,auto=false}]
    &
    z
    \ar[from=f,to=g,phantom,"\Downarrow\!\alpha_{f,g,h}"]
  \end{tikzcd}
  \zbox.
\]
Finally, a subtle difference between the different kinds higher categories is
the \index{algebraic higher category}\emph{algebraicity} of their
definition~\cite{leinster2004higher,gurski2013coherence}. This notion
essentially pertains to weak higher categories. Informally, a definition of some
sort of higher categories is algebraic when it can be equivalently described by
means of a monad. Concretely, algebraic definitions of weak higher categories
involve coherence cells that are \emph{distinguished} (like the definition of
bicategories, which requires that ``there exists an invertible
$2$\cell~$\alpha_{f,g,h}$ between~$(f\comp_0 g) \comp_0 h$ and~$f \comp_0 (g
\comp_0 h)$''), whereas non-algebraic definitions of weak higher categories
involve coherence cells that are not (a non-algebraic definition of bicategories
would only require that ``there exists \emph{some} invertible $2$\cell
between~$(f\comp_0 g) \comp_0 h$ and~$f \comp_0 (g \comp_0 h)$'').


In the remainder of this article, we will restrain our attention to globular algebraic higher
categories. A particular theory of $k$\categories can be seen as a monad on the
category of $k$\globular sets, and a $k$\category instance of this theory is an
algebra for this monad. A lot of globular higher categories that one usually
encounters fit in this setting: strict $k$\categories, bicategories, Gray
categories, \etcend Several constructions can then be defined at this level of
abstractions, like the truncation of globular algebras and its left adjoint.

In \Cref{ssec:globular-sets}, we recall the definition of globular sets and
elementary operations on them. We then introduce globular algebras as algebras
of a monad on globular sets in \Cref{text:globular-algebras}, and then define
the truncation and inclusion functors which relate globular algebras from
different dimensions in \Cref{text:alg-trunc-incl-functors}. Then, in the case
of a monad on $\nGlob\omega$, we use the truncation functors to formally express
$\Alg_\omega$ as a (bi)limit over the $\Alg_k$'s in \Cref{text:algo-as-a-limit}.
In \Cref{text:criterion-for-globular-algebras}, we introduce a criterion to
recognize a tower of categories as equivalent to the categories of globular
algebras over a monad, simplifying the concrete use of the theory developped in
this article. In \Cref{text:truncable-monads}, we introduce the notion of
truncable monads, which correspond more to the notions of higher categories that
we are accustomed to than general globular monads, and, in
\Cref{text:charact-truncable}, we prove a criterion to easily recognize such
monads in the wild.

\subsection{Globular sets and operations}
\label{ssec:globular-sets}

Here, we recall the classical notion of \emph{globular set}. It is the
underlying structure of a globular higher category which describes \emph{globes}
of different dimensions together with their sources and targets. We moreover
define the truncation and inclusion functors between globular sets of different
dimensions.

\paragraph{Definition}
\parlabel{text:globular-set-def}

Given~$n\in\N\cup \set\omega$, an \index{globular set}\emph{$n$-globular
  set}~$(X,\csrc,\ctgt)$ (often simply denoted~$X$) is the data of sets~$X_k$
for~$k\in \N_n$ together with \glossary(dround){$\csrc_i,\ctgt_i$}{the source
  and target operations of a globular set}functions~$\csrc_i,\ctgt_i \co X_{i+1}
\to X_i$ for~$i \in \N_{n-1}$ as in
\[
  \begin{tikzcd}
    X_0
    &
    \ar[l,shift right,"\csrc_0"']
    \ar[l,shift left,"\ctgt_0"]
    X_1
    &
    \ar[l,shift right,"\csrc_1"']
    \ar[l,shift left,"\ctgt_1"]
    X_2
    &
    \ar[l,shift right,"\csrc_2"']
    \ar[l,shift left,"\ctgt_2"]
    \cdots
    &
    \ar[l,shift right,"\csrc_{k-1}"']
    \ar[l,shift left,"\ctgt_{k-1}"]
    X_k
    &
    \ar[l,shift right,"\csrc_{k}"']
    \ar[l,shift left,"\ctgt_{k}"]
    X_{k+1}
    &
    \ar[l,shift right,"\csrc_{k+1}"']
    \ar[l,shift left,"\ctgt_{k+1}"]
    \cdots
  \end{tikzcd}
\]
such that
\[
  \csrc_i\circ \csrc_{i+1}=\csrc_i\circ \ctgt_{i+1}
  \qtand
  \ctgt_i\circ \csrc_{i+1}=\ctgt_i\circ \ctgt_{i+1}
  \quad
  \text{for~$i \in \N_{n-1}$.} 
\]
When there is no ambiguity on~$i$, we often write~$\csrc$ and~$\ctgt$
for~$\csrc_i$ and~$\ctgt_i$. An element~$u$ of~$X_i$ is called an
\index{globe}\emph{$i$-globe} of~$X$ and, for~$i>0$, the globes~$\csrc_{i-1}(u)$
and~$\ctgt_{i-1}(u)$ are respectively called the
\index{source!of a globe}\emph{source} and
\index{target!of a globe}\emph{target} and~$u$. Given $n$\globular sets~$X$
and~$Y$, a
\index{morphism!of globular sets}\emph{morphism of $n$\globular set}
between~$X$ and~$Y$ is a family of functions~$F = (F_k\co X_k\to Y_k)_{k \in
  \N_n}$, such that
\[
  \csrc_i\circ F_{i+1}=F_i\circ \csrc_i
  \qquad
  \text{for~$i \in \N_{n-1}$.}
\]
We \glossary(Glob){$\nGlob n$}{the category of $n$\globular sets}write~$\nGlob
n$ for the category of $n$\globular sets.
\begin{remark}
  \label{rem:glob-ess-alg}
  The above definition directly translates to an essentially algebraic theory,
  so that~$\nGlob n$ is essentially algebraic. In particular,~$\nGlob n$ is
  locally finitely presentable, complete and cocomplete by
  \myCref[Theorem~1.1.1.1]{prop:ess-alg-iff-loc-fin-pres} and
  \Cref{prop:loc-pres-nice-properties}.
\end{remark}
%

\noindent For~$\eps\in\set{-,+}$ and~$j\geq 0$, we write
\[
  \csrctgt\eps_{i,j}=\csrctgt\eps_{i}\circ\csrctgt\eps_{i+1}\circ\cdots\circ\csrctgt\eps_{i+j-1}
\]
for the \index{source!iterated}\emph{iterated source} (when~$\eps=-$) and \index{target!iterated}\emph{target} (when~$\eps=+$)
operations. We generally omit the index~$j$ when there is no ambiguity and
simply write~$\csrctgt\eps_i(u)$ for~$\csrctgt\eps_{i,j}(u)$.
Given~$i,k,l\in\N_n$ with~$i < \min(k,l)$, we \glossary(.abXkxXl){$X_k\times_i
  X_l$}{the set of pairs of $i$\composable~$k$- and $l$\globes}write~$X_k\times_i X_l$ for the
pullback
\[
  \begin{tikzcd}
    X_k\times_i X_l \ar[r,dotted] \ar[d,dotted]
    \ar[rd,phantom,"\drcorner",very near start]& X_l\ar[d,"\csrc_i"]\\
    X_k\ar[r,"\ctgt_i"']& X_i
  \end{tikzcd}
  \pbox.
\]
Given~$p \ge 2$ and~$k_1,\ldots,k_p \in \N_n$, a sequence of globes~$u_1 \in
X_{k_1}, \ldots, u_p \in X_{k_p}$ is said
\index{composable}\emph{$i$\composable} for some~$i < \min(k_1,\ldots,k_p)$,
when~$\ctgt_{i}(u_j) = \csrc_i(u_{j+1})$ for~$j \in \N^*_{p-1}$. Given~$k \in
\N_n$ and~$u,v \in X_{k}$,~$u$ and~$v$ are said \index{parallel}\emph{parallel}
when~${k = 0}$ or~$\csrctgt\eps_{k-1}(u) = \csrctgt\eps_{k-1}(v)$ for~${\eps \in
  \set{-,+}}$. To remove the side condition~$k = 0$, we use the convention
\glossary(.acXmo){$X_{-1}$}{by convention, the singleton set of $({-}1)$\globes of a globular
  set}that~$X_{-1}$ is the set~$\set{\ast}$ and that~$\csrc_{-1},\ctgt_{-1}$ are
the unique function~$X_0 \to X_{-1}$.

For~$u\in X_{i+1}$, we sometimes write~$u\co v\to w$ to indicate
that~$\csrc_i(u)=v$ and~$\ctgt_i(u)=w$. In low dimension, we use $n$-arrows such
\glossary(.adarrow){$\to$, $\To$, $\TO$, $\TOO$}{arrows which indicate the source
  and target of globes of a globular set}as~$\To$,~$\TO$,~$\TOO$, \etc to indicate the sources and the targets of
$n$-globes in several dimensions. For example, given a $2$\globular set~$X$
and~$\phi \in X$, we sometimes write~$\phi\co f \To g \co x \to y$ to indicate
that
\[
  \phi \in X_2,
  \quad 
  \csrc_1(\phi) = f,
  \quad 
  \ctgt_1(\phi) = g,
  \quad
  \csrc_0(\phi) = x
  \qtand
  \ctgt_0(\phi) = y.
\]
We also use these arrows in graphical representations to picture the
elements of a globular set~$X$. For example, given an $n$\globular set~$X$ with~$n \ge 2$, the drawing
\begin{equation}
  \label{eq:some-globular-set}
  \begin{tikzcd}[sep=4em]
    x
    \ar[r,"f",bend left=70,""{auto=false,name=fst}] 
    \ar[r,"g"{description},""{auto=false,name=snd}]
    \ar[r,"h"',bend right=70,""{auto=false,name=trd}]
    &
    y
    \ar[from=fst,to=snd,phantom,"\Downarrow\!\phi"]
    \ar[from=snd,to=trd,phantom,"\Downarrow\!\psi"]
    \ar[r,"k"]
    &
    z
  \end{tikzcd}
\end{equation}
figures two $2$\cells~$\phi,\psi \in X_2$, four $1$\cells~$f,g,h,k \in X_1$ and
three $0$\cells~$x,y,z \in X_0$ such that
\begin{gather*}
  \csrc_1(\phi) = f, \qquad \ctgt_1(\phi) = \csrc_1(\psi)= g, \qquad 
  \ctgt_1(\psi) = h, \\
  \csrc_0(f) = \csrc_0(g) = \csrc_0(h) = x, \qquad
  \ctgt_0(f) = \ctgt_0(g) = \ctgt_0(h) = \csrc_0(k) = y, \qquad \ctgt_0(k) = z.
\end{gather*}

\paragraph{Truncation and inclusion functors} 
\parlabel{text:gs-trunc-incl-functors}

Given~$m \in \N_n$ and~$X \in \nGlob n$, we denote by~$\restrictcat X m$ the
\index{truncation!of a globular set}\emph{$m$\nbd-trun\-ca\-tion of~$X$}, \ie
the $m$\globular set obtained from~$X$ by removing the $i$-globes for~$i \in
\N_n$ with~$i > m$. This operation extends to a
\glossary(Globatrunc){$\gtruncf[n] m {(-)}$, $\restrictcat X m$}{the truncation
  functor on globular sets}functor
\[
  \gtruncf[n] m {(-)}\co \nGlob n \to \nGlob m
\]
often denoted~$\gtruncf m {(-)}$ when there is no ambiguity. This functor admits
a left \glossary(Globbincl){$\gincf[m] n {(-)}$, $\inc n X$}{the inclusion
  functor on globular sets}adjoint
\[
  \gincf[m] n {(-)}\co \nGlob
  m \to \nGlob n
\]
often denoted~$\gincf n {(-)}$ when there is no ambiguity, and which maps an
$m$\globular set~$X$ to the $n$\globular set~$\inc n X$, called
\index{inclusion!of a globular set}\emph{$n$\nbd-inclu\-sion of~$X$}, and which
is defined by~${\restrict m {(\inc n X)} = X}$ and~${(\inc n X)_i = \emptyset}$
for~$i \in \N_n$ with~$i > m$. The unit of the adjunction~$\gincf n {(-)} \dashv
\gtruncf m {(-)}$ is the identity and the counit is the natural transformation
\glossary(i){$\gtrunccu[n] {m}$, $\gtrunccu {m}$}{the counit of the adjunction
  $\gincf[m] n {(-)} \dashv \gtruncf[n] m {(-)}$}denoted~$\gtrunccu[n] {m}$, or
simply~$\gtrunccu {m}$ when there is no ambiguity, which is given by the family
of canonical morphisms
\[
  \gtrunccu m_X\co\inc n {(\restrict m X)} \to X
\]
for~$X \in \nGlob n$. The functor~$\gtruncf[n] m {(-)}$ also admits a
\glossary(Globcfill){$\gincfill[n] m -$, $\incfill n X$}{the right adjoint to
  the truncation functor~$\gtruncf[m]n-$}right
adjoint
\[
  \gincfill[n] m - \co \nGlob m \to \nGlob n
\]
denoted~$\gincfill n -$ when there is no ambiguity, and which maps an
$m$\globular set to the $n$\globular set~$\incfill n X$ defined by~$\restrict m
{(\incfill n X)} = X$, and, for~$i \in \N_n$ with~$i > m$,
\[
  (\incfill n X)_i = \set{ (u,v) \in X_m \mid \text{$u$ and~$v$ are parallel}}
\]
such that, for~$(u,v) \in (\incfill n X)_i$,
\begin{gather*}
  \csrc_{m}((u,v)) = u
  \qtand
  \ctgt_m((u,v)) = v
  \shortintertext{and}
  \csrc_j((u,v)) = \ctgt_j((u,v)) = (u,v)
  \quad
  \text{for~$j \in \N_{i-1}$.} 
\end{gather*}
\noindent Note that, since they are left adjoints, the functors~$\gincf[m] {n}
{(-)}$ and~$\gtruncf[n] {m} {(-)}$ preserves colimits.

\subsection{Globular algebras}
\label{text:globular-algebras}

We now introduce categories of \index{globular algebra}\emph{globular algebras},
\ie the Eilenberg-Moore categories induced by monads on globular sets, as were
first introduced by Batanin in~\cite{batanin1998computads}. We moreover give
several additional constructions and properties on these objects.

\parlabel{text:globular-algebra-definition}

Let~$n \in \N \cup \set\omega$ and~$(T,\eta,\mu)$ be a finitary monad on~$\nGlob
n$. We write~$\Alg_n$ for the category of~$T$\algebras~$\nGlob n^{T}$ and
\[
  \fgfalgf_n\co \Alg_n \to \nGlob n \qquad \qquad \freealgf_n\co \nGlob n \to \Alg_n
\]
for the induced left and right adjoints, that were denoted~$\bigfun U^T$
and~$\bigfun F^T$ in \Cref{text:algebras-over-a-monad}: given an algebra~$(X,h)
\in \Alg_n$, the image of~$(X,h)$ by~$\fgfalgf_n$ is~$X$ and, given~$Y \in
\nGlob n$,~$\freealgf_nY$ is the free $T$\algebra
\[
  (TY,\mu_Y\co TTY \to TY).
\]
Given $k < n$, using \Cref{prop:ex-algra} with the adjunction $\gtruncf k - \dashv \gincf n -$, we
get a monad $(T^k,\eta^k,\mu^k)$ on $\nGlob k$ where
\[
  T^k = \gtruncf k - T \gincf n -
\]
and such that~$\eta^{k} \co \unit {\nGlob k} \to T^k$ is the composite
\[
  \begin{tikzcd}[column sep=5em,cramped]
    \unit {\nGlob k} \ar[r,equal] & \gtruncf k {(-)}  \gincf n {(-)}
    \ar[r,"{\gtruncf k {(-)} \eta  \gincf n {(-)}}"] & T^k
  \end{tikzcd}
\]
\ie~$\eta^{k}_X = \restrict k {(\eta_{\inc n X})}$ for~$X \in \nGlob k$, and
such that~$\mu_k \co T^k  T^k \to T^k$ is the composite
\[
  \begin{tikzcd}[column sep=8em,cramped]
    T^k  T^k
    \ar[r,"{\gtruncf k {(-)}  T  \gtrunccu k {}  T  \gincf n {(-)}}"]
    & 
    \gtruncf k {(-)}  T  T  \gincf n {(-)}
    \ar[r,"\gtruncf k {(-)}  \mu  \gincf n {(-)}"]
    &
    T^k
  \end{tikzcd}
  \zbox.
\]
So, for~$k\in\N_n$, we get a category $\Alg_k = \nGlob k^{T^k}$, and
canonical functors
\[
  \fgfalgf_k\co \Alg_k \to \nGlob k \qquad \qquad \freealgf_k\co \nGlob k \to
  \Alg_k
\]
defined like~$\fgfalgf_n$ and~$\freealgf_n$ above, forming an adjunction
$\freealgf_k \dashv \fgfalgf_k$ whose counit is denoted $\eps^k$. The objects
of~$\Alg_k$ are called \index{category@$k$-category}\emph{$k$\categories}.
Moreover, given a $k$\category~$C=(X,h)$, the elements of~$X_i$ are called the
\index{cell!of a $k$-category}\emph{$i$\cells of~$C$} for~$i \in \N_k$.
\begin{remark}
  In the above definition, we require that the monad~$(T,\eta,\mu)$ is finitary
  in order to prove later the existence of several free constructions on the
  $k$\categories. This is not too restrictive, since it includes all the monads
  of algebraic globular higher categories that have operations with finite
  arities, \ie most theories of algebraic globular higher categories.
\end{remark}
\noindent We can already derive several properties of the categories~$\Alg_k$:
\begin{prop}
  \label{prop:alg-cat-loc-fin-pres}
  For~$k \in \N \cup \set n$, the category~$\Alg_k$ is locally finitely
  presentable. In particular, it is complete and cocomplete. Moreover, the
  functor~$\fgfalgf_k$ preserves and creates directed colimits, and creates
  limits.
\end{prop}
\begin{proof}
  The category~$\Alg_k$ is locally finitely presentable as a consequence of
  \Propr{loc-fin-pres-alg-cat} since~$\nGlob k$ is locally finitely presentable
  by \myCref[Remark~1.2.3.4]{rem:glob-ess-alg}. The functor~$\fgfalgf_k$
  preserves directed colimits by \Propr{loc-fin-pres-alg-cat}. Moreover,
  since~$\fgfalgf_k$ reflects isomorphisms and~$\Alg_k$ is
  cocomplete,~$\fgfalgf_k$ creates directed colimits. Finally, it is well-known
  that the forgetful functor associated to an Eilenberg-Moore category creates
  limits (see~\cite[Proposition 4.3.1]{borceux1994handbook2} for example).
\end{proof}
\noindent We can usually derive monads from equational definitions of higher
categories as illustrated by the following examples.
\begin{example}
  \label{ex:cat-globular-algebra}
  The canonical forgetful functor~$\Cat \to \nGlob 1$ is a finitary right
  adjoint~(see \Cref{ex:grph-to-cat-theo-morphism} for a detailed argument)
  which thus induces a finitary monad~$(T,\eta,\mu)$ on~$\nGlob 1$. This monad
  maps a $1$\globular set~$G$ to the underlying $1$\globular set of the category
  of paths on~$G$ seen as a graph. Using Beck's monadicity theorem
  (\Cref{thm:monadicity}), one can verify that the functor~$\Cat \to \nGlob 1$
  is monadic, so that~$\Alg_1 \simeq \Cat$. Moreover, the
  monad~$(T^0,\eta^0,\mu^0)$ is essentially the identity monad on~$\nGlob 0$,
  and thus~$\Alg_0 \simeq \Set$. More generally, we will see
  in~\Cref{text:strict-cats-and-precats} that the monads of strict
  $k$\categories for~$k \in \N$ are derived from the monad of strict \ocats.
\end{example}
\begin{example}
  \label{ex:weird-cat-globular-algebra}
  We define a notion of \index{weird $2$-category}\emph{weird $2$\category} as follows: a weird
  $2$\category is a $2$\globular set~$C$ equipped with an operation
  \[
    {\comp} \co C_2 \times C_2 \to C_0\zbox.
  \]
  Note that we do not require the composability of the arguments of~$\comp$, and
  we do not enforce any axiom on~$\comp$. A morphism between two weird
  $2$\categories is then a morphism between the underlying $2$\globular sets
  that is compatible with~$\comp$. The \glossary(Weird){$\mathbf{Weird}$}{the
    category of weird $2$-categories}category~$\mathbf{Weird}$ of weird
  $2$\categories and their morphisms is essentially algebraic, and the functor
  which maps a weird $2$\category to its underlying $2$\globular set is induced
  by an essentially algebraic theory morphism, so that it is a right adjoint and
  finitary by \myCref[Theorem~1.2.3.4]{prop:functor-from-theo-morphism-ra}. From
  the adjunction, we derive a finitary monad~$(T,\eta,\mu)$ on~$\nGlob 2$, and,
  given~$X \in \nGlob 2$, we have that
  \[
    (TX)_0 \cong X_0 \sqcup (X_2 \times X_2) \qquad (TX)_1 \cong X_1 \qquad (TX)_2 \cong X_2
  \]
  so that, for~$\Alg_2$ derived from the monad~$T$,~$\Alg_2 \cong
  \mathbf{Weird}$. Moreover, the monads~$(T^0,\eta^0,\mu^0)$
  and~$(T^1,\eta^1,\mu^1)$ are essentially the identity monads on~$\nGlob 0$
  and~$\nGlob 1$ respectively, so that the associated notions of weird~$0$- and
  $1$\categories are simply~$0$- and $1$\globular sets.
\end{example}
\noindent The previous example moreover illustrates the unusual operations that
notions of higher categories defined in the setting of Batanin can have. It is
also an example of a monad on globular sets which is not truncable (\cf
\Cref{ex:weird-monad-not-truncable}).

\begin{example}
  Monoids can be considered in any category with a monoidal structure. Given a
  theory of globular algebraic higher category, one can define an associated
  notion of strict monoidal higher category by considering monoids in the
  category of algebras, choosing the monoidal structure to be the cartesian one.
  Monadically, we have an operation mapping a monad~$T$ on $n$\globular sets to
  a monad $T'$ on $n$\globular sets representing the theory of strictly monoidal
  higher categories which are instances of~$T$. This operation can be seen to
  preserve a finitary hypothesis on~$T$.
\end{example}

\subsection{Truncation and inclusion functors}
\label{text:alg-trunc-incl-functors}

We now introduce truncation and inclusion functors between the
categories~$\Alg_k$ together with some of their properties.

Let~$n \in \N \cup \set\omega$ and~$(T,\eta,\mu)$ be a finitary monad on~$\nGlob
n$. Given $k < n$, using \Cref{prop:ex-algra}, we get that $\gtruncf[n]k-$ lifts
to a cocartesian morphism
\[
  (\nGlob n, T) \xto{(\gtruncf k -,\gtruncf k - T \gtrunccu k)}
  (\nGlob k,T^k)
\]
with respect to the functor $\mndund$. By applying $\EM$, we get a functor
\[
  \algtruncf[n] k - \co \Alg_n \to \Alg_k
\]
also denoted $\algtruncf k-$ when there is no ambiguity.

Now, given $k,l \in \N$ with $k < l < n$, since $\gtruncf[n]k-$ can be factored
as $\gtruncf[l]k-\gtruncf[n]l-$, the cocartesianness of $(\gtruncf l -,\gtruncf
l - T \gtrunccu l)$ ensures that there is a unique morphism $(F,\alpha)$ which
factorizes $(\gtruncf k -,\gtruncf k - T \gtrunccu k)$ through $(\gtruncf l
-,\gtruncf l - T \gtrunccu l)$ in $\MND$.

\begin{lem}
  \label{lem:gtruncf-cocart-lift}
  Given $k,l \in \N$ with $k < l < n$, the morphism
  \[
    (\gtruncf[l]k-,\gtruncf k-T^l\gtrunccu[l]k) \co (\nGlob l,T^l) \to (\nGlob
    k,T^k)
  \]
  is the factorization of $(\gtruncf k -,\gtruncf k - T \gtrunccu k)$ through
  the cocartesian morphism
  \[
    (\gtruncf l -,\gtruncf l - T \gtrunccu l)
    \zbox.
    \qedhere
  \]
\end{lem}
\begin{proof}
  By \Cref{prop:ex-algra}, there is a morphism
  \[
    (\gtruncf[l]k-,\gtruncf k-T^l\gtrunccu[l]k) \co (\gtruncf l -,T^l) \to
    (\gtruncf k -,T')
  \]
  where $(T',\eta',\mu')$ is the monad on $\gtruncf k -$ defined by
  \begin{align*}
    T' &= \gtruncf k -T^l\gincf l -
    \\
    \eta' &= \gtruncf k -\eta^l\gincf l -
    \\
    \mu' &= (\gtruncf k - \mu^l \gincf k -) \circ (\gtruncf k - T^l \gtrunccu[l]k
    T^l \gincf l -)
           \zbox.
  \end{align*}
  By a straight-forward computation, we get $T' = T^k$, $\eta' = \eta^k$ and
  $\mu' = \mu^k$. To verify that the above morphism is the wanted factorization,
  we are left to check that
  \[
    (\gtruncf[l]k-,\gtruncf k-T^l\gtrunccu[l]k)
    \circ
    (\gtruncf l -,\gtruncf l - T \gtrunccu l)
    =
    (\gtruncf k -,\gtruncf k - T \gtrunccu k)
  \]
  in $\MND$. But it is straight-forward too, since
  \[
    \gtrunccu[n]l \circ (\gincf[n]l-\gtrunccu[l]k\gtruncf[n]l-)
    =
    \gtrunccu[n]k\zbox.
    \qedhere
  \]
\end{proof}
\noindent By applying $\EM$ to the factorization morphism given by
\Cref{lem:gtruncf-cocart-lift}, we get a functor
\[
  \algtruncf[l] k - \co \Alg_l \to \Alg_k
\]
also denoted $\algtruncf k-$ when there is no ambiguity. Concretely, given a
$T^l$\algebra
\[
  (X,h\co T^lX \to X)\zbox,
\]
its image by $\algtruncf k -$ is the $T^k$\algebra~$(\restrict k X,h')$,
where~$h'$ is defined as the composite
\[
  \begin{tikzcd}[column sep=6.5em,cramped]
    T^k(\restrict k X)
    \ar[r,"{({\gtruncf[l] k {(-)}T^{l}\gtrunccu[l] {k}})_X}"]
    &
    \restrict k {(T^{l}X)}
    \ar[r,"\restrict k h"]
    &
    \restrict k X
  \end{tikzcd}
  \zbox.
\]
The image of~$(X,h)$ in~$\Alg_l$ by~$\algtruncf k {(-)}$ is called the
\index{truncation!of a globular algebra}\emph{$k$\truncation} of~$(X,h)$ and we
denote it~$\restrict k {(X,h)}$. Note that the image of a morphism~${f \co (X,h)
  \to (X',h')}$ by~$\smash{\algtruncf k {(-)}}$ is~$\restrict k f$ (the globular
$k$\truncation of~$f$). The same concrete description holds for the functors
$\algtruncf[n] k -$.

\smallpar The definition of the truncation functors allows for the following
compatibility property:
\begin{prop}
  Given~$j,k,l \in \N_n\cup \set n$ with~$j < k < l$, we have
  \[
    \algtruncf[k] j {(-)} \circ \algtruncf[l] {k} {(-)} = \algtruncf[l] j {(-)}
    \zbox.
  \]
\end{prop}
\begin{proof}
  The two morphisms
  \[
    (\gtruncf[k]j-,\gtruncf j-T^k\gtrunccu[k]j)
    \circ
    (\gtruncf[l] k -,\gtruncf[l] k - T^l \gtrunccu[l] k)
    \qtand
    (\gtruncf[l] j -,\gtruncf[l] j - T^l \gtrunccu[l] j)
  \]
  provide a factorization of $(\gtruncf[n]j -,\gtruncf[n] j - T \gtrunccu[n] j)$
  through the cocartesian morphism
  \[
    (\gtruncf[n]l -,\gtruncf[n] l - T \gtrunccu[n] l)\zbox.
  \]
  Thus, they are equal. The conclusion follows from the functoriality of~$\EM$.
\end{proof}

The finitary assumption on~$T$ enables the existence of a left adjoint to
truncation functors:
\begin{prop}
  \label{prop:algtruncf-ra}
  Given~$k,l \in \N_n\cup\set n$ with~$k < l$, the functor~$\algtruncf[l] k -$
  is finitary and admits a left adjoint.
\end{prop}
\begin{proof}
  By the adjunction $\mndinc \dashv \EM$, we have a commutative diagram
  \[
    \begin{tikzcd}
      {\Alg_l}
      \ar[r,"\algtruncf k -"]
      \ar[d,"\fgfalgf_l"']
      &
      {\Alg_k}
      \ar[d,"\fgfalgf_k"]
      \\
      \nGlob l
      \ar[r,"\gtruncf k -"']
      &
      \nGlob k
    \end{tikzcd}
    \zbox.
  \]
  The functor~$\algtruncf[l] k {(-)}$ is finitary since, by
  \Propr{alg-cat-loc-fin-pres},~$\fgfalgf_k$ creates directed colimits and the
  functor
  \[
    {\fgfalgf_k \algtruncf[l] k {(-)} = \gtruncf k {(-)}
      \fgfalgf_{l}}
  \]
  preserves directed colimits. Moreover,~$\smash{\algtruncf[l] k {(-)}}$
  preserves limits since~$\fgfalgf_k$ creates limits and the
  functor~$\smash{\fgfalgf_k \algtruncf[l] k {(-)} = \gtruncf k {(-)}
    \fgfalgf_{l}}$ preserves limits (both~$\smash{\gtruncf[l] k {(-)}}$
  and~$\smash{\fgfalgf_{l}}$ are right adjoints). Then, by
  \Propr{criterion-loc-pres-cat-adjoints}, the functor~$\smash{\algtruncf[l] k {(-)}}$
  admits a left adjoint.
\end{proof}
\noindent Given~$k,l \in \N_n\cup\set n$ with~$k < l$, we
\glossary(Algbincf){$\algincf[k] {l} {(-)}$, $\incf l A$}{the inclusion functor
  for globular algebras}write
\[
  \algincf[k] {l} {(-)} \co \Alg_k \to \Alg_{l}
\]
for the left adjoint to~$\algtruncf[l] k -$, or even~$\algincf l {(-)}$ when
there is no ambiguity on~$k$. The image of~$(X,h)$ in~$\Alg_k$ by~$\algincf
{l}{(-)}$ is called the
\index{inclusion!of a globular algebra}\emph{$l$\nbd-inclusion} of~$(X,h)$ and
we denote it~$\incf l {(X,h)}$.

\subsection[\texorpdfstring{$\Alg_\omega$}{Alg-omega} as a limit]{$\bm{\Alg_\omega}$
  as a limit}
\label{text:algo-as-a-limit}

Let~$(T,\eta,\mu)$ be a finitary monad on~$\nGlob\omega$. The purpose of this
paragraph is to characterize~$\Alg_\omega$ as a limit on the categories~$\Alg_k$
for~$k \in \N$ using the truncation functors~$\algtruncf k {(-)}$.

\begin{prop}
  \label{prop:Tomega-limit-cone}
  The cone
  \[
    ((\nGlob\omega,T), (\gtruncf k-,\gtruncf k-T\gtrunccu k)_{k\in\N})
  \]
  is a limit cone in $\MND$ on the diagram
  \[
    \hss
    \begin{tikzcd}
      (\nGlob0,T^0)
      &
      (\nGlob1,T^1)
      \ar[l,"{\gtruncf 0-}"']
      &
      (\nGlob2,T^2)
      \ar[l,"{\gtruncf 1-}"']
      &
      (\nGlob3,T^3)
      \ar[l,"{\gtruncf 2-}"']
      &
      \cdots
      \ar[l,"{\gtruncf 3-}"']
    \end{tikzcd}
    \hss
  \]
  where we abbreviated by~$\gtruncf k -$ the monad functor
  $(\gtruncf k -,\gtruncf k -T^{k+1}\gtrunccu k)$ for $k \in \N$ for readability.
\end{prop}
\begin{proof}
  We already know that it is a cone by \Cref{lem:gtruncf-cocart-lift}. Let us
  prove that it is a limit cone in~$\MND$ (seen here as a $1$\category).

  Let $(\cC,S)$ be a monad and $((\Gamma^k,\gamma^k) \co (\cC,S) \to (\nGlob
  k,T^k))_{k\in\N}$ be a cone on the diagram of the statement. By forgetting the
  $2$\cells, we get a cone $(\Gamma^k \co \cC \to \nGlob
  k)_k$ and thus a functor $\Gamma \co \cC \to \nGlob\omega$. In order to get a
  monad functor, we still need to build a $2$\cell $\gamma\co T\Gamma \To \Gamma
  S$. Such a $2$\cell is the data of morphisms $\gamma_X \co T\Gamma X \to
  \Gamma S X$ for $X \in \cC$. Write $R^k$ for
  \[
    R^k = \gincf \omega {(-)} \gtruncf k {(-)} \co \nGlob
    \omega \to \nGlob \omega
  \]
  \newcommand\onecu[2][]{\operatorname{\mathrm j}^{\commatwo{#2}{#1}}}%
  and $\onecu k$ for the natural transformation
  \[
    \onecu k = \gincf \omega {(-)} \gtrunccu[k+1] k \gtruncf {k+1} {(-)} \co R^k
    \To R^{k+1}
    \zbox.
  \]
  Note that, for all~$Y \in \nGlob \omega$,~$(\gtrunccu[\omega]{k}_Y \co R^kY \to Y)_{k \in
    \N}$ is a colimit cocone in~$\nGlob \omega$ on the diagram
  \[
    \begin{tikzcd}
      R^0 Y
      \ar[r,"{\onecu {0}_{Y}}"]
      &
      R^1 Y
      \ar[r,"{\onecu {1}_{Y}}"]
      &
      \cdots
      \ar[r,"{\onecu {k-1}_{Y}}"]
      &
      R^k Y
      \ar[r,"{\onecu {k}_{Y}}"]
      &
      R^{k+1} Y
      \ar[r,"{\onecu {k+1}_{Y}}"]
      &
      \cdots
    \end{tikzcd}
  \]
  Since~$T$ is finitary,~$((T\gtrunccu[\omega]{k})_Y \co T R^k Y \to TY)_{k \in
    \N}$ is a colimit cocone on the diagram
  \[
    \hss
    \begin{tikzcd}[column sep=4em]
      TR^0Y
      \ar[r,"{(T\onecu{0})_{Y}}"]
      &
      TR^1Y
      \ar[r,"{(T\onecu {1})_{Y}}"]
      &
      \cdots
      \ar[r,"{(T\onecu{k-1})_{Y}}"]
      &
      TR^kY
      \ar[r,"{(T\onecu{k})_{Y}}"]
      &
      TR^{k+1} Y
      \ar[r,"{(T\onecu{k+1})_{Y}}"]
      &
      \cdots
    \end{tikzcd}
    \hss
  \] 
  In particular, the sought morphisms $\gamma_X$ are uniquely characterized by the
  cocone
  \[
    ((\gamma \circ (T\gtrunccu[\omega]{k}\Gamma))_X \co T R^k \Gamma X \to \Gamma S X)_{k \in
      \N}
  \]
  on the above diagram, where we put $Y = \Gamma X$. Using some rephrasing, we
  have a limit cone
  \[
    ( \precomp {(T\gtrunccu[\omega]{k}\Gamma)} \co \Hom(T \Gamma, \Gamma S) \to \Hom(T
    R^k \Gamma,\Gamma S))_{k \in \N}
  \]
  on the diagram
  \[
    \hss
    \begin{tikzcd}[column sep=2.5em]
      \Hom(TR^0\Gamma,\Gamma S)
      &
      \Hom(TR^1\Gamma,\Gamma S)
      \ar[l,"\precomp{(T\onecu {0}\Gamma)}"']
      &
      \Hom(TR^2\Gamma,\Gamma S)
      \ar[l,"\precomp{(T\onecu {1}\Gamma)}"']
      &
      \cdots
      \ar[l,"\precomp{(T\onecu {2}\Gamma)}"']
    \end{tikzcd}
    \hss
  \] 
  On the other hand, given a pair of functors $G,G' \co \cC \to \nGlob\omega$,
  every natural transformation $\alpha \co G \To G'$ is uniquely characterized
  by the projections $\gtruncf k - \alpha$ for $k \in \N$. In other words, we
  have a limit cone
  \[
    ( \gtruncf l - \co \Hom(G, G') \to \Hom(\gtruncf l - G,\gtruncf l - G')_{l \in \N}
  \]
  on the diagram
  \[
    \hss
    \begin{tikzcd}[column sep=2em]
      \Hom(\gtruncf 0 - G,\gtruncf 0 - G')
      &
      \Hom(\gtruncf 1 - G,\gtruncf 1 - G')
      \ar[l,"\gtruncf 0 -"']
      &
      \Hom(\gtruncf 2 - G,\gtruncf 2 - G')
      \ar[l,"\gtruncf 1 -"']
      &
      \cdots
      \ar[l,"\gtruncf 3 -"']
    \end{tikzcd}
    \hss
  \] 
  By combination of the two limit diagrams, we have that $\Hom(T \Gamma, \Gamma
  S)$ is limit cone on the diagram
  \[
    \hss
    \begin{tikzcd}[row sep=huge]
      &
      \vdots
      \ar[d,"{\gtruncf {l+1} -}"{description}]
      &
      \vdots
      \ar[d,"{\gtruncf {l+1} -}"{description}]
      & \\
      \cdots
      &
      \Hom(\gtruncf {l+1}- TR^k\Gamma, \gtruncf {l+1} - \Gamma S)
      \ar[l,"\precomp {(\onecu {k-1})}"']
      \ar[d,"{\gtruncf {l} -}"{description}]
      &
      \Hom(\gtruncf {l+1} - TR^{k+1}\Gamma, \gtruncf {l+1} - \Gamma S)
      \ar[l,"\precomp {(\onecu {k})}"']
      \ar[d,"{\gtruncf {l} -}"{description}]
      &
      \cdots
      \ar[l,"\precomp {(\onecu {k+1})}"']
      \\
      \cdots
      &
      \Hom(\gtruncf l - TR^k\Gamma, \gtruncf l - \Gamma S)
      \ar[d,"{\gtruncf {l-1} -}"{description}]
      \ar[l,"\precomp {(\onecu {k-1})}"']
      &
      \Hom(\gtruncf l - TR^{k+1}\Gamma, \gtruncf l - \Gamma S)
      \ar[l,"\precomp {(\onecu {k})}"']
      \ar[d,"{\gtruncf {l-1} -}"{description}]
      &
      \cdots
      \ar[l,"\precomp {(\onecu {k+1})}"']
      \\
      &
      \vdots
      &
      \vdots
    \end{tikzcd}
    \hss
  \]
  expressed on the category $\catop{(\N \times \N, \le \times \le)}$. By
  finality of the diagonal functor
  \[
    \catop{(\N,\le)} \to \catop{(\N\times\N,\le \times \le)}
    \zbox,
  \]
  $\Hom(T \Gamma, \Gamma S)$ is in fact a limit cone on the diagram
  \[
    \hss
    \begin{tikzcd}[column sep=small]
      \Hom(\gtruncf 0 - TR^0\Gamma,\gtruncf 0 - \Gamma S)
      &
      \Hom(\gtruncf 1 - TR^1\Gamma,\gtruncf 1 - \Gamma S)
      \ar[l]
      &
      \Hom(\gtruncf 2 - TR^2\Gamma,\gtruncf 2 - \Gamma S)
      \ar[l]
      &
      \cdots
      \ar[l]
    \end{tikzcd}
    \hss
  \]
  One can check that a compatible sequence for this diagram is precisely given
  by the $\gamma^k$ for $k \in \N$. Indeed, it is a consequence of the fact that
  $((\Gamma^k,\gamma^k))_{k\in \N}$ is a cone. Thus, we get $\gamma \co T\Gamma
  \To \Gamma S$ such that $(\gtruncf k - \gamma) \circ (\gtruncf k - T
  \gtrunccu[\omega]k\Gamma) = \gamma_k$ and $\gamma$ is uniquely defined by this
  property. Thus, we have the uniqueness of a factorizing $(\Gamma,\gamma)$ in
  $\MND$ for the cone we started from. We are left to show that
  $(\Gamma,\gamma)$ is indeed a monad functor for existence.

  Let's prove the first required equality, \ie $\gamma \circ (\eta^T \Gamma) =
  \Gamma \eta^S$. By an argument similar to the one we used above, it is enough
  to prove that
  \[
    (\gtruncf k - \gamma) \circ (\gtruncf k -\eta^T \Gamma) \circ (\gtruncf k -
    \gtrunccu[\omega]k \Gamma) = (\gtruncf k - \Gamma \eta^S) \circ (\gtruncf k -
    \gtrunccu[\omega]k \Gamma)
  \]
  for every $k \in \N$. We compute that
  \begin{align*}
    & 
      (\gtruncf k - \gamma) \circ (\gtruncf k -\eta^T \Gamma) \circ (\gtruncf k -
      \gtrunccu[\omega]k \Gamma)
    \\
    =&
       (\gtruncf k - \gamma) \circ (\gtruncf k -
       T\gtrunccu[\omega]k \Gamma) \circ (\gtruncf k -\eta^T R^k\Gamma)
       & \text{(by naturality)}
       \\
    =&
       \gamma^k \circ (\gtruncf k -\eta^T R^k\Gamma)
       \\
    =&
       \gamma^k \circ (\eta^{k}\Gamma^k)
       & \text{\makebox[5cm][r]{(since the unit of $\gincf \omega - \dashv \gtruncf k -$ is the identity)}}
       \\
    =&
       \Gamma^k \eta^S
       \\
    =&
       \gtruncf k -\Gamma \eta^S
       \\
    =&
       (\gtruncf k -\Gamma \eta^S) \circ (\gtruncf k -
      \gtrunccu[\omega]k \Gamma)
       & \text{(since $\gtruncf k -
          \gtrunccu[\omega]k$ is an identity)\zbox.}
  \end{align*}
  In order to show the other equality, that is,
  \begin{equation}
    \label{eq:compat-with-mult}
    \gamma \circ (\mu^T \Gamma) = (\Gamma \mu^S) \circ (\gamma S) \circ (T \gamma)
  \end{equation}
  we need to use twice the argument we used earlier to get that
  \[
    (\pi_k \co \Hom(TT\Gamma,\Gamma S) \to \Hom(\gtruncf k - T R^k T R^k\Gamma,\gtruncf k - \Gamma
    S))_{k \in \N}
  \]
  is a limit cone on the adequate diagram, where $\pi_k$ is the composite
  \[
    \hss
    \begin{tikzcd}
      \Hom(TT\Gamma,\Gamma S)
      \ar[r,"{\gtruncf k -}"]
      &
      \Hom(\gtruncf k -TT\Gamma,\gtruncf k -\Gamma S)
      \ar[r,"{\precomp{(\gtruncf k - T \gtrunccu k T \gtrunccu k \Gamma)}}"]
      &[2cm]
      \Hom(\gtruncf k - T R^k T R^k\Gamma,\gtruncf k - \Gamma S)\zbox.
    \end{tikzcd}
    \hss
  \]
  Then, using the existing equations and naturality, the reader can prove
  similarly as above that
  \[
    \hss
    (\gtruncf k -(\gamma \circ (\mu^T \Gamma))) \circ {(\gtruncf k - T
        \gtrunccu k T \gtrunccu k \Gamma)}
    = (\gtruncf k -((\Gamma \mu^S) \circ (\gamma S) \circ (T \gamma)))
    \circ
    {(\gtruncf k - T \gtrunccu k T \gtrunccu k \Gamma)}
    \hss
  \]
  so that~\eqref{eq:compat-with-mult} holds and~$(\Gamma,\gamma)$ is indeed a
  monad functor, which moreover factorizes uniquely the cone
  $(\Gamma^k,\gamma^k)_{k\in\N}$. Thus, $(\nGlob\omega,T)$ is indeed a limit
  cone in $\MND$ as stated.
\end{proof}

\begin{prop}
  \label{prop:alg-infty-limit}
  $(\algtruncf k {(-)} \co \Alg_\omega \to \Alg_k)_{k \in \N}$ is a (bi)limit cone in~$\CAT$
  on the diagram
  \[
    \begin{tikzcd}
      \Alg_0
      &
      \Alg_1
      \ar[l,"\algtruncf 0 {(-)}"']
      &
      \Alg_2
      \ar[l,"\algtruncf 1 {(-)}"']
      &
      \Alg_3
      \ar[l,"\algtruncf 2 {(-)}"']
      &
      \cdots
      \ar[l,"\algtruncf 3 {(-)}"']
    \end{tikzcd}
  \]
\end{prop}
\begin{proof}
  This is a consequence of \Cref{prop:Tomega-limit-cone} and that $\EM$ is a
  right adjoint.
\end{proof}

\begin{remark}
  \label{rem:alg-infty-2-limit}
  The above proof only show that $\Alg_\omega$ is a limit in the $1$-categorical
  sense, \ie only factorizes cones of functors by a functor (and not morphisms
  of cones as natural transformations, \emph{a priori}). But, by a generic
  argument~(stated for example as \cite[Lemma~7.5.1]{riehl2014categorical}), any
  limit in the $1$-categorical sense in $\CAT$ is also a strict $2$-limit, \ie
  factorizes morphisms between cones as well.
\end{remark}

\begin{remark}
  If we restrain our interest to \emph{weakly truncable monads} (defined in
  \Cref{text:truncable-monads}), another proof is possible which dispenses with
  the specific and technical argument of \Cref{prop:Tomega-limit-cone}. Indeed,
  writing $\MNDpsd$ for the sub-$2$-category of $\MND$ consisting of monads,
  \emph{pseudo} monad functors and monad transformations, where a monad functor
  $(F,\phi)$ is said \emph{pseudo} when~$\phi$ is an isomorphism, we have that
  the \emph{underlying functor} $\mndund \co \MNDpsd \to \CAT$, mapping a monad
  $(\cC,T)$ to its underlying category~$\cC$, creates bilimits, and the latter
  are moreover preserved by the inclusion $\MNDpsd \hookrightarrow \MND$ (this
  is left as an exercise for now). Thus, a bilimit version of
  \Cref{prop:Tomega-limit-cone} can be deduced from the fact that $\nGlob\omega$
  is a bilimit of the $\nGlob k$'s for $k \in \N$. This argument does not apply
  for $\MND$ as a whole, and it is really the finitary hypothesis on~$T$ which
  makes \Cref{prop:Tomega-limit-cone} work.
\end{remark}

\subsection{A criterion for globular algebras}
\label{text:criterion-for-globular-algebras}


Usually, a specific notion of higher category and the associated truncation and
inclusion functors are not directly derived from a monad. Instead, we often
manipulate higher categories that are defined, in each dimension~$k \in \N$, as
structures with operations satisfying some equations, and the truncation and
inclusion functors are defined by hand. Such equational definitions surely
induce monads on $k$\globular sets, but it is not immediate that the monad in
dimension~$l_1$ is obtained by truncating the monad in dimension~$l_2$ for $l_1
< l_2$, as was done earlier, nor is the equivalence between the boilerplate
definitions of truncation and inclusion functors and the ones defined earlier in
this section. Verifying the equivalences of these definitions is required in
order to use general constructions for globular algebras, like the ones of the
next section. But, without a generic argument, the verification can be tedious
since it involves, among others, an explicit description of the different
monads. In this section, we give a criterion, in the form of
\Cref{thm:charact-globular-algebras}, to recognize that categories and functors
between them are equivalent to categories of globular algebras and truncation
functors derived from some monad on globular sets.

\smallpar We first prove a technical lemma:
\begin{lem}
  \label{lem:mnd-ra-square}
  Given a commutative diagram of functors
  \[
    \begin{tikzcd}
      C
      \ar[r,"U"]
      \ar[d,"\mcal T"']
      &
      D
      \ar[d,"\mcal T'"]
      \\
      \bar C
      \ar[r,"\bar U"']
      &
      \bar D
    \end{tikzcd}
  \]
  which are part of adjunctions $F \dashv U$, $\bar F \dashv \bar U$, $\mcal I
  \dashv \mcal T$, $\mcal I' \dashv \mcal T'$, we have a diagram in $\MND$
  \[
    \begin{tikzcd}[column sep=large]
      (C,\catunit{})
      \ar[d,"{\mndinc(\mcal T)}"']
      \ar[rrr,"{(U,U\eps^{F,U})}"]
      &&&
      (D,R)
      \ar[d,"{(\mcal T',\cT'R\eps^{\cI',\cT'})}"]
      \\
      (\bar C,\catunit{})
      \ar[r,"{(\bar U,\bar U\eps^{\bar F,\bar U})}"']
      &
      (\bar D,\bar T)
      &(\bar D,\bar S)
      \ar[l,"{(\catunit{},\bar U\eta^{\cI,\cT}\bar F)}"]
      &(\bar D,S)
      \ar[l,"\sim"]
    \end{tikzcd}
  \]
  where $R$, $S$, $\bar S$ and $\bar T$ are the monads respectively derived from
  the adjunctions $F \dashv U$, $F\mcal I' \dashv \mcal T' U$, $\mcal I \bar F
  \dashv \bar U \mcal T$ and $\bar F \dashv \bar U$, and $(\bar D,S) \to (\bar
  D,\bar S)$ is the isomorphism of monads induced by $\cT'U = \bar U\cT$, and
  where we wrote $\eta^{L,R}$ and $\eps^{L,R}$ for the unit and counit of each
  adjunction $L \dashv R$. Moreover, if the unit of the adjunction $\mcal I
  \dashv \mcal T$ is an isomorphism (or, equivalently, $\mcal I$ is fully
  faithful), then the morphism $(\bar D,\bar S) \to (\bar D,\bar T)$ is an
  isomorphism.
\end{lem}
\begin{proof}
  By \Cref{prop:ex-algra,prop:mnd-morph-eta} already introduced, the morphisms
  shown on the diagram are indeed monad functors. We are left to check that the
  diagram commutes. More precisely, we need to show that pasting of $2$\cells
  \[
    \begin{tikzcd}[sep=huge]
      C
      \ar[r,"U"]
      \ar[d,"{\catunit{}}"']
      \cphar[rd,"="]
      &
      D
      \ar[r,"\cT'"]
      \ar[d,"R"'{description}]
      \cphar[rd,"\Downarrow\cT'R\eps^{\cI',\cT'}"]
      &
      \bar D
      \ar[r,"{\catunit{}}"]
      \ar[d,"S"'{description}]
      \cphar[rd,"\Downarrow\bar U\cT\finv\theta"]
      &
      \bar D
      \ar[r,"{\catunit{}}"]
      \ar[d,"\bar S"'{description}]
      \cphar[rd,"\Downarrow\bar U\eta^{\cI,\cT}\bar F"]
      &
      \bar D
      \ar[d,"\bar T"]
      \\
      C
      \ar[r,"U"']
      &
      D
      \ar[r,"\cT'"']
      &
      \bar D
      \ar[r,"{\catunit{}}"']
      &
      \bar D
      \ar[r,"{\catunit{}}"']
      &
      \bar D
    \end{tikzcd}
  \]
  is equal to
  \[
    \begin{tikzcd}[sep=huge]
      C
      \ar[r,"\cT"]
      \ar[d,"{\catunit{}}"']
      \cphar[rd,"="]
      &
      \bar C
      \ar[r,"\bar U"]
      \ar[d,"{\catunit{}}"'{description}]
      \cphar[rd,"\Downarrow\bar U \eps^{\bar F,\bar U}"]
      &
      \bar D
      \ar[d,"{\bar T}"]
      \\
      C
      \ar[r,"\cT"']
      &
      \bar C
      \ar[r,"\bar U"']
      &
      \bar D
    \end{tikzcd}
  \]
  where $\theta \co F\cI'\To \cI \bar F$ is the canonical isomorphism between
  the two functors which are left adjoint to $\bar U \cT = \cT' U$, defined by
  \[
    \theta
    \qqeq
    \satex{mnd-ra-square-theta}
    \zbox{\qquad.}
  \]
  We check that the wanted equality holds using string diagrams:
  \[
    \satexnoscale{lem-mnd-morph-1}
    =
    \satexnoscale{lem-mnd-morph-2}
    =
    \satexnoscale{lem-mnd-morph-3}
    \zbox{\qquad.}
    \qedhere
  \]
\end{proof}
The criterion for recognizing categories and functors as categories of globular
algebras and truncation functors between them is the following:
\begin{theo}
  \label{thm:charact-globular-algebras}
  Let $(T,\eta,\mu)$ be a finitary monad on $\nGlob\omega$, and
  \[
    (C_k)_{k \in \N}
    \qtand
    C_\omega
  \]
  be categories, and
  \[
    (U_k \co C_k \to \nGlob k)_{k
      \in \N}
    \qtand
    U_\omega \co C_\omega \to \nGlob \omega 
  \]
  be monadic functors, and
  \[
    (\mcal T^{k+1}_k \co C_{k+1} \to C_k)_{k \in \N}
    \qtand
    (\mcal T^\omega_k \co C_\omega \to C_k)_{k \in \N}
  \]
  be right adjoint functors with fully faithful left adjoints such that
  $\cT^{k+1}_k\cT^\omega_{k+1} = \cT^\omega_k$ and
  \[
    \begin{tikzcd}
      C_\omega
      \ar[r,"U_\omega"]
      \ar[d,"\mcal T^{\omega}_k"']
      &
      \nGlob\omega
      \ar[d,"{\gtruncf[\omega]k-}"]
      \\
      C_k
      \ar[r,"U_k"']
      &
      \nGlob k
    \end{tikzcd}
  \]
  commute for every $k \in \N$. Then, there exist equivalences of categories
  \[
    H_\omega \co C_\omega \to \Alg_\omega
    \qtand
    (H_k \co  C_k \to \Alg_k)_{k \in \N}
  \]
  such that
  \begin{equation}
    \label{eq:thm:charact-globular-algebras:obtained-squares}
    \begin{tikzcd}
      C_\omega
      \ar[r,"H_\omega"]
      \ar[d,"\mcal T^{\omega}_k"']
      &
      \Alg_\omega
      \ar[d,"{\algtruncf[\omega]k-}"]
      \\
      C_k
      \ar[r,"H_k"']
      &
      \Alg_k
    \end{tikzcd}
    \qtand
    \begin{tikzcd}
      C_{k+1}
      \ar[r,"H_{k+1}"]
      \ar[d,"\mcal T^{k+1}_k"']
      &
      \Alg_{k+1}
      \ar[d,"{\algtruncf[k+1]k-}"]
      \\
      C_k
      \ar[r,"H_k"']
      &
      \Alg_k
    \end{tikzcd}
  \end{equation}
  commute for every $k \in \N$.
\end{theo}
\begin{rem}
  The statement of the above theorem can easily be adapted when working on a
  monad $T$ on $\nGlob n$ for some $n \in \N$ by requiring instead that
  \[
    \begin{tikzcd}
      C_{k+1}
      \ar[r,"U_{k+1}"]
      \ar[d,"\mcal T^{k+1}_k"']
      &
      \nGlob{k+1}
      \ar[d,"{\gtruncf[k+1]k-}"]
      \\
      C_k
      \ar[r,"U_k"']
      &
      \nGlob k
    \end{tikzcd}
  \]
  commutes for every $k < n$, and then one only gets the right commuting square
  of~\eqref{eq:thm:charact-globular-algebras:obtained-squares}.
\end{rem}
\begin{proof}
  Let $k \in \N$. Taking $C_\omega$, $C_k$, $\nGlob\omega$ and $\nGlob k$
  respectively for $C$, $\bar C$, $D$ and $\bar D$, and $U_\omega$, $U_k$,
  $\cT^\omega_k$ and $\gtruncf[\omega]k-$ respectively for $U$, $\bar U$, $\cT$
  and $\cT'$ in \Cref{lem:mnd-ra-square}, and using the adjunction $\mndinc
  \dashv \EM$, we get a commutative square
  \[
    \begin{tikzcd}[column sep=large]
      C_\omega
      \ar[d,"{\mcal T^\omega_k}"']
      \ar[rrr,"H_\omega"]
      &&&
      \Alg_\omega
      \ar[d,"{\algtruncf[\omega]k-}"]
      \\
      C_k
      \ar[r,"\bar H_k"']
      &
      \nGlob k^{\bar T}
      &
      \nGlob k^{\bar S}
      \ar[l,"\sim"]
      &
      \Alg_k
      \ar[l,"\sim"]
    \end{tikzcd}
  \]
  where $\bar S$ and $\bar T$ are the monads defined in the lemma. Remember that
  both $\Alg_k$ and $\algtruncf[\omega] k -$ were defined using
  \Cref{prop:ex-algra} so that the arrow on the right of the above diagram is
  indeed $\algtruncf[\omega] k -$. Also, $H_\omega$ and $\bar H_k$ are
  equivalences since $U_\omega$ and $U_k$ were supposed monadic. Finally, by the
  final part of \Cref{lem:mnd-ra-square} and the hypothesis, the middle arrow in
  the bottom line is an isomorphism. Thus, by defining $H_k$ to be the
  composition of the three morphisms on the bottom line (after inverting the
  second and third ones), we get the commutative diagram on the left
  of~\eqref{eq:thm:charact-globular-algebras:obtained-squares}, and we are left
  to get the one on the right. Writing $\cI^{k+1}_\omega$ for a left adjoint to
  $\cT^\omega_{k+1}$, since the unit $\catunit{} \To
  \cT^\omega_{k+1}\cI^{k+1}_\omega$ is an isomorphism, it is enough to check
  that
  \[
    \algtruncf k -H_{k+1}\cT^\omega_{k+1}\cI^{k+1}_\omega
    =
    H_k\cT^{k+1}_k\cT^\omega_{k+1}\cI^{k+1}_\omega\zbox.
  \]
  But
  \begin{align*}
    \algtruncf k -H_{k+1}\cT^\omega_{k+1}\cI^{k+1}_\omega
    &=
      \algtruncf k -\algtruncf[\omega] {k+1}-H_\omega\cI^{k+1}_\omega
      \\
    &=
      \algtruncf[\omega] k -H_\omega\cI^{k+1}_\omega
    \\
    &=
      H_k\cT^\omega_k\cI^{k+1}_\omega
    \\
    &=
      H_k\cT^{k+1}_k\cT^\omega_{k+1}\cI^{k+1}_\omega
    &      
  \end{align*}
  which concludes the proof.
\end{proof}

\subsection{Truncable globular monads}
\label{text:truncable-monads}
\begingroup \newcommand\Tkpo{T^{k+1}}%

The general setting of higher category theories as monads over globular sets
allows defining theories with unusual operations, like compositions of
$l$\cells that produce unrelated $l'$\cells for some~$l' < l$ (\cf
\Cref{ex:weird-cat-globular-algebra}). Anticipating the next section, such
theories are badly behaved when it comes to freely adding new
$(k{+}1)$\generators to $k$\categories, since the underlying $k$\categories will
not be preserved in the process. In order not to allow such monads, we recall
from~\cite{batanin1998computads} the notion of \emph{truncable monad} which
forbids those problematic operations and still includes most usual theories for
higher categories: those are the monads which ``commute with truncation'' in a
suitable sense. As we will see in the next\link section, the $k$\categories of
these theories are preserved when freely adding $(k{+}1)$\generators.

Let~$n \in \N \cup \set\omega$ and~$(T,\eta,\mu)$ be a finitary monad on~$\nGlob
n$. For~$k \in \N_n$, remember that we used \Cref{prop:ex-algra} to define both
the monad $(T^k,\eta^k,\mu^k)$ and the cocartesian morphism
\[
  (\gtruncf[n]k-,\gtruncf k {(-)} T \gtrunccu[n] k) 
  \co
  (\nGlob n,T)
  \to
  (\nGlob k,T^k)
  \zbox.
\]
In the following, we write $\montruncn k$ for the natural transformation
$\gtruncf k {(-)} T \gtrunccu[n] k$. The monad~$T$ is said
\emph{weakly truncable} when~$\montruncn k$ is an isomorphism for each~$k < n$; it is \emph{truncable} when, for each~$k < n$, $\montruncn k$ is the
natural identity transformation (so that ${T_k\gtruncf k {(-)} = \gtruncf k
  {(-)}T}$).
\begin{example}
  The monad~$(T,\eta,\mu)$ of categories on~$\nGlob 1$ defined in
  \Exr{cat-globular-algebra} is weakly truncable. By choosing adequately the
  left adjoint~$\nGlob 1 \to \Cat$ that defines~$T$, we can even suppose
  that~$T$ is truncable. More generally, we will see in \Cref{text:sc-as-globular-algebras}
  that the monad of strict \ocats is weakly truncable, and even truncable up to
  an isomorphism of monads.
\end{example}
\begin{example}
  \label{ex:weird-monad-not-truncable}
  The monad~$(T,\eta,\mu)$ of weird $2$\categories on~$\nGlob 2$ defined in
  \Exr{weird-cat-globular-algebra} is not truncable since, for~$X \in \nGlob
  2$, we have
  \[
    (TX)_0 \cong X_0 \sqcup (X_2 \times X_2)
    \qtand
    (T^0(\restrict 0 X))_0 \cong X_0\zbox.
  \]
\end{example}
\noindent The following property tells that we can always adapt a weakly
truncable monad to a truncable monad:
\begin{prop}
  If~$(T,\eta,\mu)$ is weakly truncable, then it is isomorphic to a monad
  which is truncable.
\end{prop}
\begin{proof}
  We define a truncable monad~$(\bar T,\bar\eta,\bar\mu)$ on~$\nGlob n$ and an
  isomorphism~${\phi\co T\to \bar T}$ from their trunctations
  \[
    \gtruncf k {(-)}\bar T \qtand \restrict k \phi\co \gtruncf k {(-)}T\to \gtruncf k {(-)}\bar T
  \]
  and we define those using an induction on~$k$ for~$k \in \N_n$. In dimension~$0$, we put
  \[
    \gtruncf 0 {(-)} \bar T = T^0\gtruncf 0
    {(-)}
    \qtand
    \phi_0 = \finv {(\montruncn 0)}
  \]
  Then, given~$k \in \N_n$ and a $(k{+}1)$\globular set~$X$, we define~$\restrict {k+1} {(\bar TX)}$ as the $(k{+}1)$\globular set~$Y$ where
  \[
    \restrict k Y = \restrict {k} {(\bar TX)}
    \qtand
    Y_{k+1} = (T^{k+1}(\restrict {k+1} X))_{k+1}
  \]
  and the operation~$\csrctgt\eps_k \co Y_{k+1} \to Y_{k}$ is defined as the
  composite
  \[
    Y_{k+1} = (T^{k+1}(\restrict {k+1} X))_{k+1} \xto{\csrctgt\eps_{k}}
    (T^{k+1}(\restrict {k+1} X))_k \xto{(\montruncn {k+1})_k} (TX)_k \xto{(\phi_k)_k} (\bar TX)_k
  \]
  for~$\eps \in \set{-,+}$. Our definition extends canonically to a functor
  \[
    \gtruncf {k+1}
    {(-)} \bar T\co \nGlob n \to \nGlob {k+1}\zbox.
  \]
  We also extend~$\restrict k \phi$
  on dimension~$k+1$ by putting, for~$X \in \nGlob n$,
  \[
    (\phi_{X})_{k+1} = (\finv {(\montruncn[X] {k+1})})_{k+1} \co (TX)_{k+1} \to
    (\bar TX)_{k+1}
  \]
  So we defined~$\bar T\co \nGlob n \to \nGlob n$ together with an isomorphism~$\phi \co T\to \bar T$. Finally, we put
  \[
    \bar\eta = \phi \circ \eta \qtand \bar\mu = \phi \circ \mu \circ (\finv\phi\finv\phi)
  \]
  so that~$(\bar T,\bar\eta,\bar\mu)$ is a monad and $(\catunit,\finv\phi) \co
  (\nGlob n,T) \to (\nGlob n,\bar T)$ a monad isomorphism. By the definition
  of~$\bar T$, we easily verify that~$(\bar T,\bar\eta,\bar\mu)$ is truncable.
\end{proof}
\noindent The truncability with respect to dimension~$n$ implies the
truncability with respect to dimension~$l$ for any~$l < n$:
\begin{lem}
  \label{lem:truncable-other-dims}
  If $T$ is weakly truncable (\resp truncable), then, for every $k,l \in \N$
  with $k < l < n$, $\gtruncf k - T^l \gtrunccu[l]k$ is an isomorphism (\resp an
  identity).
\end{lem}
\begin{proof}
  We have $\gtrunccu[n]k = \gtrunccu[l]k \circ (\gincf[k]l- \gtrunccu[l]k
  \gtruncf[l]k-)$, so that
  \[
    \gtruncf[n]k- T \gtrunccu[n]k = (\gtruncf[l]k-\gtruncf[n]l-T\gtrunccu[n]l)
    \circ (\gtruncf[l]k-T^l\gtrunccu[l]k\gtruncf[n]l-)
    \zbox.
  \]
  By the 2-out-of-3 property for isomorphisms, we get that
  $\gtruncf[l]k-T^l\gtrunccu[l]k\gtruncf[n]l-$ is an isomorphism. By
  precomposing with $\gincf[l]n-$, we get that $\gtruncf[l]k-T^l\gtrunccu[l]k$
  is an isomorphism. When $T$ is truncable, this isomorphism is an equality.
\end{proof}
\noindent When~$T$ is truncable, the~$T^k$,~$\eta^k$ and~$\mu^k$ can be related through
the equations given by the following lemma:
\begin{lem}
  \label{lem:truncable-T-eta-mu}
  If~$T$ is truncable, then, for $k,l \in \N_n \cup \set n$ with~$k < l$, we have
  \[
    T^k \gtruncf[l] k - = \gtruncf[l] k - T^l
    \qtand
    \gtruncf[l] k - \eta^l = \eta^k \gtruncf[l] k -
    \qtand
    \gtruncf[l] k - \mu^l = \mu^k \gtruncf[l] k -
  \]
\end{lem}
\begin{proof}
  \begingroup
  \allowdisplaybreaks
  By \Cref{lem:truncable-other-dims}, we have that $\gtruncf k - T^l
  \gtrunccu[l]k$ is an identity. In particular,
  \[
    T^k \gtruncf[l] k -
    =
    \gtruncf[l] k - T^l
    \zbox.
  \]
  Moreover, from the equations satisfied by the monad functor
  $(\gtruncf[l]k-,\gtruncf k - T^l \gtrunccu[l]k)$, we deduce that
  \[
    \gtruncf[l] k - \eta^l = \eta^k \gtruncf[l] k -
    \qtand
    \gtruncf[l] k - \mu^l = \mu^k \gtruncf[l] k -
    \zbox.
    \qedhere
  \]
  \endgroup
\end{proof}
\noindent We now prove several properties of truncable monads regarding
truncation of algebras. First, the truncation of algebras has now a simpler
definition:
\begin{prop}
  \label{prop:algtruncf-truncable}
  If\ \ $T$ is truncable, then given~$k,l\in \N_n \cup \set n$ such that~$k
  < l$, and an $l$\algebra~${(X,h) \in \Alg_{l}}$, we have~$\restrict k {(X,h)}
  = (\restrict k X,\restrict k h)$.
\end{prop}
\begin{proof}
  Indeed, since~$T$ is truncable, we have
  \begin{align*}
    ({\gtruncf[l] k {(-)}T^{l}\gtrunccu[l] {k}})_X
    = 
    ({T^{k}\gtruncf[l] k {(-)}\gtrunccu[l] {k}})_X
    = \unit {T^k X}
  \end{align*}
  so that~$\restrict k {(X,h)} = (\restrict k X,\restrict k h)$.
\end{proof}
\noindent Moreover, the operation of truncation of algebras is now a left
adjoint:
\begin{prop}
  \label{prop:truncable-algtruncf-la}
  If\ \ $T$ is truncable, then, given~$k,l \in \N_n \cup \set n$ with~$k < l$,
  the functor
  \[
    \algtruncf[l] k -\co \Alg_{l} \to \Alg_k
  \]
  is a left adjoint. In particular, it preserves colimits.
\end{prop}
\begin{proof}
  By \Cref{lem:truncable-other-dims}, $\gtruncf k - T^l \gtrunccu[l]k$ is an
  identity and in particular an isomorphism. Thus, by \Cref{prop:und-adjl-lift},
  since we have an adjunction $\gtruncf[l]k- \dashv \gincfill[k]l-$, this
  adjunction lifts canonically through $\mndund$ to an adjunction
  \[
    (\gtruncf[l]k-,\gtruncf k - T^l \gtrunccu[l]k) \dashv (\gincfill[k]l-,\rho)
  \]
  for some $\rho$ given by \Cref{prop:und-adjl-lift}. By applying $\EM$,
  we get an adjunction $\algtruncf[l]k- \dashv \algincfill[k]l-$ where
  $\algincfill[k]l- = \EM(\gincfill[k]l-,\rho)$.
\end{proof}

\subsection{Characterization of truncable monads}
\label{text:charact-truncable}

Earlier, we introduced \Cref{thm:charact-globular-algebras} that allows
recognizing that some categories and functors between them are equivalent to the
categories of globular algebras and the associated truncation functors derived
from a monad~$T$ on globular sets, without having to explicitly describe this
monad. But, by the current definition of truncability, in order to show that the
monad~$T$ is truncable, a direct proof would require to show that the natural
transformations~$\gtruncf l - T \gtrunccu[n] l$ are isomorphisms, so that a
description of~$T$ is still needed. Below, we introduce a characterization of
the truncability of~$T$ that does not rely on such tedious description.


\smallpar We start by proving the following lemma, relating the
functors~$\freealgf_k$:
\begin{lem}
  \label{lem:freealgfk-eq}
  Let~$n \in \N \cup \set \omega$ and~$(T,\eta,\mu)$ be a finitary monad
  on~$\nGlob n$. Given~$k \in \N$ such that~$k < n$, we have
  \[
    \algtruncf k - \freealgf_n \gincf[k] n - = \freealgf_k\zbox.
  \]
\end{lem}
\begin{proof}
  By the adjunction $\mndinc \dashv \EM$, it amounts to prove that the two
  pastings of $2$\cells
  \[
    \begin{tikzcd}[sep=large]
      \nGlob k 
      \ar[r,"{\gincf n -}"]
      \ar[d,"\catunit{}"']
      \cphar[rd,"="]
      &
      \nGlob n
      \ar[r,"T"]
      \ar[d,"\catunit{}"{description}]
      \cphar[rd,"\Downarrow\mu"]
      &
      \nGlob n
      \ar[r,"{\gtruncf k -}"]
      \ar[d,"T"{description}]
      \cphar[rd,"\Downarrow \gtruncf k - T \gtrunccu k"]
      &
      \nGlob k
      \ar[d,"T^k"]
      \\
      \nGlob k 
      \ar[r,"{\gincf n -}"']
      &
      \nGlob n
      \ar[r,"T"']
      &
      \nGlob n
      \ar[r,"{\gtruncf k -}"']
      &
      \nGlob k
    \end{tikzcd}
  \]
  and
  \[
    \begin{tikzcd}[sep=large]
      \nGlob k 
      \ar[r,"T^k"]
      \ar[d,"\catunit{}"']
      \cphar[rd,"\Downarrow\mu^k"]
      &
      \nGlob k
      \ar[d,"T^k"]
      \\
      \nGlob k 
      \ar[r,"T^k"']
      &
      \nGlob k
    \end{tikzcd}
  \]
  are equal. But this directly follows from the definition of $\mu^k$. So the
  wanted equality holds.
\end{proof}
\medskip\noindent Now, we prove that truncable monads can be characterized
through the associated globular algebras:
\begin{prop}
  \label{lem:truncability-through-algeba-cats}
  Let~$n \in \N \cup \set \omega$ and~$(T,\eta,\mu)$ be a finitary monad
  on~$\nGlob n$. Then, the monad~$(T,\eta,\mu)$ is weakly truncable (\resp
  truncable) if and only if, for~$k \in \N_{n-1}$, the natural transformation
  \[
    \algtruncf k - \freealgf_n \gtrunccu[n] k \co \freealgf_k \gtruncf k- \To
    \algtruncf k- \freealgf_n
  \]
  is an isomorphism (\resp an identity).
\end{prop}

\begin{proof}
  Note that the domain of~$\algtruncf k - \freealgf_n \gtrunccu[n] k$ is the
  claimed one by~\Cref{lem:freealgfk-eq}. Now, for~$k \in \N_{n-1}$, we have that
  \begin{align*}
    \fgfalgf_k\algtruncf k - \freealgf_n \gtrunccu k
    &= \gtruncf k - \fgfalgf_n \freealgf_n \gtrunccu k \\
    &= \gtruncf k - T \gtrunccu k\zbox.
  \end{align*}
  The proposition follows from the fact that~$\fgfalgf_k$ reflects isomorphisms
  (\resp identities).
\end{proof}
\smallpar We will need the following folklore property about adjunctions:
\begin{prop}
  \label{prop:isom-btw-adj}
  Let
  \[
    L \dashv R \co C \to D
    \qtand
    L' \dashv R' \co C \to D
  \]
  be two adjunctions
  with respective unit-counit pairs $(\gamma,\eps)$ and $(\gamma',\eps')$,
  and
  \[
    \theta \co L \To L'
    \qtand
    \bar\theta \co R' \To R
  \]
  be two natural transformations such that
  $\theta = (\eps L') \circ (L
  \bar\theta L') \circ (L \gamma')$,
  \ie graphically:
  \[
    \satex{isom-btw-adj-stmt-l}
    \qqeq
    \satex{isom-btw-adj-stmt-r}
    \pbox{\qquad.}
  \]
  Then, $\theta$ is an isomorphism if and only if $\bar\theta$ is an isomorphism.
\end{prop}
\begin{proof}
  An inverse for $\theta$ can be straight-forwardly constructed from an inverse
  for $\bar\theta$ using string diagrams, and symmetrically.
\end{proof}
\noindent Given~$k \in\N$, we \glossary(jk){$\gincfillu[n] k$}{the unit of the
  adjunction $\gtruncf[n] k - \dashv \gincfill[k]n -$}write
\[
  \gincfillu[n] k \co \unit {\nGlob n} \To \gincfill[k]n -\gtruncf[n] k -
\]
or simply~$\gincfillu k$, for the unit of the adjunction~$\gtruncf[n] k - \dashv
\gincfill[k]n - \co \nGlob k \to \nGlob n$. We can now introduce the criterion
for showing the truncability of monads through their globular algebras:
\begin{theo}
  \label{thm:truncable-charact}
  Let~$n \in \N \cup \set \omega$ and~$(T,\eta,\mu)$ be a finitary monad
  on~$\nGlob n$. The monad~$(T,\eta,\mu)$ is weakly truncable if and only if,
  for~$k \in \N_{n-1}$, the functor~$\algtruncf[n] k-$ has a right adjoint, that
  we \glossary(Algcfill){$\algincfill[k]n-$, $\incfill n A$}{the right adjoint to the truncation
  functor~$\algtruncf[n]k-$}write~$\smash{\algincfill[k]n-}$, which satisfies that~$\smash{\gincfillu k
    \fgfalgf_n \algincfill[k]n-}$ is an isomorphism.
\end{theo}
\begin{proof}
  By \Propr{truncable-algtruncf-la}, if~$T$ is weakly
  truncable,~$\smash{\algtruncf[n] k-}$ has a right adjoint, so we can assume
  that this adjoint exists and denote it by~$\smash{\algincfill[k]n-}$. Then,
  the morphism~$\smash{\algtruncf k - \freealgf_n \gtrunccu k}$, pictured by
  \[
    \begin{tikzpicture}[xscale=1.3,mycircle/.style={circle,minimum size=0.7cm}]
      \path
      (0,0) node (f5) {$\algtruncf k -$}
      ++(1,0) node (f6) {$\freealgf_n$}
      ++(1,0) node (f7) {$\gincf n -$}
      ++(1,0) node (f8) {$\gtruncf k -$};
      \draw (f5) -- +(0,-1.5);
      \draw (f6) -- +(0,-1.5);
      \drawcounit{$\gtrunccu k$}{1}{f7}{f8}
    \end{tikzpicture}
  \]
  is a natural transformation between two composites of left adjoints. Then, by
  deriving the units and counits of these composite adjunctions (as in
  \cite[IV.\S8 Theorem~1]{mac2013categories} for example), and using (the dual
  of) \Cref{prop:isom-btw-adj}, the above natural transformation is an
  isomorphism if and only if the natural transformation depicted by the string diagram
  
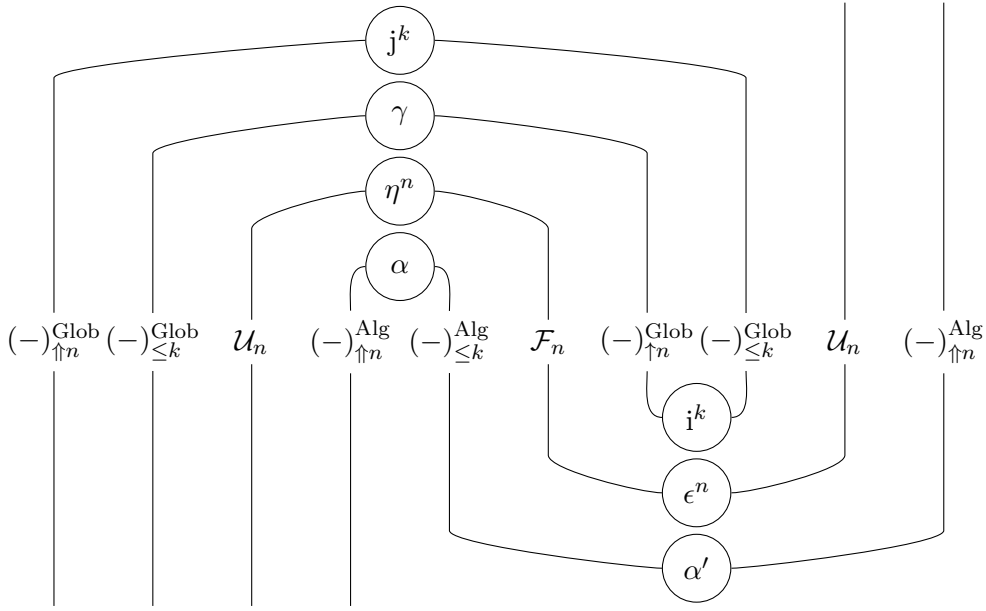
\[
    \begin{tikzpicture}[xscale=1.3,mycircle/.style={circle,minimum size=0.9cm}]
      \path (0,0) node (f1) {$\gincfill n-$}
      ++(1,0) node (f2) {$\gtruncf k -$}
      ++(1,0) node (f3) {$\fgfalgf_n$}
      ++(1,0) node (f4) {$\algincfill n -$}
      ++(1,0) node (f5) {$\algtruncf k -$}
      ++(1,0) node (f6) {$\freealgf_n$}
      ++(1,0) node (f7) {$\gincf n -$}
      ++(1,0) node (f8) {$\gtruncf k -$}
      ++(1,0) node (f9) {$\fgfalgf_n$}
      ++(1,0) node (f10) {$\algincfill n -$};
      \coordinate (middleleft) at ($(f4)!0.5!(f5)$);
      \node[draw,mycircle,radius=3] (AL) at ($(middleleft) + (0,1)$) {$\alpha$};
      \path[draw] (AL.west) to[out=180,in=90,looseness=0.8] (f4.north);
      \path[draw] (AL.east) to[out=0,in=90,looseness=0.8] (f5.north);
      \draw (f4) -- +(0,-3.5);
      \foreach \vnode/\vlabel/\vup/\vleft/\vright in {%
        BL/$\eta^n$/2/f3/f6,
        CL/$\gamma$/3/f2/f7,
        DL/$\gincfillu k$/4/f1/f8%
      }{
        \drawunit{\vlabel}\vup\vleft\vright
        \draw (\vleft) -- +(0,-3.5);
      };
      \coordinate (middleright) at ($(f7)!0.5!(f8)$);
      \node[draw,mycircle] (AR) at ($(middleright) - (0,1)$) {$\gtrunccu k$};
      \path[draw] (AR.west) to[out=180,in=-90,looseness=0.8] (f7.south);
      \path[draw] (AR.east) to[out=0,in=-90,looseness=0.8] (f8.south);
      \foreach \vnode/\vlabel/\vdown/\vleft/\vright in {%
        BR/$\eps^n$/2/f6/f9,
        CR/$\alpha'$/3/f5/f10%
      }{
        \drawcounit{\vlabel}\vdown\vleft\vright
        \draw (\vright) -- +(0,4.5);
      };
    \end{tikzpicture}
  \]
  is an isomorphism,
  where~$(\alpha,\alpha')$,~$(\eta^n,\eps^n)$,~$(\gamma,\gtrunccu k)$ are the
  pairs of units and counits associated with the adjunctions~$\algtruncf k -
  \dashv \algincfill n -$,~$\freealgf_n \dashv \fgfalgf_n$ and~$\gincf n -
  \dashv \gtruncf k-$ respectively. Using the zigzag equations satisfied by
  adjunctions to reduce the above diagram, we obtain
  
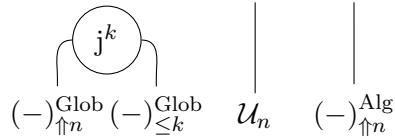
\[
    \begin{tikzpicture}[xscale=1.3,mycircle/.style={circle,minimum size=0.9cm}]
      \path (0,0) node (f1) {$\gincfill n-$}
      ++(1,0) node (f2) {$\gtruncf k -$}
      ++(1,0) node (f3) {$\fgfalgf_n$}
      ++(1,0) node (f4) {$\algincfill n -$};
      \coordinate (middleleft) at ($(f1)!0.5!(f2)$);
      \drawunit{$\gincfillu k$}{1}{f1}{f2}
      \draw (f3) -- +(0,1.5);
      \draw (f4) -- +(0,1.5);
    \end{tikzpicture}
  \]
  which is the diagram associated to the morphism~$\gincfillu k \fgfalgf_k
  \algincfill k-$. Thus,~$\algtruncf k - \freealgf_n \gtrunccu k$ is an
  isomorphism if and only if~$\smash{\gincfillu k \fgfalgf_k \algincfill k-}$ is
  an isomorphism. We conclude with \Lemr{truncability-through-algeba-cats}.
\end{proof}
\noindent We will use the above criterion to show that the monad associated to
the theories of strict categories is weakly truncable (\cf
\Cref{thm:sc-monad-wtruncable}).


\endgroup



\section{Free higher categories on generators}
\label{text:free-higher}

Given some theory of higher categories, an important construction is the one
that builds a $k$\category which is freely generated on a set of generators.
Indeed, like for other algebraic theories, a $k$\category can be described by
means of a \emph{presentation}, \ie by quotienting a free $k$\category by a set
of relations. For example, a formal adjunction can be described as the strict
$2$\category generated by two $0$\cells~$x$ and~$y$, two $1$\cells~$l \co y \to
x$ and~$r \co x \to y$, and two $2$\cells~$\gamma \co \unit y \To l \comp_0 r$
and~$\eps \co r \comp_0 l \To \unit x$ satisfying the zigzag
identities. Given a theory of higher
categories expressed in Batanin's setting, \ie as a monad~$(T,\eta,\mu)$
on~$\nGlob n$ for some~$n \in \Ninf$, there are several free constructions that
one can consider. First, the functors~$\freealgf_k \co \nGlob k \to \Alg_k$
already enable to construct the free $k$\category on a $k$\globular set.
Moreover, there is a construction which produces a $(k{+}1)$\category from a
\emph{$k$\cellular extension}, \ie a pair consisting of a $k$\category and a set
of $(k{+}1)$\generators. Such construction was introduced for strict categories
in~\cite{burroni1993higher}. Finally, one can consider the free $k$\category on
a \emph{$k$\polygraph}: the latter is a system of $i$\generators for~$i \in
\N_n$ which is organized inductively as cellular extensions. It differs from a
mere $k$\globular set in the sense that a $k$\polygraph allows generators to
have complex sources and targets that are composites of other generators,
whereas the sources and targets of generators organized in a $k$\globular set
can only be globes. Polygraphs were first introduced by
Street~\cite{street1976limits} and Burroni~\cite{burroni1993higher} for strict
categories, and then generalized to any finitary monad on globular sets by
Batanin~\cite{batanin1998computads}.

In this section, we define free constructions for algebraic globular higher
categires following the path of Burroni, giving in the process another light on
the results of Batanin, who was not relying on cellular extensions in his
proofs. More precisely, we define the notion of \emph{cellular extension} for
any algebraic theory of globular higher categories, and then derive the notion
\emph{polygraphs} from it, together with the free construction associated to
each notion. Since most of the definitions rely on pullbacks in~$\CAT$, we first
recall some properties of these pullbacks (\Cref{text:pullback-in-CAT}). Then,
we introduce \emph{cellular extensions} together with the associated free
construction for any finitary monad on globular sets, and, in the case of a
truncable monad, we show that this construction is stable, \ie that freely
adding $(k{+}1)$\generators does not change the underlying $k$\category
(\Cref{ssec:cell-exts}). Then, we introduce \emph{polygraphs} together with
the associated free construction for any finitary monad on globular sets
(\Cref{text:polygraphs}). Finally, we recover in our setting the adjunction of
Batanin between higher categories and polygraphs in \Cref{text:pol-adj}.

\subsection[Pullbacks in \texorpdfstring{$\CAT$}{CAT}]{Pullbacks in $\bm{\CAT}$}
\label{text:pullback-in-CAT}

In the following sections, we define the categories of cellular extensions and
polygraphs using pullbacks in~$\CAT$. We will be interested in showing that
these categories are cocomplete and that several of the projection functors are
left or right adjoints. Such properties are consequence of general properties of
pullbacks that we recall below. In particular, a pullback of an
\emph{isofibration}, \ie a functor which lifts isomorphisms, has good properties
with regard to cocompleteness and preservation of colimits. This is convenient
since, as we will see below, all the truncation functors introduced until now
are isofibrations.

\smallpar In the following, given~$C \in \CAT$, we write~$\unitp 2 C$ for the
identity natural transformation on the identity functor~$\unit C \co C \to C$.
We begin with a property of compatibility of pullbacks in~$\CAT$ with left and
right adjoints:
\begin{prop}
  \label{prop:pullback-la-ra}
  Given a pullback in~$\CAT$
  \[
    \begin{tikzcd}
      C'
      \ar[r,"F'"]
      \ar[d,"G'"']
      &
      C
      \ar[d,"G"]
      \\
      D'
      \ar[r,"F"']
      &
      D
    \end{tikzcd}
  \]
  and a functor~$H\co D \to C$ such that~${GH = \unit D}$, then there exists a
  canonical~$H'\co D' \to C'$ such that~${G'H' = \unit {D'}}$. Moreover, if there
  is an adjunction~$H \dashv G$ (\resp~$G \dashv H$) whose unit (\resp counit)
  is~$\unitp 2 {D}$, then there is an adjunction~$H' \dashv G'$ (\resp~$H \dashv
  G$) whose unit (\resp counit) is~$\unitp 2 {D'}$.
\end{prop}
\begin{proof}
  We define~$H'$ using the universal property of pullbacks by
  \[
    \begin{tikzcd}[cramped,row sep=small]
      D'
      \ar[rrd,"H\circ F",bend left]
      \ar[rd,"H'"{description},dotted]
      \ar[rdd,"\unit {D'}"',bend right]
      \\
      &
      C'
      \ar[r,"F'"]
      \ar[d,"G'"']
      &
      C
      \ar[d,"G"]
      \\
      &
      D'
      \ar[r,"F"']
      &
      D
    \end{tikzcd}
  \]
  which satisfies~$G'H' = \unit {D'}$ by definition. Moreover, suppose that
  there is an adjunction~$H \dashv G$ whose unit is~$\unitp 2 {D}$. Then,
  since~$C'$ is defined by a pullback, a morphism~${f\co H'X \to Y \in C'}$ is
  the data of morphisms~$f_l \co X \to G'Y$ and~$f_r \co HFX \to F'Y$
  with~$F(f_l) = G(f_r)$. But, since the unit of~${H \dashv G}$ is~$\unitp 2
  {D}$,~$G$ induces a bijective correspondence between~$C(HFX,F'Y)$
  and~$D(FX,GF'Y)$, so that~$f_r$ is uniquely defined by~$F(f_l)$. Thus,~$G'$
  induces a bijective natural correspondence between~$C'(H'X,Y)$ and~$D'(X,G'Y)$
  for all~$X \in D'$ and~$Y \in C'$, so that there is an adjunction~${H' \dashv
    G'}$ with unit~${\unitp 2 {D'}}$. The case where~$G$ is left adjoint is
  similar.
\end{proof}
\noindent
Moreover, we prove that isofibrations are well-behaved regarding pullbacks
in~$\CAT$. We recall that a functor~$G\co C \to D \in \CAT$ is an
\index{isofibration}\emph{isofibration} when it lifts isomorphisms, \ie for all~$X \in C$
and~$\tilde Y \in D$, given an isomorphism~$\tilde f\co GX \to \tilde Y$ in~$D$,
there exists~$Y \in C$ and an isomorphism~$f\co X \to Y$ such that~${GY = \tilde
  Y}$ and~${G(f) = \tilde f}$. We then have:\ndr{pas trouvé de référence pour
  cette prop}%
\begin{prop}
  \label{prop:pullback-CAT-iso-liftting-colimits}
  Given a pullback in~$\CAT$
  \[
    \begin{tikzcd}
      C'
      \ar[r,"F'"]
      \ar[d,"G'"']
      &
      C
      \ar[d,"G"]
      \\
      D'
      \ar[r,"F"']
      &
      D
    \end{tikzcd}
  \]
  such that~$G$ is an isofibration, the following hold:
  \begin{enumerate}[label=(\roman*),ref=(\roman*)]
  \item \label{prop:pullback-CAT-iso-liftting-colimits:iso-lifting} $G'$ is an isofibration,
  \item \label{prop:pullback-CAT-iso-liftting-colimits:colimits} given a small
    category~$I$, if~$C$ and~$D'$ have all $I$-colimits and~$F$ and~$G$ preserve
    them, then~$C'$ has all $I$-colimits and~$F'$ and~$G'$ preserve them.
  \end{enumerate}
\end{prop}

\begin{proof}
  \proofof{prop:pullback-CAT-iso-liftting-colimits:iso-lifting} Let~$X \in C'$,~$Y_L \in D'$ and~$\theta_L\co G'X \to Y_L$ be an isomorphism. Then, since~$G$
  is an isofibration, there is~$Y_R \in C$ and an isomorphism~$\theta_R\co F'X
  \to Y_R$ such that~$F(\theta_L) = G(\theta_R)$. Moreover,~$F(\finv\theta_L) =
  G(\finv\theta_R)$ so that~$(\theta_L,\theta_R) \co X \to (Y_L,Y_R)$ is an
  isomorphism of~$C'$.
  \proofof{prop:pullback-CAT-iso-liftting-colimits:colimits} Let~$d\co I \to C'$
  be a functor, which is the data of~$d_L\co I \to D'$ and~$d_R\co I \to C$.
  Then, there are colimit cocones~$(p_{L,i}\co d_L(i) \to X_L)_{i \in I}$ and~$(p_{R,i}\co d_R(i) \to X_R)_{i \in I}$. Since both~$F$ and~$G$ preserve
  colimits, both
  \[
    (F(p_{L,i})\co F(d_L(i)) \to F(X_L))_{i \in I}
    \qtand
    (G(p_{R,i})\co F(d_L(i)) \to G(X_R))_{i \in I}
  \]
  are colimit cocones for~$F\circ d_L$. So there exists an
  isomorphism~$\theta\co F(X_L) \to G(X_R)$ between the two cocones. Since~$G$
  is an isofibration, we can suppose that~$F(X_L) = G(X_R)$ and~$\theta = \unit
  {F(X_L)}$. Thus, we have a cocone~$((p_{L,i},p_{R,i})\co d(i) \to
  (X_L,X_R))_i$ on~$d$, and we easily verify that it is a colimit cocone.
\end{proof}
\begin{remark}
  Pullbacks in~$\CAT$ should normally raise suspicion since strict limits are
  not well-behaved in~$\CAT$ in general. Indeed, a limit cone in~$\CAT$ on a
  diagram is not stable when replacing some functors of the diagram by
  isomorphic functors. Moreover, the limit cone is defined up to isomorphism,
  and not up to equivalence of categories. To solve this problem, one usually
  considers a weaker notion of limits, where the triangles of cones commute only
  up to isomorphisms, as with weighted bilimits~\cite{makkai1989accessible}. But
  the strict limit on a diagram is generally not equivalent to the associated
  weighted bilimit. However, introducing weighted bilimits here would be an
  unnecessary pain for what we want to do, since the pullbacks along
  isofibrations are equivalent to the weighted bipullbacks
  (see~\cite[Proposition~5.1.1]{makkai1989accessible}).
\end{remark}
\smallpar We say that a monad functor is an \emph{isofibration} when the underlying
functor is an isofibration. We then have:
\begin{lem}
  \label{lem:em-pres-isofib}
  The functor $\EM$ preserves isofibrations.
\end{lem}
\begin{proof}
  Let $(F,\alpha) \co (\cC,S) \to (\cD,T)$ be a monad functor isofibration. Let
  $(X,h \co SX \to X)$ be an $S$\algebra, and $f \co F^\alpha(X,h) \to (Y,k)$ be
  a $T$\algebra isomorphism. Since $F$ is an isofibration, there is an
  isomorphism $\tilde f \co X \to \tilde Y$ such that $F(\tilde f) = f$. One can
  equip $\tilde Y$ with a structure of $S$\algebra $(\tilde Y,\tilde k)$ by
  putting $\tilde k = \tilde f \circ h \circ S(\finv{\tilde f})$. Then, we have
  an $S$\algebra morphism $\tilde f \co (X,h) \to (\tilde Y,\tilde k)$.
  Moreover, $(\tilde Y,\tilde k)$ is mapped to $(Y,k)$ by $F^\alpha$, so that
  $\tilde f$ is an isomorphism which lifts $f$. Hence, $\EM$ preserves
  isofibration.
\end{proof}
\noindent We now verify that several functors of interest to us are isofibrations:
\begin{prop}
  \label{prop:gtruncf-lifts-isom}
  Given~$k,l \in \Ninf$ with $k<l$, the functor~$\gtruncf[l] k {(-)}$ is an isofibration.
\end{prop}
\begin{proof}
  Straight-forward.
\end{proof}
\begin{prop}
  \label{prop:algtruncf-lifts-isom}
  Let~$n \in \Ninf$ and~$(T,\eta,\mu)$ be a finitary monad on~$\nGlob n$. Given $k,l \in
  \Ninf$ with $k < l$, the functor~$\algtruncf[l] k {(-)}$ is an isofibration.
\end{prop}
\begin{proof}
  The functor $\algtruncf[l] k {(-)}$ is the image by $\EM$ of
  $(\gtruncf[l]k-,\gtruncf[l]k-T^{l}\gtrunccu[l]k)$, which is an
  isofibration. Thus, by \Cref{lem:em-pres-isofib}, it is an isofibration.
\end{proof}

\subsection{Cellular extensions}
\label{ssec:cell-exts}
In this section, we introduce the notion of \emph{$k$\cellular extension}, which
describes a~$k$\category (for some theory of higher categories) equipped with a
set of \index{generator!of a cellular extension}\emph{$(k{+}1)$\generators}. We moreover give the construction of the free
$(k{+}1)$\category on a $k$\cellular extension together with more specific
results when the theory we are considering is associated with a truncable monad.


Let~$n \in \Ninf$ and~$(T,\eta,\mu)$ be a finitary monad on~$\nGlob n$. Given~$k
\in \N_{n-1}$, we define the \glossary(Algp){$\Algp_k$}{the category of
  $k$\cellular extensions}category~$\Algp_k$ of \index{cellular
  extension@$k$-cellular extension}\emph{$k$\cellular extensions} as the
\glossary(Ak){$\cetoalg_{k}$}{the functor which maps a $k$\cellular extension to
its underlying $k$\algebra}\glossary(Gk){$\cetoglob_{k}$}{the functor which maps
a $(k{-}1)$\cellular extension to its underlying $k$\globular set}pullback
\[
  \begin{tikzcd}[cramped,column sep=4em]
    \Algp_k \ar[r,dotted,"\cetoglob_{k+1}"] \ar[d,dotted,"\cetoalg_{k}"'] & \nGlob {k+1} \ar[d,"\gtruncf k {(-)}"] \\
    \Alg_k \ar[r,"\fgfalgf_k"']& \nGlob k
  \end{tikzcd}
\]
\noindent We verify that:
\begin{prop}
  \label{prop:algp-oone-pres}
  The category $\Algp_k$ is locally finitely presentable. In particular, it is
  complete and cocomplete. Moreover, both $\cetoglob_{k+1}$ and $\cetoalg_k$ are
  finitary right adjoints.
\end{prop}
\begin{proof}
  The categoy $\Algp_k$ is defined as the pullback of the finitary right adjoint
  isofibration~$\gtruncf k -$ along the right adjoint $\fgfalgf_k$. Thus, it is
  a bipullback of right adjoints between locally finitely presentable
  categories. Since the $2$\category of locally finitely presentable and
  finitary right adjoint functors is closed under bilimits (see
  \cite[Theorem~2.17]{bird}), $\Algp_k$ is locally finitely presentable and the
  two projections $\cetoglob_{k+1}$ and~$\cetoalg_k$ are finitary right adjoint
  functors.
\end{proof}
\begin{rem}
  \label{rem:algp-fin-pres}
  In fact, one can check by hand that the finitely presentable objects of
  $\Algp_k$ are exactly the compatible pairs $(C,X)$, where $C$ is a finitely
  presentable object of $\Alg_k$ and $X$ is a compatible $(k{+}1)$\globular set
  such that $X_{k+1}$ is finite.
\end{rem}

\begin{prop}
  \label{prop:cetoalg-la}
  The functor~$\cetoalg_{k}$ is an isofibration and has a right adjoint.
\end{prop}
\begin{proof}
  This is a consequence of \Cref{prop:gtruncf-lifts-isom,prop:pullback-la-ra}.
\end{proof}

\begin{prop}
  \label{prop:cetoalg-pres-presentables}
  The functor $\cetoalg_k$ preserves finitely presentable objects.
\end{prop}
\begin{proof}
  By adjunction, it is equivalent to prove that its right adjoint is finitary.
  By the constructive content of \Cref{prop:pullback-la-ra}, this right adjoint
  is built from a cone of finitary functors on the (bi)pullback defining
  $\Algp_k$, itself made of finitary functors. Thus, the factorizing functor is
  itself finitary.
\end{proof}

\noindent There is a \glossary(Vk){$\algtoce_{k}$}{the forgetful functor from a
  $(k{+}1)$\algebra to a $k$\cellular extension}functor~$\algtoce_{k}\co \Alg_{k+1} \to \Alg_k^+$ defined
as the factorization arrow
\[
  \begin{tikzcd}[column sep=4em,cramped]
    \Alg_{k+1}
    \ar[rrd,bend left,"\fgfalgf_{k+1}"] \ar[ddr,bend right,"\algtruncf {k}{(-)}"']
    \ar[rd,dotted,"\algtoce_{k}"] \\
    & \Alg_k^+
    \ar[d,"\cetoalg_{k}"']
    \ar[r,"\cetoglob_{k+1}"]&\nGlob{k+1}\ar[d,"\gtruncf k {(-)}"]\\
    & \Alg_k\ar[r,"\fgfalgf_k"']&\nGlob k
  \end{tikzcd}
  \pbox.
\]
\noindent Then, there is an operation which produces a $(k{+}1)$\category from a
$k$\cellular extension. It is the left adjoint to~$\algtoce_{k}$, that exists by
the following property:
\begin{theo}
  \label{prop:algtoce-ra}
  $\algtoce_{k}$ has a left adjoint.
\end{theo}
\begin{proof}
  Let~$\alpha^k$ be the unit of the adjunction~$\smash{\algincf[k]{k+1}- \dashv
    \algtruncf[k+1]k -}$, and let
\[
  \arraycolsep0pt
  \begin{array}{cl@{\quad}c@{\quad}c@{\quad}c}
    \Phi^{\mathrm L} &\co& \smash{\Alg_{k+1}(\algincf {k+1} {(-)},-)}\mathstrut &\To& \smash{\Alg_k(-,\algtruncf k
                       {(-)})}\mathstrut \\[0.25cm]
    \Psi^{\mathrm L} &\co& \smash{\Alg_k(\algincf {k+1}{(-)}\freealgf_k(-),-)}\mathstrut&\To& \smash{\nGlob
                       k(-,\fgfalgf_k\algtruncf[k+1] {k} {(-)})}\mathstrut \\[0.25cm]
    \Phi^{\mathrm R}&\co&
                      \Alg_{k+1}(\freealgf_{k+1}(-),-) &\To&
                                                              \nGlob
                                                             {k+1}(-,\fgfalgf_{k+1}(-))
    \\[0.25cm]
    \Psi^{\mathrm R}&\co& \smash{\nGlob {k+1}(\freealgf_{k+1}\gincf[k] {k+1} {(-)},-)}\mathstrut &\To &\smash{\nGlob k (-,\gtruncf
                      k {(-)}\fgfalgf_{k+1}(-))}\mathstrut
  \end{array}
\]
be the natural bijections derived from the associated adjunctions. Note that
these bijections can be defined using the units of the adjunctions. For example,
given~$C \in \Alg_k$ and~$D \in \Alg_{k+1}$,~$\Phi^{\mathrm L}$ maps a
morphism~$f \co \incf{k+1} C \to D \in \Alg_{k+1}$ to the morphism
\[
  \restrict k f \circ \alpha^k_C \co C \to \restrict k D \in \Alg_k
\]
where $\alpha$ is the unit of $\algtruncf k - \dashv \algincf {k+1} -$.
Since
\[
  \mathmakebox[3cm][r]{\freealgf_{k+1}\gincf {k+1}{(-)}}
  \qtand
  \mathmakebox[3cm][l]{\algincf {k+1}
    {(-)}\freealgf_k}
\]
are both left adjoint to~$\fgfalgf_{k}\algtruncf[k+1] k {(-)} = \gtruncf k
{(-)}\fgfalgf_{k+1}$, the natural morphism
\[
  \theta\co \freealgf_{k+1}\gincf[k] {k+1}{(-)} \To \algincf {k+1} {(-)}\freealgf_k
\]
defined as the composite
\[
  \hss
  \theta = (\eps^{k+1} \algincf[k]{k+1}-\freealgf_k) \circ (\freealgf_{k+1} \gtrunccu k  \fgfalgf_{k+1} \algincf[k]{k+1}-\freealgf_k) \circ (\freealgf_{k+1} \gincf[k]{k+1}- \fgfalgf_k \alpha^k \freealgf_k) \circ (\freealgf_{k+1} \gincf[k]{k+1}- \eta^k)
  \hss
\]
which can be represented by 
\[
  \satex{algtoce-ra-theta-l}
  \qqeq
  \satex{algtoce-ra-theta-r}
\]
is an isomorphism as a consequence of \Cref{prop:isom-btw-adj}. In the
following, given a morphism $f \co X \to Y$ of a category~$\mcal C$, we
write~$f^*\co \mcal C(Y,Z) \to \mcal C(X,Z)$ for the function~$g \mapsto g \circ
f$ for all~$Z \in \mcal C$. One can verify using the zigzag equations that the
natural transformation~$\theta$ makes the diagram \begingroup \makeatletter
\renewcommand{\maketag@@@}[1]{\hbox to
  0.000008pt{\hss\m@th\normalsize\normalfont#1}}%
\makeatother
\begin{equation}
  \label{eq:algtoce-ra:psil-psir}
  \begin{tikzcd}[cramped,column sep={8cm,between origins}]
    &[-5cm]
    \Alg_{k+1}(\incf {k+1} {(\freealgf_k Z)},A)
    \ar[d,"(\theta_Z)^*"']
    \ar[r]
    \ar[r,phantom,start anchor=center,end anchor=center,"\scriptstyle\Psi^{\mathrm L}_{Z,A}"{above}]
    &
    \nGlob k(Z,\fgfalgf_{k}(\restrict k {A}))
    \ar[d,equals]
    &[-5cm]
    \\
    &
    \Alg_{k+1}(\freealgf_{k+1}(\incf {k+1}{Z}),A)
    \ar[r]
    \ar[r,phantom,start anchor=center,end anchor=center,"\scriptstyle\Psi^{\mathrm R}_{Z,A}"'{below}]
    &
    \nGlob k(Z,\restrict k {(\fgfalgf_{k+1}A)})
  \end{tikzcd}
\end{equation}
\endgroup
commutes for all~$Z \in \nGlob k$ and~$A \in \Alg_{k+1}$.
Let~$(C,X) \in
\Algp_{k}$,~$D \in \Alg_{k+1}$ and~$(\restrict k D,Y)$ be~$\algtoce_{k} D$. Since
\[
  \mathmakebox[4cm][r]{\fgfalgf_k C =
  \restrict k X}
  \qtand
  \mathmakebox[4cm][l]{\fgfalgf_k D = \restrict k Y}
\]
and by the properties of adjunctions, we have a diagram
\begingroup
\makeatletter
\renewcommand{\maketag@@@}[1]{\hbox to 0.000008pt{\hss\m@th\normalsize\normalfont#1}}%
\makeatother
\begin{equation}
  \label{eq:algtoce-ra:iso-squares}
  \begin{tikzcd}[cramped,column sep={8cm,between origins}]
    &[-5cm]
    |[alias=R1C1]|\Alg_{k+1}(\incf {k+1} {C},D)
    \ar[r]
    \ar[d,"(e^{\mathrm L}_{(C,X)})^*"']
    &
    |[alias=R1C2]|\Alg_k(C,\restrict k D)
    \ar[d,"\fgfalgf_k"]
    &[-5cm]
    \\
    &
    |[alias=R2C1]|\Alg_{k+1}(\incf {k+1} {(\freealgf_k\fgfalgf_{k}C)},D)
    \ar[d,"(\theta_{\restrict k X})^*"']
    \ar[r]
    &
    |[alias=R2C2]|\nGlob k({\fgfalgf_{k} C},\fgfalgf_{k}(\restrict k {D}))
    \ar[d,equals]
    \\
    &
    |[alias=R3C1]|\Alg_{k+1}(\freealgf_{k+1}(\incf {k+1}{(\restrict k {X})}),D)
    \ar[r]
    &
    |[alias=R3C2]|\nGlob k({\restrict k {X}},\restrict k {Y})
    \\
    &
    |[alias=R4C1]|\Alg_{k+1}(\freealgf_{k+1} X,D)
    \ar[u,"(e^{\mathrm R}_{(C,X)})^*"]
    \ar[r]
    &
    |[alias=R4C2]|\nGlob {k+1}(X,Y)
    \ar[u,"\gtruncf k {(-)}"']
    \ar[r,from=R1C1.center,to=R1C2.center,phantom,"\scriptstyle\Phi^{\mathrm L}_{C,D}"{above}]
    \ar[r,from=R2C1.center,to=R2C2.center,phantom,"\scriptstyle\Psi^{\mathrm L}_{\fgfalgf_k C,D}"{above}]
    \ar[r,from=R3C1.center,to=R3C2.center,phantom,"\scriptstyle\Psi^{\mathrm R}_{\restrict k X,D}"{above}]
    \ar[r,from=R4C1.center,to=R4C2.center,phantom,"\scriptstyle\Phi^{\mathrm R}_{X,D}"{above}]
  \end{tikzcd}
\end{equation}
\endgroup such that each square commutes and where~$e^{\mathrm L}$ and~$e^{\mathrm R}$
are the natural transformations
\[
  \mathmakebox[4cm][r]{e^{\mathrm L} = \algincf {k+1} {(-)} \eps^k \cetoalg_{k}}
  \qtand
  \mathmakebox[4cm][l]{e^{\mathrm R} = \freealgf_{k+1} \gtrunccu k \cetoglob_{k+1}}
\]
respectively. Indeed, the middle square commutes
by~\eqref{eq:algtoce-ra:psil-psir} and the top and bottom squares commute by the
zigzag equations. By definition of~$\Algp_k$, the
set~$\Algp_k((C,X),\algtoce_{k}D)$ is the pullback
\[
  \begin{tikzcd}[cramped,column sep={8cm,between origins}]
    &[-5cm]
    \Algp_k((C,X),\algtoce_{k}D)
    \centeredar{r}{\cetoglob_{k+1}}{above}
    \ar[d,"\cetoalg_{k}"']
    \centeredphantom{rd}{\drcorner}{pos=0.3,xshift=-1.8cm}
    &
    \nGlob {k+1}(X,\cetoglob_{k+1}\algtoce_{k}D)
    \ar[d,"\gtruncf k {(-)}"]
    &[-5cm]
    \\
    &
    \Alg_k(C,\cetoalg_{k}\algtoce_{k}D)
    \centeredar{r}{\fgfalgf_k}{below}
    &
    \nGlob k(\restrict k {X},\restrict k {(\cetoglob_{k+1}\algtoce_{k}D)})
  \end{tikzcd}
  \pbox.
\]
Since
\[
  \mathmakebox[4cm][r]{\algtruncf k {(-)} =
  \cetoalg_{k}\algtoce_{k}}
  \qtand
  \mathmakebox[4cm][l]{\fgfalgf_{k+1} = \cetoglob_{k+1} \algtoce_{k}}
\]
and by the commutative diagram~\eqref{eq:algtoce-ra:iso-squares}, the following
diagram is also a pullback:
\[
  \begin{tikzcd}[cramped,column sep={8cm,between origins}]
    &[-5cm]
    \Algp_k((C,X),\algtoce_{k}D)
    \centeredar{r}{\finv{(\Phi^{\mathrm R}_{X,D})}\circ\cetoglob_{k+1}}{above}
    \ar[d,"\finv {(\Phi^{\mathrm L}_{C,D})}\circ \cetoalg_{k}"']
    \centeredphantom{rd}{\drcorner}{pos=0.3,xshift=-1.8cm}
    &
    \Alg_{k+1}(\freealgf_{k+1}X,D)
    \ar[d,"(e^{\mathrm R}_{(C,X)})^*"]
    &[-5cm]
    \\
    &
    \Alg_{k+1}(\incf {k+1} {C},D)
    \centeredar{r}{(e^{\mathrm L}_{(C,X)} \circ \theta_{\restrict k X})^*}{below}
    &
    \Alg_{k+1}(\freealgf_{k+1}(\incf {k+1}{(\restrict k {X})}),D)
  \end{tikzcd}
  \pbox.
\]
Since~$\Alg_{k+1}$ is cocomplete by \Cref{prop:alg-cat-loc-fin-pres}, the
diagram
\[
  \begin{tikzcd}[cramped,column sep={8cm,between origins}]
    &[-5cm]
    \Alg_{k+1}(\polextp C X,D)
    \centeredar{r}{(p^{\mathrm R}_{(C,X)})^*}{above}
    \ar[d,"(p^{\mathrm L}_{(C,X)})^*"']
    \centeredphantom{rd}{\drcorner}{pos=0.3,xshift=-1.8cm}
    &
    \Alg_{k+1}(\freealgf_{k+1}X,D)
    \ar[d,"(e^{\mathrm R}_{(C,X)})^*"]
    &[-5cm]
    \\
    &
    \Alg_{k+1}(\incf {k+1} {C},D)
    \centeredar{r}{(e^{\mathrm L}_{(C,X)} \circ \theta_{\restrict k X})^*}{below}
    &
    \Alg_{k+1}(\freealgf_{k+1}(\incf {k+1}{(\restrict k {X})}),D)
  \end{tikzcd}
\]
is also a pullback, where~$\polextp C X$,~$p^{\mathrm L}_{(C,X)}$
and~$p^{\mathrm R}_{(C,X)}$ are defined as the pushout
\[
  \begin{tikzcd}[column sep={8cm,between origins},cramped]
    &[-5cm]
    \polextp C X
    \centeredphantom{rd}{\drcorner}{pos=0.3,xshift=-1.8cm}
    &
    \freealgf_{k+1}X
    \centeredar{l,dotted}{p^{\mathrm R}_{(C,X)}}{above}
    &[-5cm]
    \\
    &
    \incf {k+1} {C}
    \ar[u,"p^{\mathrm L}_{(C,X)}",dotted]
    &
    \freealgf_{k+1}(\incf {k+1}{(\restrict k {X})})
    \ar[u,"e^{\mathrm R}_{(C,X)}"']
    \centeredar{l}{e^{\mathrm L}_{(C,X)} \circ \theta_{\restrict k X}}{below}
  \end{tikzcd}
  \pbox.
\]
Thus, there is an isomorphism
\[
  \Alg_{k+1}(\polextp C X,D) \cong \Algp_k((C,X),\algtoce_{k}D)
\]
which is natural in~$D$. Hence,~$\algtoce_{k}$ admits a left adjoint.
\end{proof}
\noindent The operation~$(C,X) \mapsto \polextp C X$ defined in the proof of
\Cref{prop:algtoce-ra} extends to a \glossary(.b){$\polextp[k] - -$, $\polextp C
  X$}{the free functor which maps a $k$\cellular extension to a $(k{+}1)$\category}functor
\[
  \polextp[k] - - \co \Algp_k \to \Alg_{k+1}
\]
that we often write~$\polextp--$ when there is no ambiguity on~$k$, and which is
left adjoint to~$\algtoce_{k}$. The image~$\polextp C X$ of some~$(C,X) \in
\Algp k$ is called the \index{free category!on a cellular extension}\index{free extension}\emph{free extension on~$(C,X)$}.

\begin{example}
  Consider the monad~$(T,\eta,\mu)$ on~$\nGlob 1$ defined in
  \Cref{ex:cat-globular-algebra}. A $0$\cellular extension~$(C,X)$ is then
  essentially the data of a set~$C_0$ of $0$\cells, a set~$X_1$
  of~$1$\generators, and functions~$\gsrc_0,\gtgt_0 \co X_1 \to C_0$, \ie a
  graph. Moreover, the $1$\category~$\polextp C X$ is the image of~$(C,X)$ seen
  as a graph by the left adjoint to the functor~$\Cat \to \Gph$ defined in
  \Cref{ex:grph-to-cat-theo-morphism}.
\end{example}

\begin{remark}
  \Cref{prop:algtoce-ra} is a particular case of the fact that the category of
  locally presentable categories and right adjoints are closed under weighted
  bilimits (see~\cite{bird} and the end of~\cite[\S5.1]{makkai1989accessible}).
  But the concrete pushout description given by the proof will be useful later
  to show properties of the functor~$\polextp--$.
\end{remark}

\begin{prop}
  \label{prop:polextp-pres-presentable}
  For every $k \in \N$, the functor $\polextp[k]--$ preserves finitely
  presentable objects.
\end{prop}
\begin{proof}
  By adjunction, it is equivalent to show that its right adjoint $\algtoce_k$
  is finitary. But, since $T$ is finitary, $\fgfalgf_k$ is finitary, and
  $\algtruncf k-$ is finitary by \Cref{prop:algtruncf-ra}, so that the same
  argument as in the proof of \Cref{prop:cetoalg-pres-presentables} applies to
  conclude that $\algtoce_k$ is finitary.
\end{proof}

\paragraph{The truncable case}

Let~$n \in \Ninf$ and~$(T,\eta,\mu)$ be a finitary monad on~$\nGlob n$. In this
paragraph, we consider the case where~$T$ is a truncable monad, and show that
the underlying $k$\category of a $k$\cellular extension is preserved
by~$\polextp[k]--$. In the following, we write $\algunit[k]$ for the unit of the
adjunction $\algincf[k]{k+1}- \dashv \algtruncf[k+1]k-$. We have:
\begin{prop}
  \label{prop:truncable-algincf-unit-iso}
  If~$T$ is truncable, then~$\algunit[k]$ is an isomorphism for $k \in \N_{n-1}$.
\end{prop}\ndr{il existe des monades~$T$ non troncables pour lesquelles ce n'est
pas vrai}%
\begin{proof}
  The adjunction $\algtruncf k-\dashv \algincfill{k+1}-$ was built from the
  adjunction $\gtruncf k -\dashv \gincfill{k+1}-$ by lifting. The counit of the
  latter being an isomorphism, the counit of the former is as well, so that
  $\algincfill{k+1}-$ is fully faithful. Since we have $\algincf{k+1}- \dashv
  \algtruncf k-\dashv \algincfill{k+1}-$, by the unity and identity of
  opposites, $\algincf{k+1}-$ is fully faithful as well, so that $\algunit[k]$
  is an isomorphism.
\end{proof}
\noindent We can conclude a conservation result for the underlying $k$\category
of $(k{+}1)$\categories produced by~$\polextp[k] - -$:
\begin{prop}
  \label{prop:algtruncf-polextp-isom}
  If~$T$ is truncable, then, given~$k \in \N_n$ and~$(C,X) \in \Algp_k$, there
  is an isomorphism~$C \cong \restrict k {\polextp C X}$ which is natural
  in~$(C,X)$.
\end{prop}
\begin{proof}
  Recall that~$\polextp C X$ was defined in the proof of \Cref{prop:algtoce-ra} as
  the pushout
  \[
    \begin{tikzcd}[cramped,column sep={8cm,between origins}]
      &[-5cm]
      \polextp C X
      \centeredphantom{rd}{\drcorner}{pos=0.3,xshift=-1.8cm}
      &
      \freealgf_{k+1}X
      \centeredar{l}{p^{\mathrm R}_{(C,X)}}{above}
      &[-5cm]
      \\
      &
      \incf {k+1} {C}
      \ar[u,"p^{\mathrm L}_{(C,X)}"]
      &
      \freealgf_{k+1}(\incf {k+1}{(\restrict k {X})})
      \ar[u,"e^{\mathrm R}_{(C,X)}"']
      \centeredar{l}{e^{\mathrm L}_{(C,X)} \circ \theta_{\restrict k X}}{below}
    \end{tikzcd}
    \pbox{\hspace{-3em}.}
  \]
  By \Propr{truncable-algtruncf-la}, the following diagram is also a pushout
  \[
    \begin{tikzcd}[cramped,column sep={8cm,between origins}]
      &[-5cm]
      \restrict k {\polextp C X}
      \centeredphantom{rd}{\drcorner}{pos=0.3,xshift=-1.8cm}
      &
      \restrict k {(\freealgf_{k+1}X)}
      \centeredar{l}{\restrict k {(p^{\mathrm R}_{(C,X)})}}{above}
      &[-5cm]
      \\
      &
      \restrict k {(\incf {k+1} {C})}
      \ar[u,"\restrict k {(p^{\mathrm L}_{(C,X)})}"]
      &
      \restrict k {(\freealgf_{k+1}(\incf {k+1}{(\restrict k {X})}))}
      \ar[u,"\restrict k {(e^{\mathrm R}_{(C,X)})}"']
      \centeredar{l}{\restrict k {(e^{\mathrm L}_{(C,X)} \circ \theta_{\restrict k X})}}{below}
    \end{tikzcd}
    \pbox{\hspace{-3em}.}
  \]
  Since~$T$ is truncable, we have~$\algtruncf k {(-)} \freealgf_{k+1} =
  \freealgf_k \gtruncf k {(-)}$. Thus,
  \begin{align*}
    \restrict k {(e^{\mathrm R})} 
    &= \algtruncf k {(-)}\freealgf_{k+1} \gtrunccu k \cetoglob_{k+1} \\
    &= \freealgf_k \gtruncf k {(-)} \gtrunccu k \cetoglob_{k+1} \\
    &= \unit {\freealgf_k\gtruncf k {(-)}\cetoglob_{k+1}} & \text{(since~$\restrict k {(\gtrunccu k)}
  = \unit {\nGlob k}$)}
  \end{align*}
  so that~$\restrict k {(e^{\mathrm R}_{(C,X)})} = \unit {\freealgf_k(\restrict
    k X)}$. Hence,~$\restrict k {(p^{\mathrm L}_{(C,X)})}$ is an isomorphism,
  since the pushout of an isomorphism is an isomorphism. By
  \Propr{truncable-algincf-unit-iso}, we conclude that the composite
  \[
    \begin{tikzcd}[sep=huge,cramped]
      C \ar[r,"{\algunit[k]_C}"] &
      \restrict k {(\incf {k+1} C)}
      \ar[r,"\restrict k {(p^{\mathrm L}_{(C,X)})}"]
      &
      \restrict k {\polextp C X}
    \end{tikzcd}
  \]
  is an isomorphism.
\end{proof}

\begin{remark}
  \label{rem:algtruncf-strict}
  If~$T$ is truncable, given~$k \in \N_n$, by \Propr{algtruncf-lifts-isom}, we
  can suppose that we chose~$\polextp - -_k$ so that the isomorphism of
  \Propr{algtruncf-polextp-isom} is the identity. When such a choice is made, we
  have~$\restrict k {\polextp C X} = C$ for all $k$\cellular extension~$(C,X)$.
\end{remark}

\subsection{Polygraphs}
\label{text:polygraphs}

In this section, we recover the generalized notion of \emph{polygraph} of
Batanin using cellular extensions. Intuitively, a polygraph is a system of
generators which organizes inductively as cellular extensions, so that a
definition of polygraphs based on the latter structures seems relevant. In the
process, we study some properties of the different functors involved, and
consider the truncable case.

\paragraph{Another definition of cellular extensions}

Let~$n \in \N \cup \set\omega$ and~$(T,\eta,\mu)$ be a finitary monad on~$\nGlob
n$. Before defining polygraphs, we first provide an alternative definition
of~$\Alg_k$ which is simpler than the one based on pullbacks given in
\Cref{ssec:cell-exts}.

\begin{prop}
  \label{prop:algp-simpler-def}
  Given~$k \in \N_{n-1}$, the category~$\Algp_k$ is isomorphic to the category
  \begin{itemize}
  \item whose objects are the pairs~$(C,S)$ where~$C \in \Alg_k$ and~$S$ is a
    set, equipped with two \glossary(dstraight){$\gsrc_k,\gtgt_k$}{the source
      and target operations for generators of cellular extensions or
      polygraphs}functions
    \[
      \gsrc_k,\gtgt_k \co S \to C_k
    \]
    such
    that~$\csrctgt\eps_{k-1} \circ \gsrc_k = \csrctgt\eps_{k-1} \circ \gtgt_k$
    for~$\eps \in \set{-,+}$,
  \item and whose morphisms between two such pairs~$(C,S)$ and~$(C',S')$ are the
    pairs~$(F,f)$ where
    \[
      F \co C \to C' \in \Alg_k
      \qtand
      f \co S \to S' \in \Set
    \]
    and such that~$\gsrctgt\eps_k \circ f = F_k \circ \gsrctgt\eps_k$ for~$\eps
    \in \set{-,+}$.
  \end{itemize}
\end{prop}
\begin{proof}
  Write~$\Algbp_k$ for the category described in the statement. An isomorphism
  between~$\Algp_k$ and~$\Algbp_k$ can be described as follows. Given~$(C,X) \in
  \Algp_k$, we map~$(C,X)$ to the pair~$(C,X_{k+1})$ and, for~$\eps \in
  \set{-,+}$ and~$x \in X_{k+1}$, we put~$\gsrctgt\eps_k(x) = \csrctgt\eps_k(x)$
  (where~$\csrctgt\eps_k$ is the operation of the globular structure on~$X$),
  and we extend this mapping to morphisms of~$\Algp_k$ as expected. Since~$\fgfalgf_k C = \restrict k X$ for~$(C,X) \in \Algp_k$, the resulting functor
  is an isomorphism of categories.
\end{proof}
\noindent In the following, we will prefer this new definition of cellular
extensions to the original one. This other description helps better understand
how a colimit of a diagram $(C^i,S^i)$ in $\Algp_k$ is computed: first, one
computes a colimit $C$ of the $C^i$'s in $\Alg_k$, then one compute a colimit
$S$ of the $S^i$'s and equips it with the evident source and target operations
$\gsrc_k,\gtgt_k \co S \to C_k$.

\paragraph{Categories of polygraphs}

Let~$n \in \N \cup \set\omega$ and~$(T,\eta,\mu)$ be a finitary monad on the
category~$\nGlob n$. For~$k \in \N_n$, we define the \glossary(Polk){$\nPol
  k$}{the category of $k$\polygraphs}category~$\nPol k$ of
\index{polygraph}\emph{$k$\polygraphs} by induction on~$k$, together with a
\glossary(.ba){$\freecat[k]-$, $\freecat{\P}$}{the functor which maps a
  $k$\polygraph to the associated free $k$\category}functor
\[
  \freecat[k] {(-)} \co\nPol k \to \Alg_k
\]
simply denoted~$\freecat {(-)}$ when there is no ambiguity on~$k$, which maps a
$k$\polygraph~$\P$ to the \index{free category!on a polygraph}\emph{free
  $k$\category}~$\freecat\P$ on~$\P$. First, we put
\[
  \nPol0 = \nGlob 0
  \qtand
  \freecat[0] {(-)} = \freealgf_0\zbox.
\]
Now suppose that~$\nPol k$ and~$\freecat[k]{(-)}$ are defined for some~$k \in
\N_{n-1}$. We define~$\nPol{k+1}$ as the \glossary(Ek){$\poltoce_{k+1}$}{the
  canonical functor which maps a $(k{+}1)$\polygraph to a $k$\cellular
  extension}pullback
\[
  \begin{tikzcd}[cramped,column sep=huge]
    \nPol{k+1}\ar[d,"{\poltruncf[k+1] k {(-)}}"',dotted]\ar[r,"\poltoce_{k+1}",dotted]
    &\Alg_k^+\ar[d,"\cetoalg_{k}"]\\
    \nPol{k}\ar[r,"{\freecat[k]{(-)}}"']&\Alg_k
  \end{tikzcd}
\]
and~$\freecat[k+1]{(-)}$ as the composite
\[
  \begin{tikzcd}[cramped,column sep=huge]
    \nPol{k+1} \ar[r,"\poltoce_{k+1}"]
    &
    \Algp_k \ar[r,"{\polextp[k] - -}"]
    &
    \Alg_{k+1}
  \end{tikzcd}
  \zbox.
\]
As usual, we write~$\restrict k \P$ for the image of~$\P \in \Pol_{k+1}$
by~$\poltruncf[k+1] k {(-)}$, and we often simply write~$\poltruncf k {(-)}$ for
the latter \glossary(poltruncf){$\poltruncf k {(-)}$, $\restrict k \P$}{the
  truncation functor for polygraphs}functor.

Using the simpler definition of~$\Algp_k$ from \Cref{prop:algp-simpler-def}, we
can give a more concrete description of~$\nPol k$ for~${k \in \N_n}$. A
$0$\polygraph~$\P$ is the data of a set~$\P_0$
of~\index{generator!of a polygraph}\emph{$0$\generators}, and a morphism~$\P \to
\P'$ in~$\nPol 0$ is the data of a function~$F_0 \co \P_0 \to \P'_0$. Given~$k
\in \N_{n-1}$, a $(k{+}1)$\polygraph is the data of a pair
\[
  \P = (\restrict k \P,\P_{k+1})
\]
where~$\restrict k \P$ is a $k$\polygraph and~$\P_{k+1}$ is a set of
\emph{$(k{+}1)$\generators}, together with functions
\[
  \gsrc_k,\gtgt_k\co
  \P_{k+1} \to (\freecat{(\restrict k \P)})_{k}
\]
such that
\[
  \csrctgt\eps_{k-1} \circ \gsrc_k = \csrctgt\eps_{k-1}
  \circ \gtgt_k
\]
for~$\eps \in \set{-,+}$, where~$\csrc_{k-1},\ctgt_{k-1} \co
(\freecat{(\restrict k \P)})_{k} \to (\freecat{(\restrict k \P)})_{k-1}$ are the
source and target operations of the $k$\category~$\freecat {(\restrict k \P)}$.
Moreover, a morphism~$\P \to \P'$ in~$\nPol{k+1}$ is the data of a
pair~$(\restrict k F,F_{n+1})$ where~$\restrict k F \co \restrict k \P \to
\restrict k {\P'}$ is a morphism of~$\nPol k$ and~$F_{n+1} \co \P_{n+1} \to
\P'_{n+1}$ is a function such that
\[
  \gsrctgt\eps_k \circ F_{n+1} = \freecat {(\restrict k F)} \circ \gsrctgt\eps_k
\]
for~$\eps \in \set{-,+}$, \ie a $(k{+}1)$\generator~$g$ is mapped by~$F_{n+1}$
to a generator~$g'$ whose $k$\source and $k$\target are exactly the images of
the $k$\source and $k$\target of~$g$ by~$\freecat{(\restrict k F)}$.
\begin{remark}
Note that the diagram
\begin{equation}
  \label{eq:pol-pullback}
  \begin{tikzcd}[column sep=4em]
    \nPol {k+1} \ar[r,"\cetoglob_{k+1}\poltoce_{k+1}"]
    \ar[d,"\poltruncf {k} {(-)}"'] & \nGlob {k+1} \ar[d,"\gtruncf k {(-)}"] \\
    \nPol k \ar[r,"{\fgfalgf_k\freecat[k]-}"']& \nGlob k
  \end{tikzcd}
\end{equation}
is a pullback, since~$\Algp_k$ is defined as a pullback and the concatenation of
two pullbacks is still a pullback.
\end{remark}
\noindent In order to better handle
side conditions, we use the convention that
\[
  \Algp_{-1} = \nGlob 0,
  \quad
  \poltoce_{0} = \unit {\nGlob 0},
  \qtand
  \polextp[0] - - = \freealgf_0
\]
so that~$\freecat[0] {(-)} = \polextp[0] - - \circ \poltoce_0$. We then have:
\begin{prop}
  \label{prop:polk-omega1-pres}
  For~$k \in \N_n$, the category $\nPol k$ is locally $\omega_1$\presentable (in
  particular, complete and cocomplete), and the functor $\poltoce_k$ (\resp
  $\poltruncf {k-1} -$ when $k > 0$) is a left adjoint which preserves
  $\omega_1$\presentable objects.
\end{prop}
\begin{proof}
  We prove this property by induction on $k$. The category $\nPol 0$ is locally
  finitely presentable since equivalent to $\Set$. Moreover, since $\fgfalgf_0$
  is finitary, it is also $\omega_1$\accessible so that $\freealgf_0$ preserves
  finitely presentable objects. Thus, the property holds for $k = 0$.

  Now assume that it holds in dimension $k$. By
  \Cref{prop:polextp-pres-presentable}, the functor $\polextp[k]--$ is a left
  adjoint which preserves finitely presentable objects. Thus, its right adjoint
  is finitary and in particular $\omega_1$\accessible. Thus, $\polextp[k]--$
  also preserves $\omega_1$\presentable objects. By induction hypothesis,
  $\freecat[k]-$ is a left adjoint which preserves $\omega_1$\presentable
  objects. By \Cref{prop:cetoalg-pres-presentables} and the same argument,
  $\cetoalg_k$ is also a left adjoint which preserves $\omega_1$\presentable
  objects. By \cite[Proposition~3.14]{bird}, the (bi)pullback $\nPol{k+1}$ of
  these two functors is a locally $\omega_1$\presentable category, and the two
  projecting functors $\poltruncf k -$ and $\poltoce_{k+1}$ are left adjoints
  which preserve $\omega_1$\presentable objects.
\end{proof}

\begin{rem}
  \label{rem:finite-pres-polk}
  Sadly, \cite[Proposition~3.14]{bird} requires us to deal with uncountable
  cardinals, so that we can not use it to prove that $\nPol k$ is locally
  finitely presentable. But one can use an \emph{ad hoc} argument as in
  \cite[Paragraph~1.3.3.16]{forest:tel-03155192} to prove that it actually is.
  Moreover, the finitely presentable objects of $\nPol k$ are the expected ones,
  that is, the $k$\polygraphs with a finite number of generators.
\end{rem}
\simon{utiliser Bird ici est compliqué car 1) il faut que $\Algp_k$ est
  présentable (démo ad hoc pour l'instant) 2) il faut que les foncteurs
  préservent les objets présentables. je vais rester sur cocomplet}
\begin{prop}
  \label{prop:poltruncf-la-isofib}
  For~$k \in \N_n$, the functor~$\poltruncf[k+1] {k} {(-)}$ is an isofibration
  which has both a left adjoint and a right adjoint.
\end{prop}
\begin{proof}
  We know already that it has a right adjoint by \Cref{prop:polk-omega1-pres}.
  But it is also the consequence, as is the remainder of the statement, of
  \Cref{prop:pullback-CAT-iso-liftting-colimits,prop:pullback-la-ra}.
\end{proof}
\noindent Given~$i,k \in \N$ such that~$i \le k < n+1$, we
\glossary(.ca){$\poltogens[k] i -$, $\P_i$}{the functor which maps a
  $k$\polygraph to its set of $i$\generators}write
\[
  \poltogens[k] i - \co \nPol k \to \Set
\]
or simply~$\poltogens i -$ when there is no ambiguity on~$k$, for the functor
which maps a $k$\polygraph~$\P$ to its set of~$i$\generators~$\P_i$. We verify
that colimits of polygraphs are computed dimensionwise:
\begin{prop}
  \label{prop:poltogens-colimits}
  Given~$i,k \in \N_n$ such that~$i \le k$, the functor~$\poltogens[k] i -$
  preserves colimits.
\end{prop}
\begin{proof}
  This functor can be described as the composite
  \[
    (-)_g \circ \poltoce_i \circ \poltruncf i - \circ \cdots \circ \poltruncf
    {k-1} -
  \]
  where $(-)_g \co \Algp_{i-1} \to \Set$ is the functor mapping a cellular
  extension $(X,S)$ to its set of $i$\generator $S$. By our comment on the
  computation of colimit using the concrete description of $\Algp_{i-1}$, the
  functor $(-)_g$ preserves colimits. Moreover, all the other functors of the
  above composition preserve colimits by \Cref{prop:polk-omega1-pres}, so that
  the conclusion holds.
\end{proof}
\noindent In the case where~$T$ is truncable, given~$k,l \in \N$ with~$k < l$,
the underlying $k$\category of the free $l$\category on an $l$\polygraph is only
determined by the underlying $k$\polygraph, as stated by the following proposition:
\begin{prop}
  \label{prop:algtruncf-pol-isom}
  If~$T$ is truncable, then, given~$k \in \N$ such that~$k < n$ and a
  $(k{+}1)$\polygraph, there exists an isomorphism~$\restrict k {(\freecat \P)}
  \cong \freecat {(\restrict k \P)}$.
\end{prop}
\begin{proof}
  By definition of~$\freecat[k]{-}$, we have
  \[
    \freecat\P = \polextp{\freecat{(\restrict k \P)}}{\P_{k+1}}
  \]
  so that the wanted isomorphism comes from \Cref{prop:algtruncf-polextp-isom}.
\end{proof}
\begin{remark}
  \label{rem:poltruncf-freecat-strict}
  When~$T$ is truncable, under the assumption of \Remr{algtruncf-strict}, the
  isomorphism given by \Propr{algtruncf-pol-isom} is the identity. This enables
  to simplify some notations: given~$k,l \in \N_n$ with~$k \le l$ and an
  $l$\polygraph~$\P$, we write directly~$\restrict k{\freecat\P}$ for
  both~$\restrict k {(\freecat \P)}$ and~$\freecat {(\restrict k \P)}$,
  and~$\freecat \P_k$ for both~$(\freecat \P)_k$ and~$(\freecat {(\restrict k
    \P)})_k$.
\end{remark}

\begin{remark}
  When~$T$ is truncable, given~$k \in \N_n$, a $k$\polygraph~$\P$ can be
  alternatively described as a diagram in~$\Set$ of the form
  \[
    \begin{tikzcd}[column sep=10ex,labels={inner sep=0.5pt}]
      \P_0\ar[d,"\polinj0"{inner sep=2pt}]
      &\P_1
      \ar[dl,shift right,"\gsrc_0"',pos=0.3]
      \ar[dl,shift left,"\gsrc_0",pos=0.3]\ar[d,"\polinj1"{inner sep=2pt}]
      &\P_2\ar[dl,shift right,"\gsrc_1"',pos=0.3]\ar[dl,shift left,"\gsrc_1",pos=0.3]\ar[d,"\polinj2"{inner sep=2pt}]
      &\ldots
      &\P_{k-1}
      \ar[dl,shift right,"\gsrc_{k-2}"',pos=0.3]
      \ar[dl,shift
      left,"\gsrc_{k-2}",pos=0.3]\ar[d,"\polinj{k}"{inner sep=2pt}]
      &\P_{k}\ar[dl,shift right,"\gsrc_{k-1}"',pos=0.3]\ar[dl,shift left,"\gsrc_{k-1}",pos=0.3]\\
      \freecat{\P_0}
      &
      \freecat{\P_1}
      \ar[l,shift right,"{\csrc_0}"']
      \ar[l,shift
      left,"{\ctgt_0}"]
      &\ldots
      \ar[l,shift
      right,"{\csrc_1}"']\ar[l,shift
      left,"{\ctgt_1}"]
      &\freecat{\P_{k-2}}
      &\ar[l,shift
      right,"{\csrc_{k-2}}"']\ar[l,shift
      left,"{\ctgt_{k-2}}"]\freecat{\P_{k-1}}
    \end{tikzcd}
  \]
  where, for~$i \in \N_{k-1}$,~$\polinj i$ is the embedding of the
  $i$\generators in the $i$\cells induced by the unit of the
  adjunction~$\polextp[i]--\dashv\algtoce_{i-1}$ at~$(\freecat{(\restrict {i-1}
    \P)},\P_i)$, such that
  \[
    {\csrc_i}\circ\gsrc_{i+1}={\csrc_i}\circ\gtgt_{i+1}
    \qqtand
    \ctgt_i\circ\gsrc_{i+1}=\ctgt_i\circ\gtgt_{i+1}
  \]
  for~$i \in \N_{k-1}$. The above description of polygraphs can already be found
  in the original paper of Burroni~\cite{burroni1993higher} for polygraphs of
  strict categories.
\end{remark}

\paragraph[\texorpdfstring{$\omega$}{omega}-polygraphs]{$\bm\omega$-polygraphs}

Let~$(T,\eta,\mu)$ be a finitary monad on~$\nGlob \omega$. We define the
category of \glossary(Polomega){$\nPol\omega$}{the category of
  $\omega$\polygraphs}$\omega$\polygraphs~$\nPol\omega$ as the limit in~$\CAT$
\[
  (\poltruncf[\omega] k {(-)} \co \nPol\omega \to \nPol k)_{k \in \N}
\]
on the diagram
\[
  \begin{tikzcd}[column sep=3.5em,cramped]
    \nPol 0
    &
    \nPol 1
    \ar[l,"\poltruncf 0 {(-)}"']
    &
    \cdots
    \ar[l,"\poltruncf 1 {(-)}"']
    &
    \nPol {k}
    \ar[l,"\poltruncf {k-1} {(-)}"']
    &
    \nPol {k+1}
    \ar[l,"\poltruncf k {(-)}"']
    &
    \cdots
    \ar[l,"\poltruncf {k+1} {(-)}"']
  \end{tikzcd}
\]
Concretely, an $\omega$\polygraph~$\P$ is the data of a sequence~$(\P^k)_{k \in
  \N}$, where~$\P^k$ is a $k$\polygraph, such that~$\restrict k {(\P^{k+1})} =
\P^k$ for~$k \in \N$. We start with a presentability result for $\oPol$:
\begin{prop}
  \label{prop:polomega-omega1-pres}
  The category $\oPol$ is locally $\omega_1$\presentable (in particular complete
  and cocomplete), and the functors $\poltruncf[\omega]k-$ are both left and
  right adjoints and preserve $\omega_1$\presentable objects.
\end{prop}
\begin{proof}
  By \Cref{prop:polk-omega1-pres}, the functors $\poltruncf[k+1]k-$ are
  isofibrations, so that the limit defining $\oPol$ is in fact a bilimit. Since
  $\poltruncf[k+1]k-$ preserves colimits as left adjoints, they are
  $\omega_1$\accessible right adjoints. By \cite[Theorem~2.17]{bird}, $\oPol$ is
  locally $\omega_1$\presentable and the functors $\poltruncf[\omega]k-$ are
  $\omega_1$\accessible right adjoint functors.

  The functors $\poltruncf[k+1]k-$ are also left adjoints which preserve
  $\omega_1$\presentable objects by \Cref{prop:polk-omega1-pres}. Thus,
  \cite[Proposition~3.14]{bird} also applies, so that we moreover get that
  $\poltruncf[\omega]k-$ are left adjoints which preserve $\omega_1$\presentable
  objects.
\end{proof}
\begin{rem}
  Assuming the finite presentability of the $\nPol k$ given by
  \Cref{rem:finite-pres-polk}, we can apply \cite[Theorem~2.17]{bird} to deduce
  that $\oPol$ is in fact locally finitely presentable. A small additional
  argument would then prove that the finitely presentable objects of $\nPol k$
  are the \opol with a finite number of generators. A proof of these facts
  without using Bird can be found
  in~\cite[Paragraph~1.3.3.16]{forest:tel-03155192}.
\end{rem}

\noindent Now, in the truncable case, we can easily define the free
$\omega$\category on an $\omega$\polygraph, just like for finite-dimensional
polygraphs:
\begin{prop}
  \label{prop:freecat-infty}
  If~$T$ is truncable, there is a functor~$\freecat[\omega] - \co \nPol\omega
  \to \Alg_{\omega}$ which is uniquely defined by
  \[
    \algtruncf[\omega] k {(-)}
    \circ \freecat[\omega]- = \freecat[k]- \circ \poltruncf[\omega] k -
  \]
  for~$k \in \N$.
\end{prop}
\begin{proof}
  By \Remr{algtruncf-strict} and \Remr{poltruncf-freecat-strict}, we have a commutative diagram
  \[
    \begin{tikzcd}[row sep=3em,column sep=4em]
      \nPol 0
      \ar[d,"{\freecat[0]-}"{description}]
      &
      \nPol 1
      \ar[l,"{\poltruncf 0 -}"']
      \ar[d,"{\freecat[1]-}"{description}]
      &
      \cdots
      \ar[l,"{\poltruncf 1 -}"']
      &
      \nPol k
      \ar[l,"{\poltruncf {k-1} -}"']
      \ar[d,"{\freecat[k]-}"{description}]
      &
      \nPol {k+1}
      \ar[l,"{\poltruncf k -}"']
      \ar[d,"{\freecat[k+1]-}"{description}]
      &
      \cdots
      \ar[l,"{\poltruncf {k+1} -}"']
      \\
      \Alg_0
      &
      \Alg_1
      \ar[l,"{\algtruncf 0 -}"]
      &
      \cdots
      \ar[l,"{\algtruncf 1 -}"]
      &
      \Alg_k
      \ar[l,"{\algtruncf {k-1} -}"]
      &
      \Alg_{k+1}
      \ar[l,"{\algtruncf k -}"]
      &
      \cdots
      \ar[l,"{\algtruncf {k+1} -}"]
    \end{tikzcd}
  \]
  which, by the definition of~$\nPol\omega$ and \Propr{alg-infty-limit}, induces
  a functor~$\freecat[\omega]-$ which satisfies the wanted properties.
\end{proof}
\begin{rem}
  In the case where $T$ is only weakly truncable, the squares in the proof of
  \Cref{prop:freecat-infty} only commute up to isomorphism, so that we only get
  a \emph{pseudocone} on the $\Alg_k$'s of vertex $\oPol$. In this case, we can
  use the fact that $\Alg_\omega$ is a bilimit on the $\Alg_k$'s (since the
  functors $\algtruncf[k+1] k -$ are isofibrations by
  \Cref{prop:algtruncf-lifts-isom}) to get a functor $\freecat[\omega]- \co
  \oPol \to \Alg_\omega$ which factorizes up to isomorphism that pseudocone.
\end{rem}
\begin{remark}
  We can still define a functor~$\freecat[\omega]- \co \nPol\omega \to
  \Alg_\omega$ in the case where~$T$ is not weakly truncable. However, this
  functor is not expected to be compatible with the functors~$\algtruncf k -$ as
  in \Propr{freecat-infty}. Indeed, in this case, the functor~$\polextp[k]--$
  does not preserve the underlying $k$\category~$C$ of a $k$\cellular
  extension~$(C,S) \in \Algp_k$.
\end{remark}


\subsection{The polygraphic adjunctions}
\label{text:pol-adj}

We now translate in our setting the definition given by Batanin of the
adjunction between globular algebras and polygraphs. In our case, this
adjunction will be derived from one between cellular extensions and polygraphs
that we are going to introduce.

The constructions of this section are done using an induction on $k \in \N$ to
build a functor
\[
  \algptopol_{k} \co \Algp_k \to \nPol k
\]
which is part of an adjunction $\poltoce_k \dashv \algptopol_{k}$, with unit
$\algppolunit k$ and counit $\freealgpcu k$, from which we derive a functor
$\algtopol_{k} \co \Alg_k \to \nPol k$ which is part of an adjunction
$\freecat[k]- \dashv \algtopol_{k}$, with unit $\algpolunit k$ and counit
$\freecatcu k$.

When $k = 0$, this is easy: we take $\algptopol_{k} = \catunit[\nGlob 0]$ and
$\algtopol_0 = \fgfalgf_0$. So now, we assume that we have defined the right
adjoints $\algptopol_{l}$ and $\algtopol_{l}$ up to dimension $k \in \N$, and we
define the right adjoints $\algptopol_{k+1}$ and $\algtopol_{k+1}$ starting with
the former.

By the pullback definition of $\nPol {k+1}$, this will require defining a
functor $\Algp_{k} \to \nPol {k}$ and a functor $\Algp_{k} \to \Algp_{k}$, the
latter being reasonably suspected to be the functor of the comonad induced by
the adjunction.

The functor $\Algp_{k} \to \nPol {k}$ is defined as expected as the composite
\[
  \Algp_{k} \xto{\cetoalg_k} \Alg_k \xto{\algtopol_k} \nPol k
  \zbox.
\]

We define the functor $\Algp_k \to \Algp_k$, denoted $\algpcom_k$, as follows:
given $(C,S) \in \Algp_k$, $\algpcom_k(C,S)$ is the pullback
\[
  \begin{tikzcd}[sep=large]
    \algpcom_k(C,S)
    \ar[r,dashed,"q^R_{(C,S)}"]
    \ar[d,dashed,"q^L_{(C,S)}"']
    \phar[rd,very near start,"\drcorner"]
    &
    (C,S)
    \ar[d,"\algpifunit k_{(C,S)}"]
    \\
    \incfill{k+1}{(\freecat{(\algtopol_k C)})}
    \ar[r,"(\algpincfill{k+1}- \freecatcu k)_C"']
    &
    \incfill{k+1}C
  \end{tikzcd}
\]
where $\algpifunit k$ denotes the unit of the adjunction $\cetoalg_k \dashv
\algpincfill[k]{k+1}-$, with $\algpincfill[k]{k+1}-$ being the right adjoint
given by \Cref{prop:cetoalg-la}, and where $\freecatcu k$ is the counit of the
adjunction $\freecat[k]- \dashv \algtopol_k$ defined by induction hypothesis.
The idea behind this pullback is the following: we forget the $k$\category $C$
as a $k$\polygraph and freely generate it back immediately to a $k$\category
(the left leg), then we attach back the generators of $S$ to this new
$k$\category, taking into account that there now several choices of cells which
evaluate to the old sources and targets of the elements of $S$ (the right leg).

Since $\algptruncf k -$ preserves limits (as a right adjoint), the diagram
\[
  \begin{tikzcd}[sep=large]
    \restrict k {(\algpcom_k(C,S))}
    \ar[r,dashed,"\restrict k {(q^R)}"]
    \ar[d,dashed,"\restrict k {(q^L)}"']
    \phar[rd,very near start,"\drcorner"]
    &
    C
    \ar[d,"{\unit C}"]
    \\
    \freecat{(\algtopol_k C)}
    \ar[r,"\freecatcu k_C"']
    &
    C
  \end{tikzcd}
\]
is a pullback (we picked $\algpincfill[k]{k+1}-$ so that the counit of
$\cetoalg_k \dashv \algpincfill[k]{k+1}-$ is the identity, which is possible by
\Cref{prop:pullback-la-ra}). Moreover, since $\algptruncf k -$ is an
isofibration, we can suppose that we defined $\algpcom_k$ so that $\restrict k
{(\algpcom_k(C,S))} = \freecat{(\algtopol_k C)}$ and $\restrict k {(q^L)} =
\unit{}$. 

Thus, we get as a factorization through $\nPol {k+1}$ of the cone made of the
above two functors a functor
\[
  \algptopol_{k+1} \co \Algp_k \to \nPol {k+1}
\]
which is characterized by $\poltruncf k -\algptopol_{k+1} = \algtopol_k\cetoalg_k$
and $\poltoce_{k+1}\algptopol_{k+1} = \algpcom_k$.

We now equip $(\poltoce_{k+1},\algptopol_{k+1})$ with a structure of an
adjunction. First, we build a unit
\[
  \algppolunit {k+1} = \prodfact{\algppolunit {k+1}_l}{\algppolunit {k+1}_r}\co
  \catunit[\nPol{k+1}] \To \algptopol_{k+1}\poltoce_{k+1}
\]
through the pullback of $\nPol{k+1}$ as
follows. The projection of $\algppolunit {k+1}$ through $\poltruncf k -$ is
$\algppolunit {k+1}_l = \algpolunit k\poltruncf k -$; this is
well-defined, since
\[
  \poltruncf k -\algptopol_{k+1}\poltoce_{k+1} =
  \algtopol_k\cetoalg_k\poltoce_{k+1} = \algtopol_k\freecat[k] -\poltruncf k -
  \zbox.
\]
Moreover, the projection of $\algppolunit {k+1}$ through $\poltoce_{k+1}$ is a
natural transformation
\[
  \algppolunit {k+1}_r\co
  \poltoce_{k+1} \To \poltoce_{k+1}\algptopol_{k+1}\poltoce_{k+1} = \algpcom_k
  \poltoce_{k+1} \co \nPol {k+1} \to \Algp_k
\]
which is defined using the pullback definition of $\algpcom_k$: we have the cone
\[
  \begin{tikzcd}[sep=large]
    \poltoce_{k+1}
    \ar[rrd,"\unit{}",bend left=20]
    \ar[rdd,"\lambda"',bend right=20]
    \ar[rd,dashed,"\algppolunit {k+1}_r"]
    &[-3em]
    &[4em]
    \\
    &
    \algpcom_k\poltoce_{k+1}
    \ar[r,"q^R"]
    \ar[d,"q^L"']
    \phar[rd,very near start,"\drcorner",xshift=-3ex]
    &
    \poltoce_{k+1}
    \ar[d,"{\algpifunit k\poltoce_{k+1}}"]
    \\
    &
    \algpincfill{k+1}-\freecat[k]-\algtopol_k \algptruncf k -\poltoce_{k+1}
    \ar[r,"\algpincfill{k+1}- \freecatcu k\algptruncf k -\poltoce_{k+1}"']
    &
    \algpincfill{k+1}-\algptruncf k -\poltoce_{k+1}
  \end{tikzcd}
\]
where $\lambda$ is defined using string diagram as
\[
  \lambda
  \qquad
  =
  \qquad
  \satex{algppol-lambda}
  \zbox.
\]
This is indeed a cone, as one can easily verify using the zigzag equations of
adjunctions that the following two string diagrams 
\[
  \satex{algppol-cone-l}
  \qquad
  \text{and}
  \qquad
  \satex{algppol-cone-r}
\]
represent the same natural transformation. Thus, we get a factorization natural
transformation $\algppolunit {k+1}_r \co \poltoce_{k+1} \To
\poltoce_{k+1}\algptopol_{k+1}\poltoce_{k+1}$.

We still need to verify that the equation $\freecat[k]-\algppolunit {k+1}_l = \algptruncf k -
\algppolunit {k+1}_r$, whose string diagrams
representation is
\[
  \satex{algppol-bxi-cone-1}
  \qquad
  =
  \qquad
  \satex{algppol-bxi-cone-2}
\]
so that the equation holds by the zigzag equations of adjunctions. Thus,
\[
  \algppolunit {k+1} = \prodfact{\algppolunit {k+1}_l}{\algppolunit {k+1}_r}\co
  \catunit[\nPol{k+1}] \To \algptopol_{k+1}\poltoce_{k+1}
\]
is well-defined.

We are now required to give a counit 
\[
  \freealgpcu {k+1} \co \poltoce_{k+1}\algptopol_{k+1} \To \catunit[\Algp_{k}]\zbox.
\]
A candidate already exists: since $\poltoce_{k+1}\algptopol_{k+1} = \algpcom_k$,
we can take $\freealgpcu {k+1} = q^R$. We now verify that we have an
adjunction:
\begin{theo}
  \label{thm:poltoce-algptopol-adj}
  There is an adjunction $\poltoce_{k+1} \dashv \algptopol_{k+1}$ with
  $\algppolunit {k+1}$ and $\freealgpcu {k+1}$ as unit and counit.
\end{theo}
\begin{proof}
  We verify the first zigzag equation, namely $(\freealgpcu {k+1}\poltoce_{k+1}) \circ
  (\poltoce_{k+1}\algppolunit {k+1}) = \unit{\poltoce_{k+1}}$:
  \begin{align*}
    (\freealgpcu {k+1}\poltoce_{k+1}) \circ (\poltoce_{k+1}\algppolunit {k+1}) 
    &
      = 
      (\freealgpcu {k+1}\poltoce_{k+1}) \circ (\algppolunit {k+1}_r) 
    \\
    &=\unit{\poltoce_{k+1}}
  \end{align*}
  by the definition $\algppolunit {k+1}_r$.

  We now verify the second zigzag equation, namely $(\algptopol_{k+1}\freealgpcu {k+1}) \circ
  (\algppolunit {k+1}\algptopol_{k+1}) = \unit{\algptopol_{k+1}}$. We first
  check that the projection along $\poltruncf k -$ is an identity using string diagrams:
  \[
    \begin{array}{cccccccc}
      &
      \satex{algppol-zigzag-l-1}
      &
        =
      &
        \satex{algppol-zigzag-l-2}
      &
        =
      &
        \satex{algppol-zigzag-l-3}
      \\
      =
      &
        \satex{algppol-zigzag-l-4}
      &
        =
      &
        \satex{algppol-zigzag-l-5}
      &
        =
      &
        \satex{algppol-zigzag-l-6}
        \\
      =
      &
        \satex{algppol-zigzag-l-7}
        \zbox.
    \end{array}
  \]
  We now check that the projection along $\poltoce_{k+1}$ is also an identity,
  \ie
  \begin{equation}
    \label{eq:zigzag-poltoce-proj}
    (\algpcom_k q^R) \circ
    (\algppolunit {k+1}_r\algptopol_{k+1}) = \unit{\algpcom_k}
    \zbox.
  \end{equation}
  As an equation between two natural transformation with codomain $\algpcom_k$,
  we check it by verifying that the projections along $q^L$ and $q^R$ are the
  same. We start with $q^L$ and use for this purpose the following property:
  \begin{lem}
    \label{lem:criterion-eq-algpincfill}
    Given two natural transformation $\alpha,\beta \co F \To \algpincfill {k+1} -
    G$ for some functors $F \co \cC \to \Algp_k$ and $G \co \cC \to \Alg_k$, we
    have $\alpha = \beta$ if and only if $\cetoalg_k \alpha = \cetoalg_k \beta$.
  \end{lem}
  \begin{proof}
    The adjunction $\cetoalg_k \dashv \algpincfill{k+1}-$ induces a correspondence
    between natural transformations of type $F \To \algpincfill {k+1} - G$ and the
    ones of type $\cetoalg_k F \To G$. Since the counit of this adjunction is an
    identity, the forward map of this correspondence is given by the application
    of $\cetoalg_k$.
  \end{proof}
  \noindent Thus, by \Cref{lem:criterion-eq-algpincfill}, $q^L$ is a cofork of
  the two sides of \eqref{eq:zigzag-poltoce-proj} if the equation
  \[
    (\algptruncf k -\algpcom_k q^R) \circ
    (\algptruncf k -\algppolunit {k+1}_r\algptopol_{k+1}) = \unit{\algptruncf k -\algpcom_k}
  \]
  holds. We compute that
  \begin{align*}
    & (\algptruncf k -\algpcom_k q^R) \circ
    (\algptruncf k -\algppolunit {k+1}_r\algptopol_{k+1})
    \\
    =\; & (\freecat[k]-\algtopol_k\algptruncf k - q^R) \circ
          (\freecat[k]-\algpolunit k\poltruncf k -\algptopol_{k+1})
    \\
    =\; & (\freecat[k]-\algtopol_k\freecatcu k\algptruncf k -) \circ
          (\freecat[k]-\algpolunit k\algtopol_k\algptruncf k -)
    \\
    =\; &
          \unit{\freecat[k]-\algtopol_k\algptruncf k -} = \unit{\algptruncf k -\algpcom_k}
          &\text{(by the zigzag equations)\zbox.}
  \end{align*}
  We proceed with verifying that $q^R$ is a cofork of the two sides of the equation:
  \begin{align*}
    & q^R \circ (\algpcom_k q^R) \circ (\algppolunit {k+1}_r\algptopol_{k+1})
    \\
    =\;
    & q^R \circ (q^R \algpcom_k) \circ (\algppolunit {k+1}_r\algptopol_{k+1})
      &\text{(by the exchange law)}
        \\
    =\;& q^R =  q^R \circ \unit{\algpcom_k}
    &\text{(by definition of $\algppolunit {k+1}_r$)\zbox.}
  \end{align*}
  Thus, by the universal property of the pullback, the
  equation~\eqref{eq:zigzag-poltoce-proj} holds. Hence, since its projections
  along $q^L$ and $q^R$ holds, the second zigzag equation holds, concluding the
  proof.
\end{proof}
\noindent Writing $\algtopol_{k+1}$ for the composite $\algptopol_{k+1}\algtoce_k$, we get:
\begin{coro}
  \label{prop:alg-pol-adj}
  There is an adjunction $\freecat[k+1]- \dashv \algtopol_{k+1} \co \Alg_{k+1}
  \to \nPol{k+1}$.
\end{coro}
\begin{proof}
  By composing the two adjunctions given by \Cref{prop:algtoce-ra} and
  \Cref{thm:poltoce-algptopol-adj}.
\end{proof}
The unit $\algpolunit {k+1}$ and counit $\freecatcu {k+1}$ are defined in the
process of composing the two adjunctions. This concludes the inductive argument
of this section.



\section{The full example of strict categories}
\label{text:strict-cats-and-precats}

In this section, we illustrate the previous constructions on the classical
example of strict categories, which is a well-known theory of higher categories.
Strict categories, as their name suggests, are a classical example of a theory
for higher categories that lies on the strict side of the strict/weak spectrum
of higher categories. As such, they do not represent faithfully the homotopical
information of topological spaces (see~\cite{simpson1998homotopy}
or~\cite{berger1999double}). Nevertheless, they admit a relatively simpler
axiomatization than weak higher categories, and can be encountered in several
situations of interest. In the following, we recall the equational definition of
strict categories and show that it is associated with a truncable monad on
globular sets using the criterions proved in the previous sections
(\Cref{thm:charact-globular-algebras} and \Cref{thm:truncable-charact}). This
allows deriving notions of cellular extensions and polygraphs with the
associated free constructions.

We start by recalling the equational definition of strict categories in
\Cref{par:ocat-def}. Then, we give the full calculation that the forgetful
functor from strict categories to globular sets is monadic in
\Cref{text:sc-monadicity}. Next, we give the boilerplate definitions of the
truncation and inclusion functors for strict categories in
\Cref{text:sc-trunc-inc-functors}. Then, we show that the categories of strict
categories that we obtain in each dimension are coherently equivalent to the
categories of globular algebras derived from a monad on $\nGlob\omega$ (the
monad of strict \ocats) in \Cref{text:sc-as-globular-algebras}. Finally, we
instantiate the free constructions introduced in \Cref{text:free-higher} in the
case of strict categories in \Cref{text:sc-free-constructions}.

\subsection{Equational definition}
\label{par:ocat-def}
Given~$n \in \N \cup \set{\omega}$, a \index{strict category}\emph{strict
  $n$\category}~$(C,\csrc,\ctgt,\unit{},\comp)$ (often simply denoted~$C$) is an
$n$\globular set~$(C,\csrc,\ctgt)$ together with, for~$k \in \N$ with~$k < n$,
identity \glossary(idaacat){$\unitp k {}$}{the identity operation for strict
  categories}operations
\[
  \unitp{k+1} {}\colon C_{k}\to C_{k+1}
\]
often writen~$\unit {}$ when there is no ambiguity on~$k$, and, for~$i,k \in
\N_n$ with~$i < k$, composition \glossary(.caa){$\comp_i$}{the composition
  operation for strict categories}operations
\[
  \comp_{i,k}\colon C_k\times_i C_k\to C_k 
\]%
often denoted~$\comp_i$ when there is no ambiguity on~$k$, which satisfy the
axioms~\ref{cat:first} to~\ref{cat:last} below. Given~$k,l \in \N_n$ such
that~${k \le l}$ and~${u \in C_k}$, we extend the notations for identity
operations and write~$\unitp{l}{}(u)$ for
\[
  \unitp{l}{}(u) = \unitp{l}{} \circ \cdots \circ \unitp{k+1}{}(u)
\]
and, for the sake of conciseness, we often write~$\unitp l u$ for~$\unitp l {}
(u)$, or even~$\unit u$ when~$l = k+1$. The axioms are the following:
\begin{enumerate}[label=(S-\roman*), ref=(S-\roman*)]
\item\label{cat:first} \label{cat:id-srctgt} for~$k \in \N_{n-1}$ and~$u \in
  C_k$,
  \[
  \csrc_k (\unitp{k+1}u) = \ctgt_k(\unitp{k+1}u)
  = u,
  \]
\item \label{cat:src-tgt}for~$i,k \in \N_n$ with~$i < k$,~$(u,v) \in C_k
  \times_i C_k$ and~$\eps \in \set{-,+}$,
  \[
    \csrctgt\eps_{k-1}(u\comp_i v) =
    \begin{cases}
      \csrctgt\eps_{k-1}(u) \comp_{i} \csrctgt\eps_{k-1} (v)&\text{if~$i < k-1$,} \\
      \csrc_{k-1}(u)&\text{if~$i=k-1$ and~$\epsilon=-$,}\\
      \ctgt_{k-1}(v)&\text{if~$i=k-1$ and~$\epsilon=+$,}
    \end{cases}
  \]
\item\label{cat:unital} for~$i,k \in \N_n$ such that~$i < k$, and~$u \in C_k$,
  \[
    \unitp{k}{} (\csrc_i (u)) \comp_{i} u = u = u \comp_{i} \unitp{k}{}
    (\ctgt_i (u)),
  \]
\item\label{cat:assoc} for~$i,k \in \N_n$ such that~$i < k$, and $i$\composable~$u,v,w \in C_k$,
  \[
    (u \comp_{i} v) \comp_{i} w = u \comp_{i} (v \comp_{i} w) ,
  \]
\item\label{cat:id-xch} for~$i,k \in \N_{n-1}$ such that~$i < k$, and~$(u,v) \in C_k
  \times_i C_k$,
  \[
    \unitp{k+1}{}(u \comp_{i} v) = \unitp{k+1}u \comp_{i} \unitp{k+1}v ,
  \]
\item\label{cat:xch}\label{cat:last} for~$i,j,k \in \N_n$ such that~$i < j < k$,
  and~$u,u',v,v' \in C_k$ such that~$u,v$ are $i$\composable, and~${u,u'}$ are
  $j$\composable, and~$v,v'$ are $j$\composable,
  \[
    (u \comp_i v) \comp_j (u' \comp_i v') = (u \comp_j u') \comp_i (v \comp_j v').
  \]
\end{enumerate}
Note that the composition that appear in Axioms~\ref{cat:unital},
\ref{cat:assoc}, \ref{cat:id-xch} and~\ref{cat:xch} are well-defined as a
consequence of Axioms~\ref{cat:id-srctgt} and~\ref{cat:src-tgt} and the
equations satisfied by the source and target operations of a globular set. The
\Axr{cat:xch} is frequently called the \index{exchange law}\emph{exchange law} of strict
categories.
\begin{example}
  Given a $2$\category~$C$ and $0$-, $1$- and $2$\globes as in the following configuration
  \[
    \begin{tikzcd}[sep=huge]
      x
      \ar[r,bend left=70,"f_1"{description},""{auto=false,name=topl}]
      \ar[r,"f_2"{description},""{auto=false,name=midl}]
      \ar[r,bend right=70,"f_3"{description},""{auto=false,name=botl}]
      &
      y
      \ar[r,bend left=70,"g_1"{description},""{auto=false,name=topr}]
      \ar[r,"g_2"{description},""{auto=false,name=midr}]
      \ar[r,bend right=70,"g_3"{description},""{auto=false,name=botr}]
      &
      z
      \ar[from=topl,to=midl,"\Downarrow\! u",phantom]
      \ar[from=midl,to=botl,"\Downarrow\! u'",phantom]
      \ar[from=topr,to=midr,"\Downarrow\! v",phantom]
      \ar[from=midr,to=botr,"\Downarrow\! v'",phantom]
    \end{tikzcd}
  \]
  we have~$(u \comp_0 v) \comp_1 (u' \comp_0 v') = (u \comp_1 u') \comp_0 (v
  \comp_1 v')$ by \Axr{cat:xch}.
\end{example}
\noindent Our definition of strict categories involves sets, but we could have
written a similar definition using classes to define \emph{large} strict
categories. For such alternative definition, we have the following classical
example:
\begin{example}
  There is a large strict $2$\category~$\Cat$ whose $0$\cells are the small
  categories, whose $1$\cells are the functors between the $1$\categories, and
  whose $2$\cells are the natural transformations between functors, and where
  the operations~$\comp_{0,1}$ is the composition of functors, and the
  operations~$\comp_{0,2}$ and~$\comp_{1,2}$ are respectively the horizontal and
  vertical compositions of natural transformations. Note that the exchange law
  \Axr{cat:xch} in this setting corresponds to the usual \emph{exchange law} for
  natural transformations.
\end{example}
\noindent Given two strict $n$\categories~$C$ and~$D$, a
\index{morphism!of strict categories}\emph{morphism}~$F$ between~$C$ and~$D$ is
the data of an $n$\globular morphism~$F \co C \to D$ which moreover satisfies
that
\begin{itemize}
\item $F(\unitp {k+1} u) = \unitp {k+1} {F(u)}$ for every~$k \in \N_{n-1}$ and~$u \in
  C_k$,
\item $F(u \comp_i v) = F(u) \comp_i F(v)$ for every~$i,k \in \N_n$ with~$i < k$
  and $i$\composable~$u,v \in C_k$.
\end{itemize}
We often call such morphisms \index{functor@$n$-functor}\emph{$n$\functors}. We
\glossary(Catn){$\nCat n$}{the category of strict $n$\categories}write~$\nCat n$ for the
category of strict $n$\categories.

\smallpar There is a functor
\[
  \ccfgff_n \co \nCat n \to \nGlob n
\]
which maps a strict $n$\category to its underlying $n$\globular set. The above
definition of strict $n$\categories directly translates into an essentially
algebraic theory (\cf \Cref{text:ess-alg-theories}), so that the
functor~$\ccfgff_n$ is induced by a morphism between the essentially algebraic
theory of $n$\globular sets (\cf \myCref[Remark~1.2.3.4]{rem:glob-ess-alg}) and
the one of strict $n$\categories. Thus, we get:
\begin{prop}
  \label{prop:catn-complete-cocomplete-ra}
  For every~$n \in \Ninf$, the category~$\nCat n$ is locally finitely
  presentable, complete and cocomplete. Moreover, the functor~$\ccfgff_n$ is a right
  adjoint which preserves directed colimits.
\end{prop}
\begin{proof}
  The category~$\nCat n$ is locally finitely presentable by
  \Cref{prop:ess-alg-iff-loc-fin-pres} and in particular cocomplete. It is
  moreover complete by \Propr{loc-pres-nice-properties}. The required properties
  on~$\ccfgff_n$ are a consequence of \Cref{prop:functor-from-theo-morphism-ra}.
\end{proof}

\subsection{Monadicity}
\label{text:sc-monadicity}

We prove here that the functors~$\ccfgff_n$ are monadic. For this purpose, we
use Beck's monadicity theorem, that we first recall quickly. Given a
category~$C$ and morphisms~$f,g \co X \to Y$ and~$h \co Y \to Z$ in~$\mcal C$,
we say that~$h$ is a \index{split coequalizer}\emph{split coequalizer of~$f$
  and~$g$} when there exist~${s \co Z \to Y}$ and~${t \co Y \to X}$ as in
\[
  \begin{tikzcd}
    X
    \ar[r,shift left,"f"]
    \ar[r,shift right,"g"']
    &
    Y
    \ar[r,"h"]
    \ar[l,bend right=70,"t"']
    &
    Z
    \ar[l,bend right=70,"s"']
  \end{tikzcd}
\]
such that~$h \circ f = h \circ g$,~$h \circ s = \unit Z$,~$f \circ t = \unit Y$,
and~$s \circ h = t \circ g$. From this data, it can be shown that~$h$ is a
coequalizer of~$f$ and~$g$. Beck's monadicity theorem is then:
\begin{theo}
  \label{thm:monadicity}
  Given a functor~$R \co C \to D$, the functor~$R$ is monadic if and only
  if the following conditions are satisfied:
  \begin{enumerateroman}
  \item $R$ is a right adjoint,
    
  \item $R$ reflects isomorphisms,
    
  \item for every pair of morphisms~$f,g \co X \to Y$ in~$ C$, if~$R(f),R(g)$ have a split coequalizer, then~${f,g}$ have a coequalizer
    which is preserved by~$R$.
  \end{enumerateroman}
\end{theo}
\begin{proof}
  See~\cite[Theorem~4.4.4]{borceux1994handbook2} or the original work of Beck~\cite{beck1967triples}.
\end{proof}
\noindent We can then prove the following:
\begin{prop}
  \label{prop:sc-cat-fgf-monadic}
  Given~$n \in \Ninf$, the functor~$\ccfgff_n$ is monadic.
\end{prop}
\begin{proof}
  By \Propr{catn-complete-cocomplete-ra},~$\ccfgff_n$ is a right adjoint. Moreover,
  given a morphism
  \[
    {F \co C \to D} \in \nCat n\zbox,
  \]
  if~$F_k \co C_k \to D_k$ is a bijection for~$k \in \N_n$, then there is a
  morphism
  \[
    \finv F \co D \to C  \in \nCat n
  \]
  defined by~$(\finv F)_k = \finv {(F_k)}$ for~$k \in \N_n$, so that~$\ccfgff_n$
  reflects isomorphisms. Now, let~$F,G \co X \to Y$ be two morphisms of~$\nCat
  n$ such that there exist~$Z \in \nGlob n$, and morphisms
  \[
    H \co \ccfgff_n Y \to Z
    ,\quad
    S \co Z \to \ccfgff_n Y
    \qtand
    T \co \ccfgff_n Y
    \to \ccfgff_n X
  \]
  of~$\nGlob n$, as in
  \[
    \begin{tikzcd}
      \ccfgff_nX
      \ar[r,shift left,"\ccfgff_n(F)"]
      \ar[r,shift right,"\ccfgff_n(G)"']
      &
      \ccfgff_nY
      \ar[r,"H"]
      \ar[l,bend right=80,"T"']
      &
      Z
      \ar[l,bend right=80,"S"']
    \end{tikzcd}
  \]
  that witness that~$\ccfgff_n(F),\ccfgff_n(G)$ is a split coequalizer. We prove that~$F,G$
  has a coequalizer which is preserved by~$\ccfgff_n$. For this purpose, we shall
  equip~$Z$ with a structure of a strict $n$\category. For~$i,k \in \N_n$ with~$i < k$
  and~$(u,v) \in Z_k \times_i Z_k$, we put
  \[
    u \comp_i v = H(S(u) \comp_i S(v))
  \]
  and, given~$k \in \N_{n-1}$ and~$u \in C_k$, we put
  \[
    \unitp {k+1} u = H(\unitp {k+1} {S(u)})
  \]
  We verify that the axioms of strict $n$\categories are verified. Let~$k \in
  \N_{n-1}$,~$u \in Z_k$ and~$\eps \in \set{-,+}$. We have
  \begin{align*}
    \csrctgt\eps_k(\unitp {k+1} u) &=
    \csrctgt\eps_k(H(\unitp {k+1} {S(u)}))
    \\
    &=H(\csrctgt\eps_k(\unitp {k+1} {S(u)}))
    \\
    &=H(S(u)) = u
  \end{align*}
  so that \Axr{cat:id-srctgt} is satisfied. Now, let~$i,k \in \N_{n}$
  such that~$i < k$,~$(u,v) \in Z_k \times_i Z_k$ and~$\eps \in
  \set{-,+}$. We have
  \begin{align*}
    \csrctgt\eps_{k-1}(u \comp_i v)
    &= H(\csrctgt\eps_{k-1}(S(u)\comp_i S(v)))
    \\
    &=
    \begin{cases}
      H(\csrctgt\eps_{k-1}(S(u)) \comp_{i} \csrctgt\eps_{k-1} (S(v)))&\text{if~$i < k-1$,}\\
      H(\csrc_{k-1}(S(u)))&\text{if~$i = k-1$ and~$\epsilon=-$,}\\
      H(\ctgt_{k-1}(S(v)))&\text{if~$i = k-1$ and~$\epsilon=+$,}
    \end{cases}
  \end{align*}
  so that, by reducing the last expressions, we see that \Axr{cat:src-tgt}
  is satisfied. Now, let~$i,k \in \N_n$ such that~$i < k$, and~$u \in Z_k$.
  We have
  \begingroup
  \allowdisplaybreaks
  \begin{align*}
    \unitp k {}(\csrc_i(u)) \comp_i u
    &=
      H ( S( H(\unitp k {S(\csrc_i(u))}) ) \comp_i S(u) )
    \\
    &=
      H ( S( H(\unitp k {\csrc_i(S(u))}) ) \comp_i SHS(u) )
    \\
    &=
      H ( GT(\unitp k {\csrc_i(S(u))})  \comp_i GTS(u) )
    \\
    &=
      HG ( T(\unitp k {\csrc_i(S(u))})  \comp_i TS(u) )
    \\
    &=
      HF ( T(\unitp k {\csrc_i(S(u))})  \comp_i TS(u) )
    \\
    &=
      H ( FT(\unitp k {\csrc_i(S(u))})  \comp_i FTS(u) )
    \\
    &=
      H ( \unitp k {\csrc_i(S(u))}  \comp_i S(u) )
    \\
    &= H(S(u)) = u
  \end{align*}
  \endgroup and, similarly,~$u \comp_i \unitp k {}(\ctgt_i(u)) = u$, so that
  \Axr{cat:unital} holds. Now, let~$i,k \in \N_n$ such that~$i < k$, and
  $i$\composable~$u,v,w \in C_k$. We have \begingroup \allowdisplaybreaks
  \begin{align*}
    (u \comp_i v) \comp_i w
    &=H(S(H(S(u) \comp_i S(v))) \comp_i S(w))
    \\
    &=H(SH(S(u) \comp_i S(v)) \comp_i SHS(w))
    \\
    &=H(GT(S(u) \comp_i S(v)) \comp_i GTS(w))
    \\
    &=HG(T(S(u) \comp_i S(v)) \comp_i TS(w))
    \\
    &=HF(T(S(u) \comp_i S(v)) \comp_i TS(w))
    \\
    &=H(FT(S(u) \comp_i S(v)) \comp_i FTS(w))
    \\
    &=H((S(u) \comp_i S(v)) \comp_i S(w))
    \\
    &=H(S(u) \comp_i S(v) \comp_i S(w))
  \end{align*}
  \endgroup and, similarly,~$u \comp_i (v \comp_i w) = H(S(u) \comp_i S(v)
  \comp_i S(w))$. So that \Axr{cat:assoc} is satisfied. Axioms~\ref{cat:id-xch}
  and~\ref{cat:xch} are proved similarly, so~$Z$ is equipped with a structure of
  a strict $n$\category.

  \medskip\noindent We now verify that~$H$ is a strict $n$\category morphism. Given~$k
  \in \N_{n-1}$ and~$u \in Y_k$, we have
  \begin{align*}
    \unitp k {H(u)}
    = H(\unitp k {SH(u)})
    = H(\unitp k u)
  \end{align*}
  and, given~$i,k \in \N_n$ with~$i < k$, and~$(u,v) \in Y_k \times_i Y_k$, we
  have
  \begin{align*}
    H(u) \comp_i H(v)
    &= H(SH(u) \comp_i SH(v))
    \\
    &= H(GT(u) \comp_i GT(v))
    \\
    &= HG(T(u) \comp_i T(v))
    \\
    &= HF(T(u) \comp_i T(v))
    \\
    &= H(FT(u) \comp_i FT(v))
    \\
    &= H(u \comp_i v)
  \end{align*}
  so that~$H$ is a strict $n$\category morphism.

  \medskip\noindent We now prove that~$H$ is the coequalizer of~$F$ and~$G$
  in~$\nCat n$. Let~$K \co Y \to W$ be an $n$\functor such that~${KF = KG}$.
  Then, since~$H$ is the coequalizer of~$\ccfgff_n(F)$ and~$\ccfgff_n(G)$, there is a unique
  morphism
  \[
    K' \co \ccfgff_nZ \to \ccfgff_nW
  \]
  of~$\nGlob n$ such that~$K'H = K$. We are only left to prove that~$K'$ is an
  $n$\functor. First, note that we have
  \[
    K' = K'H S = K S
    \qtand
    KSH = KGT = KFT = K\zbox.
  \]
  Now, given~$k \in \N_{n-1}$ and~$u \in C_k$, we have
  \begin{align*}
    KS(\unitp {k+1} u)
    &= KSH(\unitp {k+1}{S(u)})
    \\
    &= K(\unitp {k+1}{S(u)})
    \\
    &= \unitp {k+1}{KS(u)}\zbox.
  \end{align*}
  Moreover, given~$i,k \in \N_n$ with~$i < k$, and~$(u,v) \in C_k \times C_k$,
  we have
  \begin{align*}
    KS(u \comp_i v)
    &= KSH(S(u) \comp_i S(v))
    \\
    &= K(S(u) \comp_i S(v))
    \\
    &= KS(u) \comp_i KS(v)\zbox,
  \end{align*}
  so that~$K'$ is an $n$\functor. Hence,~$H$ is the coequalizer in~$\nCat n$
  of~$F$ and~$G$. We can conclude with~\Thmr{monadicity}.
\end{proof}

\subsection{Truncation and inclusion functors}
\label{text:sc-trunc-inc-functors}

Let~$k,l \in \Ninf$ such that~$k < l$. There is a truncation
\glossary(catatruncf){$\sctruncf[l] k -$, $\restrict k C$}{the truncation functor
for strict categories}functor
\[
  \sctruncf[l] k - \co \nCat l \to \nCat k
\]
which maps a strict $l$\category~$C$ to its evident underlying strict $k$\category,
denoted~$\restrict k C$, and called the \index{truncation!of a strict category}\emph{$k$\truncation of~$C$}.

\medskip\noindent Conversely, there is an inclusion
\glossary(catbinclusion){$\scincf[k] l -$, $\incf l C$}{the inclusion functor for
  strict categories}functor
\[
  \scincf[k] l - \co \nCat k \to \nCat l
\]
which maps a strict $k$\category~$C$ to the strict $l$\category~$\incf l C$,
called the \index{inclusion!of a strict category}\emph{$l$\inclusion of~$C$}, and defined by
\[
  \restrict k {(\incf l C)} = C
  \qtand
  (\incf l C)_m = C_k
\]
for~$m \in \N_l$ with~$k < m$, and such that
\begin{itemize}
\item for~$m \in \N_{l-1}$ with~$k \le m$ and~$u \in (\incf l C)_{m+1}$,~$\csrc_m(u) = \ctgt_m(u) =  u$,
  
\item for~$m \in \N_{l-1}$ with~$k \le m$ and~$u \in (\incf l C)_m$,~$\unitp {m+1} u = u$,
  
\item for~$i,m \in \N_l$ with~$i < k < m$ and~$(u,v) \in (\incf l C)_m \times_i
  (\incf l C)_m$,~$u \comp_{i,m} v = u \comp_{i,k} v$,
\item for~$i,m \in \N_l$ with~$k \le i < m$ and~$(u,v) \in (\incf l C)_m
  \times_i (\incf l C)_m$,~$u
  \comp_{i,m} v = u = v$.
\end{itemize}
There is an adjunction~$\scincf[k] l - \dashv \sctruncf[l] k -$ whose unit is
the identity and whose counit~$\sctrunccu[l] k$ is such that, given a strict
$l$\category~$C$, the $l$\functor~$\sctrunccu[l] k_C \co \incf l {(\restrict k
  C)} \to C$ is defined by~$\restrict k {(\sctrunccu[l] k_C)} = \unit {\restrict
  k C}$ and, for~$m \in \N_l$ with~$m > k$,~$\smash{\sctrunccu[l] k_C}$ maps~$u \in
(\incf l {(\restrict k C)})_m = C_k$ to~$\unitp m u$.

\subsection{Strict categories as globular algebras}
\label{text:sc-as-globular-algebras}

By \Propr{catn-complete-cocomplete-ra}, each functor~$\ccfgff_n$ admits a left
adjoint~$\ccfreef_n$ for~$n \in \Ninf$. In particular, the
adjunction~$\ccfreef_\omega \dashv \ccfgff_\omega$ defines a
monad~$(T,\eta,\mu)$, which is finitary by \Propr{catn-complete-cocomplete-ra},
and it induces categories of algebras~$\Alg_n$ for~${n \in \N \cup
  \set\omega}$ as explained in \Cref{text:globular-algebras}. By
\Propr{sc-cat-fgf-monadic}, the comparison functor~$H_\omega \co \nCat \omega
\to \Alg_\omega$ is an equivalence of categories, that moreover satisfies
that~$\fgfalgf_\omega H_\omega = \ccfgff_\omega$. Using the criterion introduced
in \Cref{text:criterion-for-globular-algebras}, we prove that the other
categories~$\nCat n$ are, up to equivalence, the categories of
algebras~$\Alg_n$:
\begin{theo}
  \label{thm:str-cat-globular-algebras}
  There exists a family of equivalences
  \[
    (H_k \co \nCat k \to \Alg_k)_{k \in \N}
  \]
  making the diagrams
  \[
    \begin{tikzcd}[baseline=(mybase.base)]
      \nCat {k+1}
      \ar[r,"H_{k+1}"]
      \ar[d,"\sctruncf k -"'{name=mybase}]
      &
      \Alg_{k+1}
      \ar[d,"\algtruncf k -"]
      \\
      \nCat k
      \ar[r,"H_k"']
      &
      \Alg_k
    \end{tikzcd}
    \qqtand
    \begin{tikzcd}[baseline=(mybase.base)]
      \nCat \omega
      \ar[r,"H_\omega"]
      \ar[d,"\sctruncf k -"'{name=mybase}]
      &
      \Alg_\omega
      \ar[d,"\algtruncf k -"]
      \\
      \nCat k
      \ar[r,"H_k"']
      &
      \Alg_k
    \end{tikzcd}
  \]
  commute and such that~$\fgfalgf_k H_k = \ccfgff_k$ for every $k \in \N$.
\end{theo}
\begin{proof}
  The unit of the adjunction~$\scincf[n] \omega - \dashv \sctruncf n -$ is the
  identity, so that~$\scincf[n] \omega -$ is fully faithful and
  \Cref{thm:charact-globular-algebras} applies.
\end{proof}
\noindent Finally, we prove the truncability of the monad of strict $\omega$\categories:
\begin{theo}
  \label{thm:sc-monad-wtruncable}
  The monad~$(T,\eta,\mu)$ on~$\nGlob\omega$ derived from~$\ccfreef_\omega \dashv
  \ccfgff_\omega$ is weakly truncable.
\end{theo}
\begin{proof}
  By \Thmr{truncable-charact} and \Thmr{str-cat-globular-algebras}, it is enough
  to show that, for every~$k \in \N$, the functors~$\sctruncf[\omega] k -$ have
  right adjoints such that~$\gincfillu k \fgfalgf_\omega \scincfill[k]\omega-$
  is an isomorphism, where recall that~$\gincfillu k$ is the counit
  of~$\gtruncf[\omega]k- \dashv \gincfill[k]\omega-$. So let~$k \in \N$. Given a
  strict $k$\category~$C$, we define a strict $\omega$\category~$C'$ whose
  underlying globular set is the image the underlying $k$\globular set of~$C$
  by~$\gincfill[k] \omega -$, \ie
  \[
    \restrict k {C'} = C
    \qtand
    C'_l = \set{ (u,v) \in C^2_k \mid \text{$u,v$ are parallel}} \text{ for~$l > k$},
  \]
  and we equip~$C'$ with a structure of a strict $\omega$\category that extends
  the one on~$C$ by putting
  \[
    \unitp {k+1} u = (u,u)
    \quad
    \text{for~$u \in C_k$},
    \qquad
    \unitp {l+1} {(u,v)} = (u,v)
    \quad
    \text{for~$l \in \N$ with~$l > k$ and~$(u,v) \in C'_l$,}
  \]
  and moreover, for~${i,l \in \N}$ with~$\max(i,k) < l$ and $i$\composable~$(u,v), (u',v') \in C'_l$,
  \[
    (u,v) \comp_{i,l} (u',v') =
    \begin{cases}
      (u \comp_{i,k} u',v \comp_{i,k} v') & \text{if~$i < k$,} \\
      (u,v') & \text{if~$i \ge k$.}
    \end{cases}
  \]
  One can show that the axioms of strict $\omega$\categories are verified by~$C'$. Now, let~$D$ be a strict $\omega$\category and~$F \co \restrict k D \to
  C$ be a $k$\functor. By the properties of the adjunction~$\gtruncf[\omega]k-
  \dashv \gincfill[k]\omega-$, there is a unique $\omega$\globular morphism~$F'
  \co D \to C'$ such that~$\restrict k {F'} = F$, which is defined by
  \[
    F'(u) = (F(\csrc_k(u)),F(\ctgt_k(u)))
  \]
  for every~$l \in \N$ with~$k < l$ and~$u \in D_l$. We verify that~$F'$ is an
  $\omega$\functor by checking the compatibility with the~$\unitp {l} {}$ and~$\comp_{i,l}$ operations. Given~$l \in \N$ with~$l \ge k$ and~$u \in D_l$, we have
  \[
    F'(\unitp {l+1} u) = (F(\csrc_k(u)),F(\ctgt_k(u))) = \unitp {k+1} {F'(u)}\zbox.
  \]
  Moreover, given~$i,l \in \N$ with~$\max(i,k) < l$ and $i$\composable~$u,v \in
  D_l$, we have
  \begin{align*}
    F'(u \comp_i v) &=
    \begin{cases}
      (F(\csrc_k(u)) \comp_i F(\csrc_k(v)),F(\ctgt_k(u)) \comp_i F(\ctgt_k(v)))
      & \text{if~$i < k$}
      \\
      (F(\csrc_k(u)),F(\ctgt_k(v))) & \text{if~$i \ge k$}
    \end{cases}
    \\
    &= F'(u) \comp_i F'(v)\zbox.
  \end{align*}
  Thus,~$F'$ is an $\omega$\functor. Hence, the natural bijective correspondence
  \[
    \gtruncf[\omega] k -\co \nGlob \omega (D, C') \to \nGlob k (\restrict k D,C)
  \]
  restricts to a bijective correspondence
  \[
    \sctruncf[\omega] k -\co \nCat  \omega (D, C') \to \nCat k (\restrict k D,C)
  \]
  so that the operation~$C \mapsto C'$ extends to a functor~$\scincfill[k]
  \omega -$ which is right adjoint to~$\sctruncf[\omega] k -$. Moreover, by the
  definition of~$C'$ above, the natural morphism~$\gincfillu k \fgfalgf_\omega
  \scincfill[k]\omega-$ is an isomorphism. Hence, \Thmr{truncable-charact}
  applies and~$(T,\eta,\mu)$ is a weakly truncable monad.
\end{proof}

\begin{remark}
  We highlight that the criterions given by \Thmr{charact-globular-algebras} and
  \Thmr{truncable-charact} enabled us to prove that the categories~$\nCat n$ are
  globular algebras derived from a truncable monad on~$\nGlob \omega$ without
  giving an explicit description of this monad, which could have been a tedious
  exercise~\cite{penon1999approche}.
\end{remark}

\subsection{Free constructions}
\label{text:sc-free-constructions}

Using \Thmr{str-cat-globular-algebras} and
\Thmr{sc-monad-wtruncable}, we can instantiate the definitions and properties
developed in \myCref[Section 1.3]{text:free-higher} to define free constructions
on strict $n$\categories. In particular, for every~$n \in \N$, there is a notion
of $n$\cellular extension, with associated \glossary(Catp){$\nCatp n$}{the
  category of $n$\cellular extension for strict categories}category~$\nCatp n$ defined
like~$\Algp_n$. Moreover, there is a canonical forgetful functor~$\nCat {n+1}
\to \nCatp n$ which has a left adjoint
\[
  \polextp[n] - - \co \nCatp n \to \nCat {n+1}
\]
which can be chosen such that~$\restrict n {\polextp C X} = C$ for~$(C,X) \in
\nCatp n$. As was shown in~\cite{metayer2008cofibrant}, the $(n{+}1)$\cells of a
free extension admit a syntactical description consisting of ``well-typed''
terms considered up to the axioms of strict categories (\cf
\Cref{par:ocat-def}). 

Using the functors~$\polextp[k]--$, we can define, for every~$n \in
\Ninf$, a notion of \index{polygraph!of strict categories}\emph{$n$\polygraph}
with associated category~$\nPol n$, and a functor
\[
  \freecat[n]-\co \nPol n \to \nCat n
\]
which maps an $n$\polygraph~$\P$ to the free strict $n$\category~$\freecat\P$
induced by the generators contained in~$\P$. Note that, when~$n > 0$, as a
consequence of the compatibility of~$\polextp[n-1]--$ with truncation, the
underlying strict $(n{-}1)$\category~$\restrict {n-1} {(\freecat\P)}$
of~$\freecat\P$ is exactly~$\freecat{(\restrict {n-1} \P)}$. As before, the
cells of $\freecat\P$ admit a syntactic description as equivalence classes of
well-typed terms (see \cite{makkai2005word} or
\cite[Proposition~1.4.1.16]{forest:tel-03155192}).

\begin{example}
  Given the $1$\polygraph~$\P$ with~$\P_0 = \set x$ and~$\P_1 = \set{f \co x \to
    x}$, the strict $1$\category~$\freecat\P$ is the monoid of natural
  numbers~$(\N,0,+)$.
\end{example}

\begin{example}
  \label{ex:monoid-pol-sc}
  We define a \emph{$3$\polygraph~$\P$} that aims at encoding the structure of a
  pseudomonoid in a $2$-monoidal category as follows. We put
  \begin{align*}
    \P_0&= \set{x}
    &
    \P_1&= \set{\bar 1\co x \to x}
    &
    \P_2&= \set{\mu\co \bar 2 \To \bar 1, \eta\co \bar 0 \To \bar 1}
  \end{align*}
  where, given~$n\in\N$, we write~$\bar n$ for the composite~$\bar 1 \comp_0 \cdots \comp_0 \bar 1$ of~$n$ copies of~$\bar 1$, and we define~$\P_3$ as the
  set with the following three elements
  \[
    \setlength\arraycolsep{0pt}
    \begin{array}{clccc}
      \monL&\co &(\eta \comp_0 \unitp 2 {\bar 1}) \comp_1 \mu
      &\hspace*{1em}\TO\hspace*{1em}
      & \unitp 2 {\bar 1}  \\
      \monR&\co &(\unitp 2 {\bar 1} \comp_0 \eta) \comp_1 \mu
      &\hspace*{1em}\TO\hspace*{1em}
      & \unitp 2 {\bar 1} \\
      \monA&\co\hspace*{0.5em} 
                &(\mu \comp_0 \unitp 2 {\bar 1}) \comp_1 \mu
      &\hspace*{1em}\TO\hspace*{1em}
      &(\unitp 2 {\bar 1} \comp_0 \mu) \comp_1 \mu \pbox{.}
    \end{array}
  \]
  It is convenient to represent the $2$\cells of~$\freecat\P$ using string
  diagrams. In this representation, the $2$\generators~$\eta$ and~$\mu$ are
  represented by~$\satex{eta}$ and~$\satex{mu}$ respectively, and the $2$\cells
  of the form~$\unitp 2 {\bar n}$ are represented by sequences of~$n$ wires~$\satex{wires}$ for~$n \in \N$. Moreover, given~$u,v \in \freecat\P_2$, when~$u,v$ are $0$\composable (\resp $1$\composable), a representation of the
  $2$\cell~$u \comp_0 v$ (\resp~$u \comp_1 v$) is obtained by concatenating
  horizontally (\resp vertically) representations of~$u$ and~$v$. For example,
  using this representation, the $3$\generators~$\monL$,~$\monR$ and~$\monA$,
  can be pictured by
  \[
    \setlength\arraycolsep{0pt}
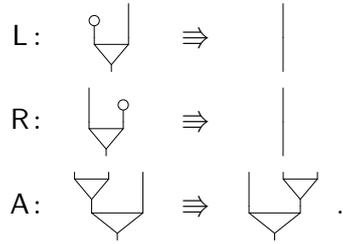
    \begin{array}{clccc}
      \monL&\co
      &\satex{mon-unit-l}
                           &\hspace*{1em}\TO\hspace*{1em}
      &  \satex{mon-unit-c}\\[12pt]
      \monR&\co
      &\satex{mon-unit-r} &\hspace*{1em}\TO\hspace*{1em}
      & \satex{mon-unit-c}\\[12pt]
      \monA&\co\hspace*{0.5em}
      &\satex{mon-assoc-l} &\hspace*{1em}\TO\hspace*{1em}
      &  \satex{mon-assoc-r}\pbox.
    \end{array}
  \]
  Note that, by \Axr{cat:xch}, a $2$\cell can admit several representations as
  string diagrams. For example, the $2$\cell
  \[
    \mu \comp_0 \unitp 2 {\bar 3} \comp_0 \mu = (\mu \comp_0 \unitp 2 {\bar 5}) \comp_1
    (\unitp 2 {\bar 4} \comp_0 \mu) = (\unitp 2 {\bar 5} \comp_0 \mu) \comp_1
    (\mu \comp_0 \unitp 2 {\bar 4})
  \]
  can be represented by the three string diagrams
  \[
    \satex{mon-mu-3-mu} \qtand \satex{mon-mu-3-mu-l} \qtand \satex{mon-mu-3-mu-r}\pbox.
  \]
\end{example}

\clearpage
\appendix

\section{Local presentability}
\label{text:loc-fin-pres-cats}

Locally presentable categories are a standard tool for deriving elementary
properties on categories of algebraic structures (monoids, groups, but also
categories, $2$\categories, \etc). They are those categories where every object
is a directed colimit of ``finitely presentable'' objects, which are a
generalization of the notions of finitely presentable monoids or groups. Knowing
that some categories are locally finitely presentable category is helpful since
those categories are complete, cocomplete and satisfy other nice properties. For
a more complete presentation, we refer to the existing
literature~\cite{gabriel2006lokal,adamek1994locally,borceux1994handbook2}.

We first recall the definition of locally finitely presentable categories
(\Cref{text:presentability}) and then introduce \emph{essentially algebraic
  theories}, which are a standard tool to show that some categories are locally
finitely presentable (\Cref{text:ess-alg-theories}).

\subsection{Presentability}
\label{text:presentability}

In this section, we define the notion of locally finitely presentable category,
after recalling directed colimits and presentable objects of categories.

\paragraph{Directed colimits}
\parlabel{text:directed-colims}

A partial order $(D,\le)$ is \index{directed}\emph{directed} when $D \neq \emptyset$ and for all
$x,y \in D$, there exists~$z \in D$ such that~${x \le z}$ and~${y \le z}$. A
small category~$I$ is called \emph{directed} when it is isomorphic to a directed
partial order~$(D,\le)$.

Given a category~$\mcal C \in \CAT$, a \index{diagram}\emph{diagram} in~$\mcal C$ is the data
of a functor $d\co I \to \mcal C$ where $I$ is a small category. We say that it
is a \emph{directed diagram} when $I$ is moreover directed. A \emph{directed
  colimit of~$C$} is a colimit cocone $(p_i\co d(i)\to X)_{i \in I}$ on a
directed diagram~$d\co I \to \mcal C$.
\begin{example}
  A set is a directed colimit of its finite subsets. A monoid is a directed
  colimit of its finitely generated submonoids.
\end{example}
\noindent In $\Set$, we have the following characterization of directed
colimits:
\begin{prop}
  \label{prop:directed-colimits-set}
  Let $d\co I \to \Set$ be a directed diagram in~$\Set$ and $(p_i\co d(i)
  \to C)_{i \in I}$ be a cocone on~$d$. Then, $(p_i\co d(i) \to C)_{i \in I}$ is
  a directed colimit on~$d$ if and only if
  \begin{enumerateroman}
  \item for all $x \in C$, there is $i \in I$ and $x' \in d(i)$ such that
    $p_i(x') = x$,
  \item for all $i_1,i_2 \in I$, $x_1 \in d(i_1)$ and $x_2 \in d(i_2)$, if
    $p_{i_1}(x_1) = p_{i_2}(x_2)$, then there exists $i \in I$ such that

    \centering $i_1 \to i \in I$,\quad $i_2 \to i \in I$\quad and%
    \quad$d(i_1 \to i)(x_1) = d(i_2 \to i)(x_2)$.
  \end{enumerateroman}
\end{prop}
\begin{proof}
  See for example~\cite[Proposition~2.13.3]{borceux1994handbook1}.
\end{proof}

\paragraph{Finitely presentable objects}

Let~$\mcal C \in \CAT$. An object $P \in \mcal C$ is \index{finitely presentable!object}\emph{finitely presentable}
when its hom-functor
\[
  \mcal C(P,-)\co \mcal C \to \Set
\]
commutes with directed colimits. By \Cref{prop:directed-colimits-set}, it means
that, given a directed colimit
\[
  (p_i \co d(i) \to X)_{i \in I}
\]
on a directed diagram~$d \co I \to \mcal C$, we have
\begin{enumerateroman}
\item for every~$X \in \mcal C$ and~$f \co P \to X$, there is a
  \emph{factorization of~$f$ through~$d$}, \ie there exists $i \in I$ and $g \co P
  \to d(i)$ such that $f = p_i \circ g$;
\item this factorization is \index{essentially unique factorization}\emph{essentially unique}, \ie if there exist
  others~$i' \in I$ and~$g' \co P \to d(i)$ such that $f = p_{i'} \circ g'$,
  then there exist $j \in I$, $h\co i \to j \in I$ and $h'\co i' \to j \in I$
  such that
  \[
    d(h) \circ g = d(h') \circ g'\zbox.
  \]
\end{enumerateroman}

    

\begin{example}
  \label{ex:finitely-presentable-sets}
  Given a set~$S$, $S$ is finitely presentable if and only if it is finite.
  See~\cite[Example~1.2(1)]{adamek1994locally} for details.
\end{example}
\begin{example}
  \label{ex:finitely-presentable-groups}
  A monoid is \index{finitely presentable!monoid}\emph{finitely presentable} when it admits a presentation
  consisting of a finite number of generators and equations. A similar
  description of finitely presentable objects holds for the other categories of
  algebraic structures (groups, rings, etc.).
  See~\cite[Theorem~3.12]{adamek1994locally} for details.
\end{example}

\paragraph{Locally finitely presentable categories}
A locally small category~$\mcal C \in \CAT$ is \index{locally finitely
  presentable category}\emph{locally finitely
  presentable} when
\begin{enumerateroman}
\item it has all small colimits,
\item every object of~$\mcal C$ is a directed colimit of locally finitely
  presentable objects,
\item the full subcategory of~$\mcal C$ whose objects are the finitely
  presentable objects is essentially small.
\end{enumerateroman}
\begin{example}
  The category $\Set$ is locally finitely presentable. Indeed, it is cocomplete
  and every set is a directed colimit of its finite subsets, which are finitely
  presentable objects of~$\Set$.
\end{example}
\begin{example}
  The category $\Mon$ of monoids is locally finitely presentable. More
  generally, the categories of algebraic structures (groups, rings, \etc) are
  locally finitely presentable. This is the consequence of the fact that such
  categories can be described by means of essentially algebraic theories, as we
  will see in the next section.
\end{example}

\ndr{commenté: caractérisation par des strong generators}

\noindent Identifying a category as locally finitely presentable enables to
derive several elementary properties, like completeness:
\begin{prop}
  \label{prop:loc-pres-nice-properties}
  A locally finitely presentable category is complete.
\end{prop}
\begin{proof}
  See~\cite[Corollary 1.28,Remark 1.56(1),Theorem 1.58]{adamek1994locally} for details.
\end{proof}
\noindent Moreover, showing that a functor between two locally finitely
presentable categories is a left or right adjoint is easier than in the general
case, since we do not need the existence of solution set like in Freyd's adjoint
theorem (\cite[Theorem~3.3.3]{borceux1994handbook1}):
\begin{prop}
  \label{prop:criterion-loc-pres-cat-adjoints}
  Given a functor $F\co \mcal C \to \mcal D$ between two locally finitely presentable
  categories $\mcal C$ and~$\mcal D$, the following hold:
  \begin{enumerateroman}
  \item $F$ is left adjoint if and only if it preserves colimits,
    
  \item if $F$ preserves limits and directed colimits, then it is right adjoint.
  \end{enumerateroman}
\end{prop}
\noindent Finally, there is a simple criterion for a category of algebras on a
monad to be locally finitely presentable. We recall that a functor~$F$ is
\index{finitary!functor}\emph{finitary} when~$F$ preserves directed colimits, and a monad~$(T,\eta,\mu)$
on a category~$\mcal C$ is \index{finitary!monad}\emph{finitary} when~$T$ is finitary. We then have:
\begin{prop}
  \label{prop:loc-fin-pres-alg-cat}
  Given a locally finitely presentable category~$C$ and a finitary
  monad~$(T,\eta,\mu)$ on~$\mcal C$, the category of algebras $\mcal C^T$ is
  locally finitely presentable. Moreover, the canonical forgetful functor~$\mcal
  C^T \to \mcal C$ preserves directed colimits.
\end{prop}
\begin{proof}
  The category~$\mcal C^T$ is finitely locally presentable by~\cite[Theorem 2.78
  and the following remark]{adamek1994locally}. Moreover, since~$T$ is
  finitary, the directed colimits of~$\mcal C^T$ are computed in~$\mcal C$, so
  that the mentioned forgetful functor preserves directed colimits.
\end{proof}

\begin{example}
  The category $\Mon$ is equivalent to the category of algebras~$\Set^T$
  where~$(T,\eta,\mu)$ is the free monoid functor on~$\Set$. It can be shown
  that~$T$ is finitary, so that we obtain another proof that~$\Mon$ is locally
  finitely presentable using \Propr{loc-fin-pres-alg-cat}.
\end{example}

\paragraph[Locally $\omega_1$\presentable categories]{Locally $\bm{\omega_1}$\presentable categories}

The above definitions and properties can be generalized to higher cardinals in
order to define more general notions of presentable categories, in particular
for the first non-countable cardinal $\aleph_1 = \card{\omega_1}$.

A partial order $(D,\le)$ is \emph{$\omega_1$-directed} when every countable
subset $S \subseteq D$ admits an upper bound in $D$. A small category~$I$ is
called \emph{$\omega_1$-directed} when it is isomorphic to an
$\omega_1$-directed partial order~$(D,\le)$.

Given a diagram $d\co I \to \mcal C$ in a category $\cC$, we say that it is an
\emph{$\omega_1$-directed diagram} when $I$ is $\omega_1$-directed. An
\emph{$\omega_1$-directed colimit of~$C$} is a colimit cocone $(p_i\co d(i)\to
X)_{i \in I}$ on a directed diagram~$d\co I \to \mcal C$.

Let~$\mcal C \in \CAT$. An object $P \in \mcal C$ is
\emph{$\omega_1$-presentable} when its hom-functor
\[
  \mcal C(P,-)\co \mcal C \to \Set
\]
commutes with $\omega_1$-directed colimits.

A locally small category~$\mcal C \in \CAT$ is \emph{locally
  $\omega_1$-presentable} when
\begin{enumerateroman}
\item it has all small colimits,
\item every object of~$\mcal C$ is an $\omega_1$-directed colimit of locally
  $\omega_1$-presentable objects,
\item the full subcategory of~$\mcal C$ whose objects are the
  $\omega_1$-presentable objects is essentially small.
\end{enumerateroman}

\begin{prop}
  \label{prop:loc-oone-pres-nice-properties}
  A locally $\omega_1$-presentable category is complete.
\end{prop}
\begin{proof}
  See~\cite[Corollary 1.28]{adamek1994locally}.
\end{proof}

\begin{prop}
  \label{prop:criterion-loc-oone-pres-cat-adjoints}
  Given a functor $F\co \mcal C \to \mcal D$ between two locally
  $\omega_1$-presentable categories $\mcal C$ and~$\mcal D$, the following hold:
  \begin{enumerateroman}
  \item $F$ is left adjoint if and only if it preserves colimits,
    
  \item if $F$ preserves limits and $\omega_1$-directed colimits, then it is
    right adjoint.
  \end{enumerateroman}
\end{prop}
\noindent We stop the copy-and-paste here and refer the reader to
\cite{adamek1994locally} for other properties shared by locally $\omega_1$- and
finitely presentable categories.

\subsection{Essentially algebraic theories}
\label{text:ess-alg-theories}

Verifying that some category is locally finitely presentable with the above
definition can be tedious. A simpler way consists in describing it as the
category of models of some \emph{essentially algebraic theory}. The latter is
similar to an algebraic theory (theory of monoids, theory of groups, \etc),
except that operations with partial domains are allowed, as long as those
domains are specified by equations. Another interesting property is that
morphisms between such theories induce functors between the associated
categories of model, and those functors are moreover right adjoints and preserve
directed colimits. The main reference here
is~\cite[Section~3.D]{adamek1994locally}.

\paragraph{Definition}

Given a set~$S$, an \index{signature}\emph{$S$\sorted signature} is the data of
a set~$\Sigma$ of \index{symbol}\emph{symbols} such that each~$\sigma \in
\Sigma$ has an \index{arity of a symbol}\emph{arity} under the form of a finite
sequence~$(s_i)_{i \in \N^*_n}$ of elements of~$S$ for some~$n \in \N$, and a
\index{target!of a symbol}\emph{target} in the form of an element~$s \in S$ and
we write
\[
  \sigma \co s_1 \times \cdots \times s_n \to s
\]
such a symbol~$\sigma$ of~$\Sigma$ with such arity and target.

Let $(x_i)_{i \in \N}$ be a chosen sequence of distinct variable names. Given a
set~$S$, an \index{context (for a theory)}\emph{$S$\sorted context} is the data of a finite sequence~$\Gamma =
(s_i)_{i \in \N^*_n}$ of elements of~$S$ for some~$n \in \N$. Under the
context~$\Gamma$, the variable~$x_i$ should be thought ``of type $s_i$'' for~$i
\in \N_n$ so that we often write
\[
  x_1\co s_1,\ldots,x_n \co s_n
\]
for such a context~$\Gamma$.

Given a set~$S$ and $S$\sorted signature~$\Sigma$ and context~$\Gamma$, we
define~\index{term on a signature}\emph{$\Sigma$\terms} on~$\Gamma$ together
with judgements~$\Gamma \vdash t \co s$ where~$t$ is a $\Sigma$\term and~$s \in
S$, inductively as follows:
\begin{itemize}
  
\item if~$\Gamma = (s_i)_{i \in \N^*_n}$ for some~$n \in \N$ and~$s_1,\ldots,s_n
  \in S$, then, for every~$i \in \N^*_n$,~$\Gamma \vdash x_i \co s_i$,
\item given~$\sigma\co s_1 \times \cdots \times s_n \to s \in \Sigma$
  and~$\Sigma$\terms~$t_1,\ldots,t_n$ such that~$\Gamma \vdash t_i \co s_i$
  for~$i \in \N^*_n$, then~$\Gamma \vdash \sigma (t_1,\ldots,t_n) \co s$.
\end{itemize}
Note that~$s$ is uniquely determined by~$t$ in a judgement~$\Gamma \vdash t \co
s$.

\smallpar An \index{essentially algebraic!theory}\emph{essentially algebraic theory} is a tuple
\[
  \stdtheory = (S,\Sigma,E,\Sigma_t,\eatdef)
\]
where
\begin{itemize}
\item $S$ is a set,
\item $\Sigma$ is an $S$\sorted signature,
\item $E$ is a set of triples~$(\Gamma,t_1,t_2)$ where~$\Gamma$ is an $S$\sorted
  context, and~$t_1,t_2$ are $\Sigma$\terms on~$\Gamma$ such that there
  exists~$s \in S$ so that~$\Gamma \vdash t_i \co s$ for~$i \in \set{1,2}$,
\item $\Sigma_t$ is a subset of~$\Sigma$,
  
\item $\eatdef$ is a function which maps~$\sigma\co s_1\times \cdots \times s_n
  \to s \in \Sigma\setminus\Sigma_t$ to a set of pairs~$(t_1,t_2)$
  of~$\Sigma_t$\terms such that there exists~$s \in S$ so that~$(x_1\co
  s_1,\ldots,x_n\co s_n) \vdash t_i \co s$ for~$i \in \set{1,2}$.
\end{itemize}
The set~$S$ represents the different \index{sort}\emph{sorts} of the theory, the
set~$\Sigma$ the different operations that appear in the theory, the set~$E$ the
global equations satisfied by the theory, the set~$\Sigma_t$ the operations
whose domains are total, and the function~$\eatdef$ the equations that define
the domains of the partial operations. Given such an essentially algebraic
theory~$\stdtheory$, a \index{model of an essentially algebraic
  theory}\emph{model of~$\stdtheory$}, or~\emph{$\stdtheory$\model}, is the data
of
\begin{itemize}
\item for all~$s \in S$, a set~$M_s$,
  
\item for all~$\sigma \co s_1 \times \cdots \times s_n \to s \in \Sigma_t$, a function
  \[
    M_\sigma \co M_{s_1} \times \cdots \times M_{s_n} \to M_s\zbox,
  \]
  
\item for all~$\sigma \co s_1 \times \cdots \times s_n \to s \in
  \Sigma\setminus\Sigma_t$, a partial function
  \[
    M_\sigma \co M_{s_1} \times \cdots \times M_{s_n} \to M_s\zbox,
  \]\vskip-\belowdisplayskip\vskip-\baselineskip
\end{itemize}
such that
\begin{itemize}
\item for all~$\sigma \co s_1 \times \cdots \times s_n \to s \in
  \Sigma\setminus\Sigma_t$, $M_\sigma$ is defined at $\bar y = (y_1,\ldots,y_n)
  \in M_{s_1} \times \cdots \times M_{s_n}$ if and only if, for all~$(t_1,t_2)
  \in \eatdef(\sigma)$, we have~$\eateval{t_1}_{\bar y} = \eateval{t_2}_{\bar
    y}$,
\item for every triple~$(\Gamma,t_1,t_2) \in E$ where~$\Gamma = (s_i)_{i \in
    \N^*_n}$ for some~$n \in \N$ and sorts~$s_1,\ldots,s_n \in S$, given a
  tuple~$\bar y = (y_1,\ldots,y_n) \in M_{s_1} \times \cdots \times M_{s_n}$, if
  both~$\eateval {t_1}_{\bar y}$ and~$\eateval {t_2}_{\bar y}$ are defined,
  then~$\eateval {t_1}_{\bar y} = \eateval {t_2}_{\bar y}$,
\end{itemize}
where, given an $S$\sorted context~$\Gamma = (s_i)_{i \in \N^*_n}$, a sort~$s \in
S$, a $\Sigma$\term~$t$ such that~$\Gamma \vdash t \co s$, and a tuple~$\bar y =
(y_1,\ldots,y_n) \in M_{s_1}\times \cdots \times M_{s_n}$, the \emph{evaluation
  of~$t$ at~$\bar y$}, denoted~$\eateval t_{\bar y}$, is either undefined or an
element of~$M_s$, and is defined by induction on~$t$ by
\begin{itemize}
\item if~$t = x_i$ for some~$i \in \N^*_n$, then $\eateval t_{\bar y}$ is
  defined and
  \[
    \eateval t_{\bar y} = y_i\zbox,
  \]
  
\item if~$t = \sigma (t_1,\ldots,t_k)$ for some~$k \in \N^*$
  and~$\Sigma_t$\terms~$t_1,\ldots,t_k$, then~$\eateval t_{\bar y}$ is defined
  if and only if $\eateval {t_1}_{\bar y},\ldots,\eateval {t_k}_{\bar y}$ are
  defined and $M_\sigma$ is defined at~$\eateval {t_1}_{\bar y},\ldots,\eateval
  {t_k}_{\bar y}$ and, in this case,
  \[
    \eateval t_{\bar y} =
    M_\sigma(\eateval {t_1}_{\bar y},\ldots,\eateval {t_k}_{\bar y})\zbox.
  \]
\end{itemize}
Given two models~$M$ and~$M'$ of~$\stdtheory$, a \index{morphism!of models}\emph{morphim
  of~$\stdtheory$\model} between~$M$ and~$M'$ is a family of functions~$f = (f_s
\co M_s \to M'_s)_{s \in S}$ such that
\begin{itemize}
\item for all~$\sigma \co s_1 \times \cdots \times s_n \to s \in \Sigma_t$,
  $f_s \circ M_\sigma = M'_\sigma \circ (f_{s_1}\times \cdots \times f_{s_n})$,
  
\item for all~$\sigma \co s_1 \times \cdots \times s_n \to s \in
  \Sigma\setminus\Sigma_t$ and~$\bar y = (y_1,\ldots,y_n) \in M_{s_1} \times
  \cdots \times M_{s_n}$ such that~$M_\sigma$ is defined on~$\bar y$, $f_s \circ
  M_s (\bar y) = M'_s (f_{s_1}(y_1),\ldots,f_{s_n}(y_n))$.
\end{itemize}
We then write~$\Mod(\stdtheory)$ for the category of~$\stdtheory$\model{}s and
their morphisms. We say that a (big) category~$\mcal C \in \CAT$ is
\index{essentially algebraic!category}\emph{essentially algebraic} when it is equivalent to the category of models of
some essentially algebraic theory.

Identifying a category as essentially algebraic enables to deduce that it is
locally finitely presentable, since the two notions are the same:

\begin{theo}
  \label{prop:ess-alg-iff-loc-fin-pres}
  Given a category~$\mcal C \in \CAT$,~$\mcal C$ is essentially algebraic if and
  only if it is locally finitely presentable.
\end{theo}
\begin{proof}
  See the proof of~\cite[Theorem~3.36]{adamek1994locally}.
\end{proof}

\begin{example}
  The category~$\Set$ is essentially algebraic since it is the category of
  models of the essentially algebraic
  theory~$(\set s,\emptyset,\emptyset,\emptyset,\bot)$.
\end{example}
\begin{example}
  \label{ex:mon-ess-alg-theory}
  The category~$\Mon$ is essentially algebraic since it is the category of
  models of the essentially algebraic theory
  \[
    \stdtheory^{\mathrm{mon}} = (\set{s},\set{e \co \termobj \to {s},m \co {{s}} \times {{s}} \to {s}},E,\set{e,m},\bot)
  \]
  where~$E$ consists of three equations
  \begin{itemize}
  \item $m (e,x_1) = x_1$ in the context~$(x_1 \co {s})$,
  \item $m (x_1,e) = x_1$ in the context~$(x_1 \co {s})$,
  \item $m (m(x_1,x_2),x_3) = m(x_1,m(x_2,x_3))$ in the context~$(x_1 \co
    {{s}},x_2 \co {{s}},x_3 \co {{s}})$.
  \end{itemize}
  In particular, it gives a simple proof that~$\Mon$ is locally finitely presentable.
\end{example}
\begin{example}
  \label{ex:cat-ess-alg-theory}
  The category~$\Cat$ of small categories is essentially algebraic since it is
  the category of models of the essentially algebraic
  theory~$\stdtheory^{\mathrm{cat}} = (S,\Sigma,E,\Sigma_t,\eatdef)$ defined as
  follows. The set~$S$ consists of two sorts~$c_0$ and~$c_1$ corresponding to
  $0$\cells and~$1$\cells, and
  \[
    \Sigma = \set{\csrc_0 \co c_1 \to c_0,\quad \ctgt_0 \co c_1 \to c_0,\quad\unitp 1{} \co c_0 \to c_1,\quad {\comp}\co c_1 \times c_1
      \to c_1}\zbox.
  \]
  Moreover,~$E$ consists of the equations
  \begin{itemize}
  \item $\csrc_0 (\unitp 1 {} (x_1)) = x_1$ and~$\ctgt_0(\unitp 1 {} (x_1)) =
    x_1$ in the context~$(x_1 \co c_0)$,
    
  \item $\csrc_0(\comp(x_1,x_2)) = \csrc_0(x_1)$ and~$\ctgt_0(\comp(x_1,x_2)) =
    \ctgt_0(x_2)$ in the context~$(x_1 \co c_1,x_2 \co c_1)$,
    
  \item $\comp(\unitp 1 {}(\csrc_0(x_1)),x_1) = x_1$ and~$\comp(x_1,\unitp 1
    {}(\ctgt_0(x_1))) = x_1$ in the context~$(x_1 \co c_1)$,
    
  \item $\comp(\comp(x_1,x_2),x_3) = \comp(x_1,\comp(x_2,x_3))$ in the
    context~$(x_1 \co c_1,x_2 \co c_1,x_3 \co c_1)$.
  \end{itemize}
  Finally,~$\Sigma_t = \set{\csrc_0,\ctgt_0,\unitp 1 {}}$, and~$\eatdef(\ast)$
  is the singleton set containing the equation~$\ctgt_0(x_1) = \csrc_0(x_2)$.
  This shows that~$\Cat$ is a locally finitely presentable category.
\end{example}

\paragraph{Morphisms of theories}

Given two essentially algebraic theories
\[
  \stdtheory = (S,\Sigma,E,\Sigma_t,\eatdef)
  \qtand
  \stdtheory' = (S',\Sigma',E',\Sigma'_t,\eatdef')
\]
a \index{morphism!of essentially algebraic theories}\emph{morphism of essential
  algebraic theories} between~$\stdtheory$ and~$\stdtheory'$ is the data of
\begin{itemize}
\item a function~$f \co S \to S'$,
  
\item a function~$g \co \Sigma \to \Sigma'$, 
\end{itemize}
such that
\begin{itemize}
\item given~$\sigma \co s_1 \times \cdots \times s_n \to s \in \Sigma$, we
  have~$g(\sigma) \co f(s_1) \times \cdots \times f(s_n) \to f(s) \in \Sigma'$,
\item given $\sigma \in \Sigma$, $\sigma \in \Sigma_t$ if and only if~$g(\sigma)
  \in \Sigma'_t$,
\item given $(\Gamma,t_1,t_2) \in E$, we have~$(f(\Gamma),g(t_1),g(t_2)) \in E'$,
  
\item given~$\sigma \in \Sigma \setminus \Sigma_t$ and
  two~$\Sigma_t$\terms~$t_1$ and~$t_2$, we have that~$(t_1,t_2) \in
  \eatdef(\sigma)$ if and only if~$(g(t_1),g(t_2)) \in \eatdef'(g(\sigma))$,
\end{itemize}
where, given~$\Gamma = (s_i)_{i \in \N^*_n}$, we write~$f(\Gamma)$
for~$(f(s_i))_{i \in \N^*_n}$ and, given a $\Sigma$\term~$t$, we write~$g(t)$
for the $\Sigma'$\term defined by induction on~$t$ by
\begin{itemize}
\item for all variable $x_i$,
  \[
    g(x_i) = x_i\zbox,
  \]
  
\item for all~$\sigma \co s_1 \times \cdots \times s_n \to s \in \Sigma$ and
  $\Sigma$\terms~$t_1,\ldots,t_n$,
  \[
    g(\sigma (t_1,\ldots,t_n)) = g(\sigma)(g(t_1),\ldots,g(t_n))\zbox.
  \]
\end{itemize}
Such a morphism~$(f,g)\co \stdtheory \to \stdtheory'$ induces a functor
\[
  \Mod((f,g)) \co \Mod(\stdtheory') \to \Mod(\stdtheory)
\]
which maps a model~$M' \in \Mod(\stdtheory')$ to a model~$M \in
\Mod(\stdtheory)$ defined by
\begin{itemize}
\item for all $s \in S$, $M_s = M'_{f(s)}$,
  
\item for all $\sigma \in \Sigma$, $M_{\sigma} = M'_{\smash{g(\sigma)}}$,
\end{itemize}
and which maps morphisms of models as expected. The functors induced this way by
morphisms between theories have good properties:

\begin{theo}
  \label{prop:functor-from-theo-morphism-ra}
  Given a morphism~$(f,g) \co \stdtheory \to \stdtheory'$ between two
  essentially algebraic theories~$\stdtheory$ and~$\stdtheory'$, the
  functor~$\Mod((f,g))$ is a right adjoint which preserves directed colimits.
\end{theo}
\begin{proof}
  The fact that it is a right adjoint is given
  by~\cite[Theorem~5.4]{palmgren2007partial}. Moreover, one easily verifies that
  the directed colimits are computed pointwise in both~$\Mod(\stdtheory)$
  and~$\Mod(\stdtheory')$, so that they are preserved by~$\Mod((f,g))$.
\end{proof}
\begin{remark}
  A more general definition of morphisms between essentially algebraic theories
  for which \Cref{prop:functor-from-theo-morphism-ra} holds can be defined.
  However, it would require the introduction of formal deduction systems, which
  would be quite long and technical. This would be in vain since our definition
  of morphisms is enough for our purposes.
\end{remark}
\begin{example}
  One can define the essentially algebraic theory~$\stdtheory^{\mathrm{grp}}$ of
  groups from the one of monoids given in \Cref{ex:mon-ess-alg-theory} by adding
  a symbol~$i \co s \to s$ representing a total function, and by adding the
  equations~$m(i(x_1),x_1) = e$ and~$m(x_1,i(x_1)) = e$ in the context~$(x_1 \co
  s)$. The canonical embedding~$\stdtheory^{\mathrm{mon}} \to
  \stdtheory^{\mathrm{grp}}$ induces a functor~$\Grp \to \Mon$ between the
  categories of groups and monoids which is the expected forgetful functor. This
  functor is a right adjoint and preserves directed colimits by
  \Cref{prop:functor-from-theo-morphism-ra}.
\end{example}
\begin{example}
  \label{ex:grph-to-cat-theo-morphism}
  The essentially algebraic theory
  \[
    \stdtheory^{\mathrm{gph}} = (\set{c_0,c_1},\set{\gsrc_0 \co c_1 \to
      c_0,\gtgt_0 \co c_1 \to c_0},\emptyset,\set{\gsrc_0,\gtgt_0},\bot)
  \]
  exhibits the category~$\Gph$ of graphs as an essentially algebraic category.
  Recalling from \Cref{ex:cat-ess-alg-theory} the definition
  of~$\stdtheory^{\mathrm{cat}}$, the mappings~$\gsrc_0 \mapsto \csrc_0$
  and~$\gtgt_0 \mapsto \ctgt_0$ define a morphism of essentially algebraic
  theories~$\stdtheory^{\mathrm{gph}} \to \stdtheory^{\mathrm{cat}}$, which
  induces a functor~$\Cat \to \Gph$ that is the expected forgetful functor. This
  functor is a right adjoint and preserves directed colimits by
  \Cref{prop:functor-from-theo-morphism-ra}.
\end{example}




\clearpage
\printbibliography

\end{document}